\documentclass[a4paper, pdftex]{amsart}

\usepackage[hyphens,spaces,obeyspaces]{url}
\usepackage[utf8]{inputenc}
\usepackage[T1]{fontenc}
\usepackage[pdftex]{graphicx}
\usepackage{amsmath}
\usepackage{amsfonts}
\usepackage{amssymb}
\usepackage{amsthm}
\usepackage{relsize}
\usepackage{stmaryrd}
\SetSymbolFont{stmry}{bold}{U}{stmry}{m}{n} 
\usepackage{setspace}
\usepackage{caption}
\usepackage{thmtools}
\usepackage{thm-restate}
\usepackage{mathtools}
\usepackage{mathabx}
\usepackage{bm}
\usepackage{bbm}
\usepackage{colonequals}
\usepackage[colorlinks, linkcolor=blue, citecolor=blue]{hyperref}
\usepackage[capitalize]{cleveref}
\usepackage{verbatim}
\usepackage[all, cmtip]{xy}
\usepackage{multirow}
\usepackage{enumitem}
\usepackage{permute}
\usepackage[margin=1in]{geometry}

\usepackage[textsize = tiny]{todonotes}
\usepackage{aliascnt}
\usepackage{tikz}
\usepackage{tikz-cd}
\usepackage[format=plain,justification=centering]{caption}

\begin{document}
\newcounter{theoremcounter}
\numberwithin{theoremcounter}{section}
\newaliascnt{theoremauto}{theoremcounter}
\newcommand{\theoremautoautorefname}{Theorem}
\newcommand{\atheoremautorefname}{Theorem}
\newaliascnt{Defauto}{theoremcounter}
\newcommand{\Defautoautorefname}{Definition}
\newaliascnt{exampleauto}{theoremcounter}
\newcommand{\exampleautoautorefname}{Example}
\newaliascnt{lemmaauto}{theoremcounter}
\newcommand{\lemmaautoautorefname}{Lemma}
\newaliascnt{propositionauto}{theoremcounter}
\newcommand{\propositionautoautorefname}{Proposition}
\newaliascnt{corollaryauto}{theoremcounter}
\newcommand{\corollaryautoautorefname}{Corollary}
\newaliascnt{remarkauto}{theoremcounter}
\newcommand{\remarkautoautorefname}{Remark}
\newaliascnt{notationauto}{theoremcounter}
\newcommand{\notationautoautorefname}{Notation}
\newaliascnt{claimauto}{theoremcounter}
\newcommand{\claimautoautorefname}{Claim}
\newaliascnt{warningauto}{theoremcounter}
\newcommand{\warningautoautorefname}{Warning}
\newaliascnt{questionauto}{theoremcounter}
\newcommand{\questionautoautorefname}{Question}
\newaliascnt{discussionauto}{theoremcounter}
\newcommand{\discussionautoautorefname}{Discussion}
\newaliascnt{computationauto}{theoremcounter}
\newcommand{\computationautoautorefname}{Computation}
\newaliascnt{conjectureauto}{theoremcounter}
\newcommand{\conjectureautoautorefname}{Conjecture}
\newaliascnt{convauto}{theoremcounter}
\newcommand{\convautoautorefname}{Convention}
\renewcommand{\sectionautorefname}{Section}
\renewcommand{\subsectionautorefname}{Section}
\renewcommand{\itemautorefname}{Condition}

\newcounter{atheoremcounter}

\newtheorem{theorem}[theoremauto]{Theorem}
\newtheorem{lemma}[lemmaauto]{Lemma}
\newtheorem{proposition}[propositionauto]{Proposition}
\newtheorem{corollary}[corollaryauto]{Corollary}
\newtheorem*{corollary*}{Corollary}
\newtheorem{conjecture}[conjectureauto]{Conjecture}
\newtheorem{atheorem}{Theorem}
\renewcommand{\theatheorem}{\Alph{atheorem}}
\newtheorem{acorollary}[atheorem]{Corollary}
\renewcommand{\theacorollary}{\Alph{acorollary}}
\theoremstyle{definition}
\newtheorem{definition}[Defauto]{Definition}
\newtheorem{notation}[notationauto]{Notation}
\newtheorem{convention}[convauto]{Convention}
\newtheorem{example}[exampleauto]{Example}
\newtheorem{remark}[remarkauto]{Remark}
\newtheorem{claim}[claimauto]{Claim}
\newtheorem{warning}[warningauto]{Warning}
\newtheorem{question}[questionauto]{Question}
\newtheorem{discussion}[discussionauto]{Discussion}
\newtheorem{computation}[computationauto]{Computation}
%Custom commands:

	%Kuerzel von latex-Befehlen:
	\renewcommand{\ll}{\left\langle}
	\newcommand{\rr}{\right\rangle}
	\newcommand{\ls}{\left\{}
	\newcommand{\rs}{\right\}}
	\newcommand{\sm}{\setminus}

\newcommand{\cone}{\ensuremath{{cone}}}
    
	\newcommand{\mcA}{\ensuremath{\mathcal{A}}}
	\newcommand{\mcB}{\ensuremath{\mathcal{B}}}
	\newcommand{\mcC}{\ensuremath{\mathcal{C}}}
	\newcommand{\mcE}{\ensuremath{\mathcal{E}}}
	\newcommand{\mcF}{\ensuremath{\mathcal{F}}}
	\newcommand{\mcG}{\ensuremath{\mathcal{G}}}
	\newcommand{\mcH}{\ensuremath{\mathcal{H}}}
	\newcommand{\mcI}{\ensuremath{\mathcal{I}}}
	\newcommand{\mcJ}{\ensuremath{\mathcal{J}}}
	\newcommand{\mcN}{\ensuremath{\mathcal{N}}}
	\newcommand{\mcO}{\ensuremath{\mathcal{O}}}
	\newcommand{\mcP}{\ensuremath{\mathcal{P}}}
	\newcommand{\mcQ}{\ensuremath{\mathcal{Q}}}
 	\newcommand{\mcR}{\ensuremath{\mathcal{R}}}
	\newcommand{\mcS}{\ensuremath{\mathcal{S}}}
	\newcommand{\mcU}{\ensuremath{\mathcal{U}}}
	\newcommand{\mcV}{\ensuremath{\mathcal{V}}}
	\newcommand{\mcW}{\ensuremath{\mathcal{W}}}
	\newcommand{\mcZ}{\ensuremath{\mathcal{Z}}}
	\newcommand{\mbQ}{\ensuremath{\mathbb{Q}}}
	\newcommand{\mbZ}{\ensuremath{\mathbb{Z}}}
    \newcommand{\mbN}{\ensuremath{\mathbb{N}}}

  \newcommand{\bn}{\ensuremath{{\bf n}}}
 \renewcommand{\bm}{{\bf m}}
 %\renewcommand{\bM}{{\bf M}}

	%Helpful stuff:
	\newcommand{\on}[1]{\operatorname{#1}}
	\newcommand{\set}[1]{\ensuremath{ \left\lbrace #1 \right\rbrace}}	
	\newcommand{\grep}[2]{\ensuremath{\left\langle #1 \, \middle| \, #2\right\rangle}}
	\newcommand{\real}[1]{\ensuremath{\left\lVert #1\right\rVert}}
	\newcommand{\overbar}[1]{\mkern 2mu\overline{\mkern-2mu#1\mkern-2mu}\mkern 2mu}
	
	%Operatornames:
	\newcommand{\Aut}{\ensuremath{\operatorname{Aut}}}
	\newcommand{\AutO}{\ensuremath{\operatorname{Aut}^0}}
	\newcommand{\Out}{\ensuremath{\operatorname{Out}}}
	\newcommand{\Outo}[1][A_G]{\ensuremath{\operatorname{Out}^0(#1)}}
	\newcommand{\Stab}{\operatorname{Stab}}
	\newcommand{\Sym}{\operatorname{Sym}}
	\newcommand{\St}{\operatorname{St}}
	\newcommand{\Stsymp}{\operatorname{St}^\omega}
	
	\newcommand{\op}{\operatorname{op}}
	\newcommand{\st}{\operatorname{star}}
	\newcommand{\lk}{\operatorname{lk}}
	\newcommand{\im}{\operatorname{im}}
	\newcommand{\rk}{\operatorname{rk}}
	\newcommand{\corank}{\operatorname{crk}}
	\newcommand{\spn}{\operatorname{span}}
	\newcommand{\Ind}{\operatorname{Ind}}

    \newcommand{\fb}{\operatorname{FB}}
    \newcommand{\hfb}{\operatorname{hFB}}
    \newcommand{\fin}{\operatorname{FI}}
    \newcommand{\FI}{\operatorname{FI}}

    \newcommand{\FIc}{\mathrm{FI}(c)}
    
       \newcommand{\C}{\mathbb{C}}
        \newcommand{\Q}{\mathbb{Q}}
         \newcommand{\Z}{\mathbb{Z}}
    \newcommand{\ch}{\operatorname{Ch}}
    \newcommand{\topo}{\operatorname{Top}}
    \newcommand{\sets}{\operatorname{Set}}
    \newcommand{\sMod}{\operatorname{sMod}}
    \newcommand{\Mod}{\operatorname{Mod}}
    \newcommand{\lmod}{\operatorname{mod}}
    \newcommand{\Alg}{\operatorname{Alg}}
    \newcommand{\vs}{\operatorname{Vec}}
	
	\newcommand{\GL}[2]{\ensuremath{\operatorname{GL}_{#1}(#2)}}
	\newcommand{\SL}[2]{\ensuremath{\operatorname{SL}_{#1}(#2)}}
	\newcommand{\Sp}[2]{\ensuremath{\operatorname{Sp}_{#1}(#2)}}
	%More specific shortcuts:
    \newcommand{\N}{\mathbb{N}}
    \newcommand{\R}{\mathbb{R}}
	
	\newcommand{\class}{\operatorname{cl}}
	\newcommand{\Hom}{\operatorname{Hom}}
    \newcommand{\Emb}{\operatorname{Emb}}
	\newcommand{\buildingssymp}{\operatorname{T}^{\omega}}
    \newcommand{\buildingssln}{\operatorname{T}}
    \newcommand{\building}{\Delta}
	\newcommand{\Rn}{R^{2n}}

    \newcommand{\hcl}{\operatorname{cl(R)}}
    \newcommand{\dual}[1]{#1^{'}}
    \newcommand{\signp}[1][n]{S^B_{#1}} %{\bar{\Sigma}_{#1}}
    \newcommand{\sat}[1]{\operatorname{Sat}(#1)}
    \newcommand{\len}{\operatorname{len}}
    \newcommand{\classconst}{\kappa}
    %\newcommand{\apartment}{\Sigma}
    %\newcommand{\symf}[2]{\ensuremath{\operatorname{\omega}(#1,#2)}}    
    %Some Chevalley notation for outlook section
    %\newcommand{\chevalley}{\ensuremath{\mathcal{G}}}
    %\newcommand{\productSL}{\ensuremath{\mathcal{S}}}
    \newcommand{\ohur}[1][]{\mathrm{OHur}^{c}_{G#1}}
    \newcommand{\hur}[1][]{\mathrm{Hur}^{c}_{G#1}}
    \newcommand{\chain}[1]{C_{*}(#1 ; k)}
    \newcommand{\oconf}[1][n]{\mathrm{OConf}_{#1}}
    \newcommand{\conf}[1][n]{\mathrm{Conf}_{#1}}

    \newcommand{\bunit}{\mathbbm{1}}
    \newcommand{\bK}{\mathbb{K}}
    \newcommand{\bM}{\mathbf{M}}
    \newcommand{\bN}{\mathbf{N}}
   \newcommand{\bD}{\mathbf{D}}
        \newcommand{\bQ}{\mathbf{Q}}
    
    \newcommand{\bR}{\mathbf{R}}

    \newcommand{\delc}{\widetilde{\Delta}_{\mathrm{inj,c}}}
    \newcommand\hocolim{\operatorname*{hocolim}}
    \newcommand\holim{\operatorname*{holim}}
    \newcommand\hocofib{\operatorname*{hocofib}}
    \newcommand\colim{\operatorname*{colim}}
    \newcommand\rep{\operatorname*{Rep}}
    \newcommand\relrep{\operatorname*{RelRep}}
    \newcommand{\csig}[1][p]{(c\times\Sigma)_{#1}}
\newcommand{\PBr}{\mathrm{PBr}}
\newcommand{\Br}{\mathrm{Br}}
    
    \newcommand{\Res}{\mathrm{Res}}

\title{Representation stability for ordered Hurwitz spaces}

 \author{Zachary Himes}
 \email{himesz@umich.edu}  
\address{University of Michigan \\
 	Department of Mathematics, 530 Church St, Ann Arbor MI, 48109 \\USA}  

 \author{Jeremy Miller}\thanks{Jeremy Miller was supported by NSF grants DMS-2504473 and DMS-2202943 and a Simons Foundation Travel Support for Mathematicians grant.}
 \email{jeremykmiller@purdue.edu}  
\address{Purdue University \\
 	 West Lafayette IN, 47907 \\USA}  

     \author{Jennifer C. H. Wilson}
\email{jchw@umich.edu}
\address{University of Michigan \\ Department of Mathematics \\
 	 530 Church St\\
 	 Ann Arbor MI, 48109 \\USA}
\thanks{Jennifer Wilson was supported in part by NSF CAREER grant DMS-2142709}
	
\begin{abstract} 
In this paper, we study the topology of ordered Hurwitz space. These are moduli spaces of branched covers with a choice of ordering on the branched points. Answering a question of Ellenberg, we prove that the homology of ordered Hurwitz spaces exhibit representation stability. 
\end{abstract}

\tolerance=1000
    
\maketitle

\vspace{-.25in}
\tableofcontents
\vspace{-.25in}
\section{Introduction}

Hurwitz spaces are moduli spaces of branched covers of curves. Homological stability properties of Hurwitz spaces can often be used to answer questions concerning arithmetic statistics in the context of function fields. For example, they have been used to study Cohen--Lenstra heuristics for class groups \cite{MR3488737, landesman2025cohenlenstramomentsfunctionfields, landesman2025homologicalstabilityhurwitzspaces}, Malle's conjecture on Galois groups \cite{ellenberg2023foxneuwirthfukscellsquantumshuffle}, and the minimalist conjecture for Selmer ranks of families of abelian varieties \cite{ellenberg2025homologicalstabilitygeneralizedhurwitz}.

In this paper, we study a cover of Hurwitz space called \emph{ordered Hurwitz space}. These spaces parameterize branched covers with a choice of order on the set of branched points. Motivated by questions concerning interactions between factorization statistics and the Cohen-Lenstra heuristics, Ellenberg \cite[Problem 5]{MWProblemSession} asked if ordered Hurwitz spaces exhibt representation stability in the sense of Church--Farb \cite{CF}. The goal of this paper is to verify Ellenberg's conjecture.

The homotopy type of the Hurwitz spaces we consider have the following description. Let $\Br_n$ denote the braid group with $n$ strands, $\PBr_n$ denote the subgroup of pure braids, and let $\alpha_{j} \in \Br_n$ be the braid that passes the $j$th strand over the $(j+1)$st. Let $G$ be a group and $c \subseteq G$ be a conjugacy class. Let $\Br_n$ act on $c^n$ via the formula:  $$
    %\label{eq: brd grp action}
    \alpha_j \cdot (c_1,\ldots, c_j, c_{j+1},\ldots, c_n)=(c_1,\ldots, c_{j-1}, c_{j}c_{j+1} c_{j}^{-1}, c_{j}, c_{j+2},\ldots, c_n).
$$
Define $$\hur[,n] \colonequals c^n /\!\!/ \Br_n$$ and $$\ohur[,{\bn}]\colonequals c^n /\!\!/ \PBr_n.$$ Here $/\!\!/$ denotes the so-called Borel construction or homotopy quotient. Note that $\hur[,n]$ has the homotopy type of the moduli space of branched covers of $\C$ where the monodromy around each branched point is in $c$ and with a choice of trivialization of the branched cover in a fixed boundary direction \cite{MR3488737,EVW2}. A point in the space $\ohur[,n]$ additionally has the data of an ordering on the branch points.

Ellenberg--Venkatesh--Westerland \cite[Theorem 6.1]{MR3488737} proved a homological stability result for $H_i(\hur[,n];\Q)$ under the the assumption that the group $G$ and the conjugacy class $c$ satisfies a hypothesis called the \emph{non-splitting property}
(see \cref{defNonSplit}).  Combined with subsequent work of Davis--Schlank \cite{davis2023hilbertpolynomialquandlescolorings}, Ellenberg--Venkatesh--Westerland's stability result implies that the dimensions of the vector spaces $H_i(\hur[,n] ;\Q)$ do not depend on $n$ for $n$ sufficiently large compared with $i$. See \cite{MR4224644,Bloop,BMhur, landesman2024alternatecomputationstablehomology} for other results on the topology of Hurwitz spaces.
%See also Randal-Williams \cite{MR4224644} for an alternative perspective on the proof of homological stability for Hurwitz spaces.

We instead investigate the behavior of the groups $H_i(\ohur[,{\bn}] ;\Q)$ as $n$ tends to infinity. The space $\ohur[,{\bn}]$ has a natural action of the symmetric group $\Sigma_n$ by permuting the order of the branched points.  We study stability properties of the groups $H_i(\ohur[,{\bn}] ;\Q)$ as $\Sigma_n$-representations instead of simply as vector spaces.

The symmetric groups $\Sigma_n$ come with families of representations $V(\lambda)_n$ for each partition $\lambda = (\lambda_1, \dots, \lambda_{\ell})$ of a non-negative integer. These are zero for $n < |\lambda| + \lambda_1$ and are irreducible $\Sigma_n$-representations otherwise (see \cref{DefVLn}). Following Church--Farb \cite[Definitions 2.3 and 2.7]{CF}, we say that a sequence $\{M_n\}_{n\in \N}$ of $\Q[\Sigma_n]$-modules has \emph{uniform multiplicity stability}  with stable range $N_0$ (\cref{DefnUniformMultStability}) if the multiplicity of $V(\lambda)_n$ in $M_n$ does not depend on $n$ for all $n \geq N_0$ and all partitions $\lambda$. The ``uniformity'' here is the condition that the stable range does not depend on the partition $\lambda$. An advantage of this uniformity is that it implies that, in the stable range,  the representation $M_n$ determines the representation $M_{n+1}$.

Our main theorem is the following.

\begin{theorem}\label{thm: mult stability for ohur}
    Let $G$ be a finite group and $c$ a conjugacy class satisfying the non-splitting property. There exist constants $\alpha$ and $\beta$ depending only on $G$ and $c$ such that $\{H_{i}(\ohur[,{\bn}]; \Q) \}_n,$ viewed as a sequence of $\Sigma_n$-representations,  has uniform multiplicity stability with stable range $\alpha i+\beta$.
\end{theorem}

We note that this result is closely related to work of Ellenberg--Shusterman \cite{ElShusterman} and Landesman--Levy \cite{landlevyApp}. Uniform multiplity stability implies polynomial growth of dimension.

\begin{corollary} \label{HurPolyGrowth}
    Let $G$ be a finite group and $c$ a conjugacy class satisfying the non-splitting property. There exist constants $\alpha$ and $\beta$ depending only on $G$ and $c$ and integer-valued polynomials $p_{i,G,c}$ of degree $\leq \alpha i+\beta$ such that $$\dim_{\Q} H_{i}(\ohur[,{\bn}]; \Q) =p_{i,G,c}(n)  \text{ for all }n\geq \alpha i+\beta.$$

\end{corollary}

%Thus, for all $i$, there is an integer-valued polynomial $p_{i,G,c}$ of degree $\leq \alpha i+\beta$ such that $$\dim H_{i}(\ohur[,n]; \Q) =p_{i,G,c}(n)  \text{ for all }n\geq \alpha i+\beta;$$ see \cref{PropPoly}. 

Since the family of irreducible trivial representations is the sequence $V(\lambda)_n$ for $\lambda=\varnothing$, uniform multiplicity stability also implies stability for coinvariants. In this way \cref{thm: mult stability for ohur} recovers Ellenberg--Venkatesh--Westerland's homological stability result for $\hur[,n]$. 

A categorical approach to the theory of representation stability was introduced by Church--Ellenberg--Farb \cite{CEF} using so-called $\FI$-\emph{modules}. The sequence of representations $\{H_i(\ohur[,{\bn}])\}_n$ do not appear to naturally form an $\FI$-module. We describe an analogue of the category $\FI$ adapted to Hurwitz spaces which we call $\FIc$; see \cref{sec: uni rep stab} for relevant definitions. Our stability results give the following.

\begin{corollary}
  \label{OhurFIcIntro}  Suppose that $(G, c)$ has the non-splitting property. There exists constants $a,b$ such that $H_i(\ohur;\Q)$ is generated as an $\FIc$-module degrees $ \leq ai+b$.
\end{corollary}

\vspace{.2in}

{\bf \noindent Proof overview.}
We adapt the tools used in Miller--Patzt--Petersen--Randal-Williams \cite{miller2024uniformtwistedhomologicalstability} to prove uniform twisted homological stability to apply to ordered Hurwitz spaces. A key input in this theory is vanishing for certain derived indecomposables. In this case, we need to consider the derived indecomposables of $\ohur[]$ acting on $\pi_0\left(\ohur[]\right)$. Using ideas from Randal-Williams \cite{MR4767884}, we show that these derived indecomposables are equivalent to a combinatorial chain complex that is
an ordered variant of the ``Koszul-like complex'' of Ellenberg--Venkatesh--Westerland \cite{MR3488737}. We use results and techniques of Ellenberg--Venkatesh--Westerland \cite{MR3488737}  and Shusterman \cite{MR4666043} to show the homology of this ordered Koszul complex vanishes in a range.

\vspace{.2in}

{\bf \noindent Acknowledgments.}
We thank Andrea Bianchi, Jordan Ellenberg, Anh Trong Nam Hoang, Aaron Landesman, Ishan Levy, Peter Patzt, Oscar Randal-Williams, Mark Shusterman, and Craig Westerland for helpful conversations. Additionally, we thank Peter Patzt and Benson Farb for organizing the conference \emph{Midwest Representation Stability Research Meeting 2019} \cite{MWProblemSession} where we learned of this problem. Part of this work was done at the American Institute for Mathematics. The authors are grateful to AIM for their support.

\section{Notation and categorical setup}

\label{SubsectionCategorySetup}

In this section, we describe some categorical preliminaries that will let us formulate our stability results. As is common in the modern homological and representation stability literature, we work in categories of representations of groupoids (see e.g. \cite{ SSIntroTCA, RWW,MR4125676,e2cellsI,MPPpoly}). 
%Our categorical setup here and in \cref{sec: uni rep stab} follows that of Miller--Patzt--Petersen--Randal-Williams  \cite[Section 2]{miller2024uniformtwistedhomologicalstability}.
\begin{convention}
We fix a commutative ring $\bK$. When we take homology, we will use $\bK$-coefficients. We will often suppress $\bK$ from the notation.

\end{convention}

\begin{notation}
    We will use the notation $\N$ to denote the category of non-negative integers $n\in \N$ and identity maps, viewed as a symmetric monoidal category by addition of integers.
    %\footnote{rewrite this and next definition}
\end{notation}
%\begin{notation} 
%    Let $[n]$ denote the set $\{0,\ldots,n\}$. \footnote{I'm not sure if this is the right choice of notation for $[n]$ or if $[n]=\{1,\ldots,n\}$ makes more sense. For working with $\fb$-modules, I think it makes sense for $[n]$ to denote $\{1,\ldots,n\}$, since people often write $\oconf[n]$ for $\text{OConf}_{[n]}$. On the other hand, for simplicial stuff, I think the notation $[n]=\{0,\ldots,n\}$ is more common. Oscar goes with the convention that $[n]=\{0,\ldots,n\}$. Since some of the notation and objects are copied from Oscar's paper, I think it makes more sense to go with $[n]=\{0,\ldots,n\}$.}
%\end{notation}
\begin{definition}
    
     Let $\fb$ denote the category of finite sets and bijections, viewed as a symmetric monoidal category, with monoidal product given by the disjoint union of sets. Given $S\in \fb$, let $\Sigma_{S}$ denote $\Hom_{\fb}(S, S)$, the permutation group of $S$. We let $\bn$ denote the set $\{1, 2, \ldots, n\}$ and ${\bf 0}$ the empty set $\varnothing$. 
     When $S=\bn$, we will  write $\Sigma_{n}$ for $\Sigma_{S}$.
\end{definition}
We often will  conflate $\fb$ and its skeleton, the full subcategory with objects $\{ \bn \; | \; n \geq 0 \}$.

%I think going with the convention $[n]=\{1,\ldots,n\}$ makes the most sense. My issue is that something like $\oconf[n]$ should denote the ordered configuration space of $n$ points, but under this notation, it means the configuration space of $n+1$ points, since the cardinality of $[n]$ is $n+1$. On the other hand, Oscar goes by the convention that $[n]=\{0,\ldots,n\}$ }
%\footnote{I need to explain the monoidal product on $\fb$ here or somewher nearby} 
%An object in $\ch_{k, \fb}$, $\topo_{\fb}$, etc. assigns to each element $n\in \N$ an object in 
\begin{definition}   %Let $\mathcal{G}$ denote either $\N$ or $\fb$.
Let $\mathcal{G}$ and $\mathcal{C}$ be categories. Let $\mathcal{C}^{\mathcal{G}}$ denote the category of covariant functors from $\mathcal{C}$ to $\mathcal{G}$. 
Given objects $X$ in  $\mathcal{C}$ and $S$ an  in $\mathcal{G}$, we sometimes will write $X_{S}$ for $X(S)$.

\end{definition}

\begin{definition}
Let $\bK$ be a ring. Let $\Mod_{\bK}$ denote the category of $\bK$-modules. Let $\sMod_{\bK}$ denote the category of simplicial $\bK$-modules. Let $\ch_{\bK}$ denote the category of chain complexes of $\bK$-modules. Let $\sets$ denote the category of sets. Let $\topo$ denote the category of compactly generated weak Hausdorff topological spaces.
\end{definition}

We will refer to a functor $X \colon \mathcal{C} \to \topo$ as a $\mathcal{C}$-space, etc. We will often view a $\mathcal{C}$-set as a $\mathcal{C}$-space or a $\mathcal{C}$-module as a $\mathcal{C}$-chain complex (concentrated in homological degree $0$), etc.

 % We will sometimes refer to objects of $\topo^{\mathcal{\N}}$ and $\topo^{\mathcal{\fb}}$ as $\N$-graded spaces and $\fb$-spaces respectively.
  
  %We will often treat an object $X$ in $ \Mod_{\bK}^{\mathcal{G}}$ as one in  $\sMod_{\bK}^{\mathcal{G}}$ by viewing $\Mod_{\bK}$ as the full subcategory of $\sMod_{\bK}$ of objects $Y$ with $Y_{p}=0$ for $p\neq 0$.  Similarly we may view $X$ in $ \Mod_{\bK}^{\mathcal{G}}$ as an object  in $\ch_{\bK}^{\mathcal{G}}$ by embedding $\Mod_{\bK}$ as the full subcategory of $\ch_{\bK}$  consisting of objects $Y$ with $Y_{i}=0$ for $i\neq 0$.
 %We will sometimes call such an $X$ \emph{discrete}.
 %We will sometimes refer to  to such 
 %(and likewise for objects of $\ch^{\mathcal{\N}}_{\bK}$ and $\ch^{\mathcal{\fb}}_{\bK}$). 

 Let $(\mathcal{G}, \odot)$ be a symmetric monoidal category that is equivalent to a small category. Let $(\mathcal{C},  \otriangleup )$ be a  closed symmetric monoidal category with all small limits and colimits. The functor category $\mathcal{C}^{\mathcal{G}}$ inherits a symmetric monoidal structure $\otimes_{\mathcal{G}}$  called \emph{Day convolution},  due to Day \cite{DayConvolutionI,DayConvolutionII}. It is characterized by a natural isomorphism
 $$ \Hom_ {\mathcal{C}^{\mathcal{G}}}((X \otimes_{\mathcal{G}}Y)(-), Z(-)) \; \cong \; \Hom_ {\mathcal{C}^{(\mathcal{G} \times \mathcal{G}) }}(X(-)  \otriangleup Y(-), Z(- \odot -)) .
 $$
  For example, in $\Mod_{\bK}^{\mathcal{\fb}}$ it is given by the formula
$$(X\otimes_{\mathrm{FB}} Y)_S \colonequals \colim_{f\colon A\sqcup B\xrightarrow{\cong} S} X_A\otimes_{\bK} Y_B \cong \bigoplus_{i+j=|S|}\Ind^{\Sigma_{|S|}}_{\Sigma_{i}\times \Sigma_{j}} X_i \otimes_{\bK} Y_j.$$
 %\footnote{Define Day convolution} 
The action of an FB morphism $\phi \colon S \to T$ is as follows. Given an isomorphism $f \colon A\sqcup B\xrightarrow{\cong} S$, the action of $\phi$ on $(X\otimes_{\mathcal{G}} Y)(S)$ is induced by the map $$\left( (\phi \circ f|_A) \otimes_{\bK} (\phi \circ f|_B)\right) \colon X_A\otimes_{\bK} Y_B \longrightarrow X_{\phi(f(A))} \otimes_{\bK} Y_{\phi(f(B))}.$$
In $\topo^{\N}$ it is given by the formula 
$$(X\otimes_{\N} Y)_n\colonequals \bigsqcup_{i+j=n} X_i\times  Y_j.$$

Using the symmetric monoidal structure provided by the Day convolution product, we can define monoid objects and modules over such objects in  these functor categories.
%such as $\sMod_{\bK}^{\mathcal{G}}$ and $\topo^{\mathcal{G}}$.
\begin{notation}
 Let $(\mathcal{G}, \odot)$ be a symmetric monoidal category that is equivalent to a small category. Let $(\mathcal{C},  \otriangleup  )$ be a  closed symmetric monoidal category with all small limits and colimits. Let $\bunit$ denote the unit object of $\mathcal{C}^{\mathcal{G}}$.

 % Let $\mathcal{G}$ denote either $\N$ or $\fb$.  By abuse of notation, let $\bunit$
    %let   $k\in \sMod_{\bK}^{\mathcal{G}}$ and $*\in\topo^{\mathcal{G}}$ 
   % denote the unit object of the symmetric monoidal categories $\topo^{\mathcal{G}}$ and $\sMod_{\bK}^{\mathcal{G}}$.
    %\footnote{need to fix font issue}
    %\footnote{I'm concerned this notation is confusing.}
    %\footnote{will need to fix notation in later sections}
\end{notation}

\begin{example} In $\Mod_{\bK}^{\mathcal{\fb}}$, $\bunit(\bn) =  \left\{ \begin{array}{ll} \bK, & n=0 \\ 0, & n>0. \end{array}\right.$

In $\topo^{\N}$, $\bunit(n)$ is a point for $n=0$ and is the empty set for $n>0$.
\end{example}

%For example, the unit object $\bK$ in  $\sMod_{\bK}^{\N}$ is the $\N$-graded chain complex which is $\bK$ in grading $0$ and is zero in grading $n$ for $n>0$.
%with $k_{0}= \bK$ and $k_{n}=0$ for $n>0$.\footnote{this is wrong}

Observe that if $X$ is a (semi)-simplicial $\mathcal{C}$-space, then its geometric realization is a  $\mathcal{C}$-space. If $Y$ is a  $\mathcal{C}$-space, its singular chains are a  $\mathcal{C}$-chain complex, its $i$th homology is a  $\mathcal{C}$-module, etc. If you apply a monoidal functor to a monoid object in one of these categories, you obtain a new monoid object.

  We endow $\sMod_{\bK}$ with the standard model structure (see Quillen \cite[Section 2.4]{QuillenHomotopicalAlgebra}) and $\topo$ with the 
  standard (Quillen) model structure (see for example Quillen \cite[Section 2.3] {QuillenHomotopicalAlgebra}  or May--Ponto \cite[Section 17.2]{MayPonto}).

For a category $\mathcal{G}$ that is equivalent to a small category, the  categories $\sMod^{\mathcal{G}}_{\bK}$ and $\topo^{\mathcal{G}}$ then inherit the projective model structure (Hirschhorn \cite[Theorem 11.6.1]{Hirschhorn}), wherein a morphism $X \to Y$ is a weak homotopy equivalence (respectively, a fibration) if and only if $X_S \to Y_S$ is for every $S \in \mathcal{G}$.

We now define a simplicial analogue of spheres. This will be relevant later when constructing generalizations of CW complex structures on spaces to more abstract contexts. 

\begin{definition} Let $\bK$ be a ring. Let $\mathcal{G}$ be a groupoid. Fix $T \in \mathcal{G}$ and $d \geq 0$. 

Let $S_{\bK}^{T,d}$ in $\sMod_{\bK}^{\mathcal{G}}$  be the functor defined as follows. For $U \in \mathcal{G}$ and $p \geq 0$, let 
$$ S_{\bK}^{T,d}(U)_p= \left\{ \begin{array}{ll} \bK[\Hom_{\mathcal{G}}(T,U)]	& \text{ if	}p=d, \\ 
 0		&	\text{otherwise.} \end{array} \right.$$ 

Face maps and degeneracies are (necessarily) trivial.

\end{definition}

We now recall the definition of the two-sided bar construction. See, for example, May \cite[Section 10]{MayGeometryIterated}.  

\begin{definition}
    Let $\mathcal{G}$ be a symmetric monoidal groupoid, and let $\mathcal{C}$ be a symmetric monoidal category that has all small colimits. %Let $\mathcal{C}$ be $(\topo^{\mathcal{G}}$,  $\Mod_{\bK}^{\mathcal{G}},$ $\sMod_{\bK}^{\mathcal{G}}$, or $\ch_{\bK}^{\mathcal{G}}$. 
    Let $\bR$ be a monoid object in $\mathcal{C}^{\mathcal{G}}$ with respect to Day convolution $\otimes_{\mathcal{G}}$. Let $\bM$ be a left $\bR$-module and $\bN$ a right $\bR$-module. Define $B_{\bullet}(\bN, \bR, \bM)$ to be the following simplicial object in $\mathcal{C}^{\mathcal{G}}$. Let 
    $$B_{p}(\bN,\bR, \bM) \colonequals \bN \otimes_{\mathcal{G}} \bR^{\otimes_{\mathcal{G}}\, p} \otimes_{\mathcal{G}} \bM.$$
    The face map 
    $$d_0 \colon (\bN \otimes_{\mathcal{G}} \bR) \otimes_{\mathcal{G}} \left(\bR^{\otimes_{\mathcal{G}}\, (p-1)} \otimes_{\mathcal{G}} \bM\right) \longrightarrow \bN \otimes_{\mathcal{G}} \left(\bR^{\otimes_{\mathcal{G}}\, (p-1)} \otimes_{\mathcal{G}} \bM\right)$$ 
    is induced by the right module structure $\bN \otimes_{\mathcal{G}} \bR \to \bN $.  The face map 
        $$d_p \colon \left(\bN \otimes_{\mathcal{G}} \bR^{\otimes_{\mathcal{G}} \, (p-1)}\right) \otimes_{\mathcal{G}}  (\bR \otimes_{\mathcal{G}} \bM )
        \longrightarrow 
      \left(\bN \otimes_{\mathcal{G}} \bR^{\otimes_{\mathcal{G}} \, (p-1)}\right) \otimes_{\mathcal{G}} \bM$$ 
    is induced by the left module structure $\bR \otimes_{\mathcal{G}} \bM \to \bM $. For $0<i<p$ the face maps 
     $$ d_i \colon \left(\bN \otimes_{\mathcal{G}} \bR^{\otimes_{\mathcal{G}} \, (i-1)} \right) \otimes_{\mathcal{G}}({\bR}\otimes_{\mathcal{G}}{\bR})\otimes_{\mathcal{G}} \left( \bR^{\otimes_{\mathcal{G}} \, (p-i-1)} \otimes_{\mathcal{G}} \bM \right)  $$ $$      \longrightarrow 
     \left(\bN \otimes_{\mathcal{G}} \bR^{\otimes_{\mathcal{G}} \, (i-1)} \right) \otimes_{\mathcal{G}} \bR \otimes_{\mathcal{G}} 
     \left( \bR^{\otimes_{\mathcal{G}} \, (p-i-1)} \otimes_{\mathcal{G}} \bM \right)$$   
 are induced by the monoidal structure $\bR \otimes_{\mathcal{G}} \bR \to \bR$. The degeneracy maps 

$$
s_i \colon  \;  \bN \otimes_{\mathcal{G}} \bR^{\otimes_{\mathcal{G}} \, p} \otimes_{\mathcal{G}} \bM  
  \xrightarrow{\cong}
\left(\bN \otimes_{\mathcal{G}} \bR^{\otimes_{\mathcal{G}} \, (i-1)}   \right) \otimes_{\mathcal{G}} {\bunit}\otimes_{\mathcal{G}} 
\left(\bR^{\otimes_{\mathcal{G}} \, (p-i)}  \otimes_{\mathcal{G}} \bM \right)
 \longrightarrow 
 $$ $$ 
\left( \bN \otimes_{\mathcal{G}} \bR^{\otimes_{\mathcal{G}} \, (i-1)}  \right ) \otimes_{\mathcal{G}} {\bR}\otimes_{\mathcal{G}} \left( \bR^{\otimes_{\mathcal{G}} \, (p-i)}  \otimes_{\mathcal{G}} \bM \right)
$$
are induced by the unit $\bunit \to \bR$ of $\bR$. 
 
%    When $\mathcal{C}$ is $\topo$, we let $B({\bf N},R,\bM)$ denote the $\mathcal{G}$-space whose value on an object $S$ of $\mathcal{G}$ is the thin geometric realization of the simplicial space  $B_{\bullet}({\bf N}, R, \bM)_S$. 

%{\color{cyan}ZZZZZ DO WE WANT NOTATION FOR THE TOTAL COMPLEX OF TH E DOUBLE COMPLEX IN CHAINS? OR DIAGONAL IN SIMPLICIAL THINGS? ZZZZ}

\end{definition}

Let $\mathcal{C}$ be a symmetric monoidal model category. Let $\bR$ be a monoid object in $\mathcal{C}$ and suppose that there is an augmentation map $\epsilon\colon \bR\to \bunit$ from $\bR$ to the unit object $\bunit$ in  $\mathcal{C}$. Let $\bR\text{-mod}$ denote the category of left $\bR$-modules. There is an indecomposables functor 
$$Q^{\bR}\colon \bR\text{-mod}\to \mathcal{C}, $$
(for more details, see \cite[Section 12.2, Section 9.4]{e2cellsI}). The functor $Q^{\bR}$ is a left Quillen functor and so it admits a left derived functor called the \textit{derived indecomposables functor}  $$Q^{\bR}_{\mathbb{L}}\colon \bR\text{-mod}\to \mathcal{C}.$$ The derived indecomposables functor is homotopy invariant in the sense that if $f\colon \bM\to \bM'$ is a map of left $\bR$-modules and a weak homotopy equivalence, then the induced map $Q^{\bR}_{\mathbb{L}}(f)\colon Q^{\bR}_{\mathbb{L}}(\bM)\to Q^{\bR}_{\mathbb{L}}(\bM')$ is also a weak homotopy equivalence. Under certain mild conditions on $\bR$ and $\bM$, we have the following model for  $Q^{\bR}_{\mathbb{L}}(\bM)$.
\begin{lemma}[{\cite[Lemma 9.18]{e2cellsI}}]\label{lem: model for der-ind}
    Let  $\mathcal{C}$ be a symmetric monoidal model category. Let $\bR$ be a monoid object in $\mathcal{C}$ and let $\bM$ be a left $\bR$-module. If $\bR$ and $\bM$ are cofibrant in $\mathcal{C}$, then there is a weak homotopy equivalence
    $$Q^{\bR}_{\mathbb{L}}(\bM)\simeq \Vert B_{\bullet}(\bunit, \bR, \bM)\Vert.$$
\end{lemma}
An example of a cofibrant space  in $\topo$ is any space weakly homotopy equivalent to a CW-complex. If $\bK$ is a field, then any object in $\ch_{\bK}$ or $\sMod_{\bK}$ is cofibrant. If $\mathcal{G}$ is a groupoid, then an object $X\in \mathcal{\topo}^{\mathcal{G}}$ is cofibrant if for each object $S$ in $\mathcal{G}$, $ X_{S}$ is $G$-homotopy equivalent to a $G$-CW complex with $G=\Hom_{\mathcal{G}}(S, S)$. If $\bK$ is characteristic zero and $\Hom_{\mathcal{G}}(S, S)$ is finite for all objects $S$, then all objects of $\ch_{\bK}^\mathcal{G}$ and $\sMod_{\bK}^\mathcal{G}$ are cofibrant.

We conclude this section with a brief discussion of homotopy colimits.

\begin{definition} 
    Let $\Delta_{\text{inj}}$ denote the category of finite linearly ordered sets $$[p]=\{0< 1<\dots < p\},$$ with $p\in\N$, and morphisms the set of all order-preserving maps. Composition is given by composition of functions.
\end{definition}
We remark that with our notation conventions, the underlying set of $[n-1]$ is isomorphic to  $\bn$.  %These distinct notational conventions will help distinguish semi-simplicial parameters from FB paramters. 
    
\begin{definition}
    Let $\mathcal{C}$ be a topologically enriched category.
    %\footnote{might want to move to the beginning}
    A \emph{co-$\mathcal{C}$-space} is an enriched covariant functor $X_{\sqbullet} \colon \mathcal{C}^{\text{op}}\to \text{Top}$.
    %Let $X_{\bullet} \colon \mathcal{C}^{\text{op}}\to \text{Top}$ be an enriched covariant functor. 
    Given $A\in \mathcal{C}$, let $X_{A}$ denote the value of the functor $X_{\sqbullet}$ at $A$.
\end{definition}

\begin{definition}    
    An \emph{augmented co-$\mathcal{C}$-space} is a triple $(X_{\sqbullet}, X_{-1}, \epsilon_{\sqbullet})$, with $X_{-1}\in \topo$, $X_{\sqbullet}$ a co-$\mathcal{C}$-space, and $\epsilon_{\sqbullet}$ an enriched natural transformation from $X_{\sqbullet}$ to the constant co-$\mathcal{C}$-space at $X_{-1}$. 
     The natural transformation $\epsilon_{\sqbullet}$ is called an \emph{augmentation}. We denote the value of an augmentation at an object $A\in \mathcal{C}$ by $\epsilon_{A}\colon X_{A}\to X_{-1}$.
    %\footnote{define map of augmented semi-simplicial stuff}
\end{definition}

The following definition gives an explicit model for certain homotopy colimits. 

\begin{definition} \label{DefnModelHocolim}
    %The bar construction $B(*, \delc, R_{\bullet}(\bM))$ is
    Let $\mathcal{C}$ be a small topologically enriched category. Let $X_{\sqbullet} \colon \mathcal{C}^{\text{op}}\to \text{Top}$ be a co-$\mathcal{C}$-space. Let $B_{\bullet}(*, \mathcal{C}, X_{\sqbullet})$ be the semi-simplicial space with $p$-simplices
    $$\bigsqcup_{A_{0},\ldots, A_{p}\in \text{ob}(\mathcal{C})}\mathcal{C}(A_{0}, A_{1})\times\mathcal{C}(A_{1}, A_{2})\times\cdots \times\mathcal{C}(A_{p-1}, A_{p})\times X_{A_{p}}.$$ The face map $d_{i}$ is induced by dropping the map in $\mathcal{C}(A_{0}, A_{1})$ for $i=0$, by composing the maps in $\mathcal{C}(A_{i-1}, A_{i})$ and $\mathcal{C}(A_{i}, A_{i+1})$ for $1\leq i\leq p-1$, and by evaluation $\mathcal{C}(A_{p-1}, A_{p})\times X_{A_{p}}\to X_{A_{p-1}}$ for $i=p$. 
    The functorial model for the homotopy colimit $\hocolim\limits_{\sqbullet\in \mathcal{C}^{\text{op}}}X_{\sqbullet}$ we use in this paper is
    the thick geometric realization $$\hocolim\limits_{\sqbullet\in \mathcal{C}^{\text{op}}}X_{\sqbullet} \colonequals ||B_{\bullet}(*, \mathcal{C}, X_{\sqbullet}) ||. $$
\end{definition}
%The space $|X_{\bullet}|$ is a model for the homotopy colimit $\hocolim_{A\in \mathcal{C}^{\text{op}}}X_{A}$. 

Note that when $\mathcal{C}$ is $\Delta_{\text{inj}}$, a co-$\Delta_{\text{inj}}$-space $X_{\bullet}$ is  a semi-simplicial space and the  homotopy colimit $ \hocolim\limits_{\bullet \in \Delta_{\text{inj}}^{\op}}X_{\bullet}$ is weakly homotopy equivalent to the thick geometric realization $||X_{\bullet}||$ (see Dugger \cite[Example 9.15]{dug}).

Given an augmentation $\epsilon_{\sqbullet}\colon X_{\sqbullet}\to X_{-1}$, there is an
induced augmented co-$\Delta_{\text{inj}}$-space $$(B(*, \mathcal{C}, X_{\sqbullet}), X_{-1}, B(\epsilon)_{\bullet}),$$ with the 
augmentation
$$B(\epsilon)_{\bullet}\colon B_{\bullet}(*, \mathcal{C}, X_{\sqbullet})\to X_{-1}$$ coming from $\epsilon_{\sqbullet}$. As a result, an augmentation $\epsilon_{\sqbullet}\colon X_{\sqbullet}\to X_{-1}$ induces a map
 $$\hocolim\limits_{\sqbullet\in \mathcal{C}^{\op}}\epsilon_{\sqbullet}\colonequals ||B(\epsilon)_{\bullet}||\colon \hocolim\limits_{\sqbullet\in \mathcal{C}^{\op}}X_{\sqbullet}\to X_{-1}.$$ 
%When $\mathcal{C}$ is the category $\Delta$ of finite linearly ordered sets $\{0< 1<\ldots < n\}$, with $n\in\N$, and maps $\{0< 1<\ldots < n\}\to \{0< 1<\ldots < m\}$, a $\mathcal{C}$-space is called a semi-simplicial space 

We note that the semi-simplicial object $B_{\bullet}(*, \mathcal{C}, X_{\sqbullet})$ naturallly has the structure of a simplicial object, where
the degeneracy $s_i$ is given by insertion of the identity morphism in $\mathcal{C}(A_i, A_i)$. However, we will not use this additional structure; we only consider the {\em thick} geometric realization.

\section{Models of Hurwitz space}

In this section, we describe a model for  Hurwitz spaces and ordered Hurwitz spaces. We use variants of the configuration mapping space models of Hurwitz space from Ellenberg--Venkatesh--Westerland \cite{EVW2}. We add extra parameters to produce strictly associative monoid structures on these spaces.
%\subsection{Notation and Categorical Setup} 

%\begin{notation} Let $G$ be a finite group and $c \subset G$ be a conjugation-invariant subset. Given $a, b\in G$, let $a^{b}\in G$ denote the conjugate $b^{-1}a b$.
%\end{notation}
%\begin{notation}
%    Given $a, b\in G$, let $a^{b}\in G$ denote the conjugate $b^{-1}a b$.
%\end{notation}

\begin{definition}
    Given topological spaces $X$ and $Y$, let $\Emb(X, Y)$ denote the set of continuous injective maps $f\colon X\hookrightarrow Y$, topologized with the compact-open topology.  We define an $\fb$-space $\oconf[]$, and we write $\oconf[S]$ for its value on a finite set $S$.  For $S \neq \varnothing$, let 
    $$\oconf[S]\colonequals\left\{(\lambda, \xi)\in \mathbb{R}_{\geq 0}\times \Emb\big(S, (0,\infty)\times (0,1)\big) \; \middle| \;\im(\xi)\subset  (0,\lambda)\times (0,1)\right\}.$$
    %\begin{align*}
    %        \oconf[S]=\{(\lambda, \xi)&\in \mathbb{R}_{> 0}\times \Emb(S, (0,\infty)\times (0,1)):\\ &g(s)\in  (0,\lambda)\times (0,1) \text{ for all } s\in S\}
            %%(&x_{1},\ldots, x_{|S|}))\in \Hom_{FB}([|S|], S)\times (0,\infty)\times ((0,\infty)\times (0,1))^{|S|}:\\ & x_{i}\in (0, \lambda)\times (0,1) \text{ for all }i \text{ and } x_{i}\neq x_{j} \text{ for all } i\neq j\}
    %\end{align*}
    Let $\oconf[\varnothing]$ be a single point, written as $(0, \varnothing)$.
    \end{definition}

   The parameter $\lambda$ is analogous to the length parameter in a Moore loop space.   We view an element of $\oconf$ as a length parameter $\lambda$ and an $n$-element subset of $(0,\lambda) \times (0,1) \subseteq \R^2$ which is labeled bijectively by elements of $S$, as illustrated in Figure \ref{NewFigurePointOConf}.

   \begin{figure}[ht!]
    \centering
    \includegraphics[scale=.5]{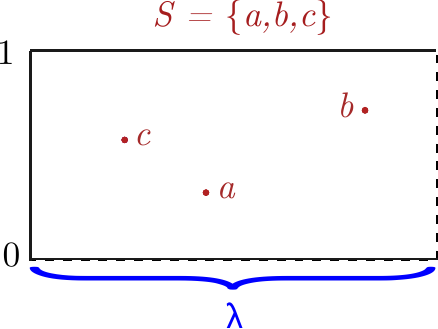}
    \caption{A point in $\oconf[\{a,b,c\}]$.}
    \label{NewFigurePointOConf}
\end{figure}

  We defined $\oconf[\varnothing]$ to be a point instead of a contractible space for technical convenience.

    Given a bijection of finite sets $h\colon S\to T$, we have an induced map $h_{*}\colon \oconf[S]\to \oconf[T] $ by precomposing with $h^{-1}$, 
    \begin{align*}
        h_{*}\colon \oconf[S]&\longrightarrow \oconf[T]\\
        (\lambda, \xi)&\longmapsto (\lambda, \xi\circ h^{-1}).
        %(g, \lambda, (x_{1},\ldots, x_{|S|}))&\mapsto (h\circ g, \lambda, (x_{1},\ldots, x_{|S|})).
    \end{align*}
 See Figure \ref{o-FB-Action}.
   \begin{figure}[ht!]
    \centering
    \includegraphics[scale=.5]{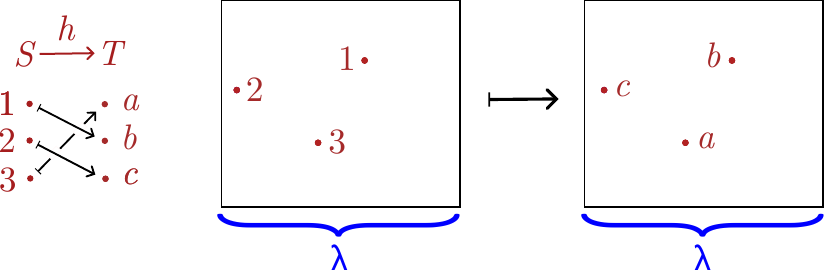}
    \caption{The action of a bijection $h\colon S \to T$.}
    \label{o-FB-Action}
\end{figure}

 \begin{definition}
    Let $\conf[]$ be the $\N$-graded space with $\conf[n]=\oconf[\bn]/\Sigma_{n}$ for all $n \geq 0$.
\end{definition}
  We view an element of $\conf$ as a length parameter $\lambda$ and an $n$-element subset of $(0,\lambda) \times (0,1) \subseteq \R^2$, as in Figure \ref{FigurePointConf}. By abuse of notation, we may denote the equivalence class of an injection $\xi \colon \bn \to \R^2$ by its image. 

\begin{figure}[ht!]
    \centering
      \includegraphics[scale=0.5]{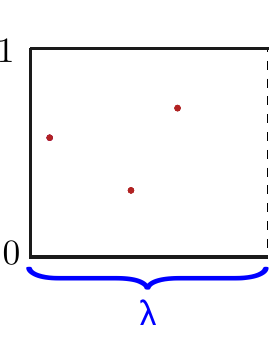}
    \caption{A point in $\conf[3]$.}
    \label{FigurePointConf}
\end{figure}

%We will often use the notation $(\lambda, \xi)$ to denote an element of $\oconf[]$ or $\conf[]$. 
%The map
%\begin{align*}
%    \phi\colon \oconf[S]&\to \conf[|S|]\\
%    (\lambda, \xi) &\mapsto (\lambda, \{\im(\xi)\})
%\end{align*}
%is a covering space.
%\footnote{maybe drop this. In fact, it probably makes more sense to change $(\lambda, \xi)$ to $(\lambda, \xi)$} 
The space $\oconf[]$ is a monoid object in $\topo^{\fb}$, with the unit given by $(0, \varnothing)$. The product is defined by the formula

\begin{align*} 
\oconf[S_{0}] \times \oconf[S_{1}] &\longrightarrow 
 \oconf[S_{0}\sqcup S_{1}] \\ 
(\lambda_{0}, \xi_{0})\cdot (\lambda_{1}, \xi_{1}) & \colonequals (\lambda_{0}+ \lambda_{1}, \xi)
 \end{align*} 

    %$$(\xi_{0}, \lambda_{0}, \xi_{0})\cdot (\xi_{1}, \lambda_{1}, \xi_{1})\colonequals (\xi_{0}\sqcup \xi_{1}, \lambda_{0}+ \lambda_{1}, \xi_{2}),$$ where $$\xi_{2}=( (a_{0, 1}, b_{0, 1}),\ldots, (a_{0, |S_{0}|}, b_{0, |S_{0}|}), (a_{1, 1}+\lambda_{0}, b_{1, 1}),\ldots , (a_{1, |S_{1}|}+\lambda_{0}, b_{1, |S_{1}|})),$$ and
    %$$\xi_{0}\sqcup \xi_{1}\colon [|S_{0}\sqcup S_{1}|]=[|S_{0}|+ |S_{1}|]\to S_{0}\sqcup S_{1}$$ is the map which sends $j$ to $\xi_{0}(j)$ if $j=1,\ldots, |S_{0}|$ and which sends it to $\xi_{1}(j-|S_{0}|)$ if $j=|S_{0}|+1,\ldots, |S_{0}|+|S_{1}|$.
    
   $$ \text{where } \xi(s)=\begin{cases} \xi_{0}(s) & \text{for } s\in S_{0},\\
    \xi_{1}(s) + (\lambda_{0}, 0) & \text{for } s\in S_{1}.
    \end{cases}$$    
    %$\xi_{0}\sqcup \text{sh}_{\lambda_{0}}(\xi_{1})\colon S_{0}\sqcup S_{1}\to (0,\infty)\times (0,1)$ is the map sending 
See Figure \ref{FigureConfProduct}.
\begin{figure}[ht!]
    \centering
      \includegraphics[scale=0.5]{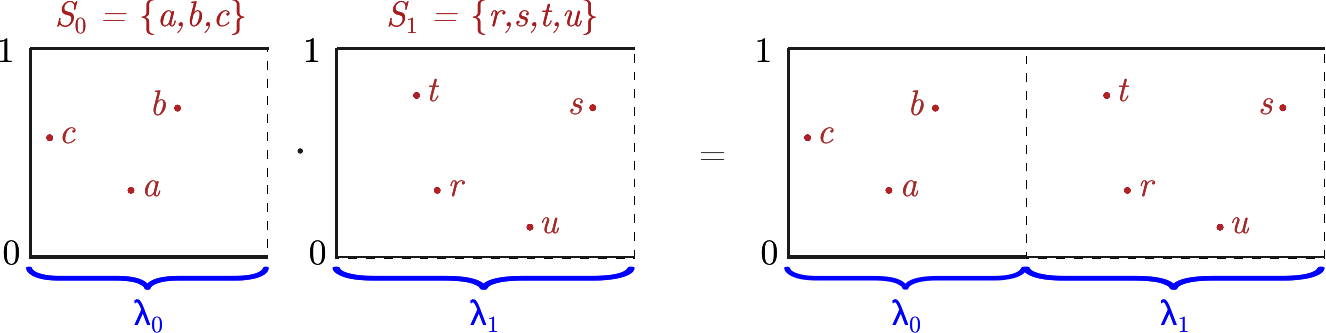}
    \caption{The product $(\lambda_{0}, \xi_{0})\cdot (\lambda_{1}, \xi_{1})$ is defined by concatenation of configurations.}
    \label{FigureConfProduct}
\end{figure}

   A monoid structure on $\conf[]$ is defined analogously.

\begin{definition}
    Let $\sigma\in \conf[1]$ denote the element $(1, \left\{\left(\frac12, \frac12)\right\}\right)$.   For $p\geq 1$,  let $\sigma^{p}=\sigma\cdot\ldots\cdot \sigma$ ($p$ times) in  $\conf[p]$.
\end{definition}

Our goal is to define an $\fb$-space that we will call ordered Hurwitz space, and denoted $\ohur[]$. For a finite set $S$, we will first describe $\ohur[,S]$ as a set and then describe its topology. Our model for Hurwitz space will be a Moore  path version of the configuration mapping space model from Ellenberg--Venkatesh--Westerland \cite[Definition 2.2]{EVW2}.

\begin{definition}
Let $G$ be a finite group and $c \subseteq G$ be a conjugation-invariant subset.    Let $\ohur[]$ denote the $\fb$-space whose value $\ohur[,S]$ on a finite set $S$ has underlying set the set of tuples $((\lambda, \xi), f)$ defined as follows. 
    %$((g, \lambda, \xi), f)$ 
   For $S\neq \varnothing$,
    %\footnote{change $t$ to $\lambda$}
    \begin{enumerate}
        \item $(\lambda, \xi)\in\oconf[S]$,
        %\item $\xi$ is a configuration of $n$ distinct ordered points in $(0,\lambda)\times (0,1)$, and
        \item $f\colon \Big([0, \lambda]\times [0,1]\Big)\setminus \im(\xi) \to BG$ is a continuous map satisfying the following conditions. It sends the subset $${\textstyle \boldsymbol{\bigsqcup}_{\lambda}}\colonequals (\{0, \lambda\} \times [0,1])\cup ([0,\lambda]\times \{0\})$$ to the basepoint of $BG$. For each $s \in S$ the map $f$ sends a small loop encircling $\xi(s)$ (and no other points of $\xi(S)$) to a loop in $BG$ representing an element of $c\subseteq G$. 
        We remark that we will write ${\textstyle \boldsymbol{\bigsqcup}}$ for ${\textstyle \boldsymbol{\bigsqcup}_{\lambda}}$ when $\lambda$ is clear from context. 
    \end{enumerate}
  For $S= \varnothing$, we set $\ohur[,0]$ to be a single point $((0, \varnothing), f_{0})$, where   
    %we set $\ohur[,0]=\{1\}$ to consist of a single point.   
    %given by a single point, obtained we set $t$ and $\xi$ to be $0$ and  $\{\emptyset\}$ respectively, and 
    $f_{0}$ is the map $f_0 \colon \{0\}\times [0,1] \to BG$ sending everything to the basepoint in $BG$.
       \end{definition}

A point in $\ohur[,\{r,s,t\}]$ is illustrated in Figure \ref{FigurePointOHur}.
\begin{figure}[ht!]
    \centering
      \includegraphics[scale=0.65]{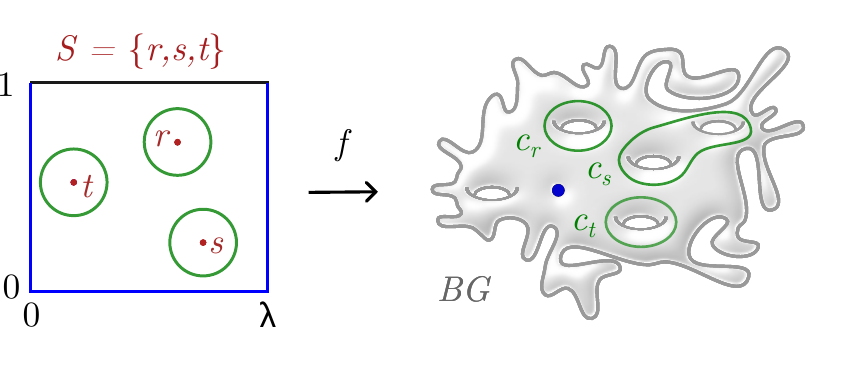}
    \caption{An element of $\ohur[,\{r,s,t\}]$. Each of the (unbased homotopy class of) loops $c_i$ corresponds to a conjugacy class of $G$ contained in $c$. The subset $\boldsymbol{\bigsqcup}$, and its image the basepoint of $BG$, are shown in blue.}
    \label{FigurePointOHur}
\end{figure}

 The space $\ohur[, S]$ is topologized as in Ellenberg--Venkatesh--Westerland 
 \cite[Definition 2.2]{EVW2}, as follows.  For this paragraph we let $\boldsymbol{\Box}$  denote the square $[0,1] \times [0,1]$ in $\mathbb{R}^2$. For $\lambda \in \R_{>0}$, the matrix $$ M_{\lambda}\colonequals \begin{bmatrix} \lambda & 0 \\ 0 & 1 \end{bmatrix}$$ scales $\boldsymbol{\Box}$ by a factor of $\lambda$ in the horizontal direction. Given a finite set $S \neq \varnothing$, fix an embedding $\xi_0 \colon S \to \boldsymbol{\Box}$.
 Let $\mathrm{Diff}^+(\boldsymbol{\Box}\, , \partial \,\boldsymbol{\Box})$ denote the group of diffeomorphisms of $\boldsymbol{\Box}$ that fix the boundary pointwise, and let  $\mathrm{Diff}^+(\boldsymbol{\Box}, \xi_0)$ denote the subgroup of $\mathrm{Diff}^+(\boldsymbol{\Box}\, , \partial \, \boldsymbol{\Box})$ that fixes the subset $\xi_0(S) \subseteq \boldsymbol{\Box}$ pointwise. 
Then the map 
\begin{align*}
   \R_{>0} \times \mathrm{Diff}^+(\boldsymbol{\Box}, \partial \, \boldsymbol{\Box})\times \mathrm{Map}(\boldsymbol{\Box} \setminus \xi_0(S), BG) & \longrightarrow \mathrm{OHur}^G_{G, S}  \\ (\lambda, \phi, f) & \longmapsto ((\lambda, \,  M_{\lambda} \circ \phi \circ \xi_0), \, f\circ \phi^{-1} \circ M_{\lambda}^{-1}) 
\end{align*}
is invariant under the action of $\mathrm{Diff}^+(\boldsymbol{\Box}\, , \xi_0)$ by $$\rho \colon (\lambda, \phi, f) \longmapsto (\lambda, \,  \phi \circ \rho, \, f \circ \rho).  $$ 
As in Ellenberg--Venkatesh--Westerland 
 \cite[Equation 2.3.1]{EVW2}, the map factors through a bijection of sets
\begin{align*}
   \R_{>0} \times \mathrm{Diff}^+(\boldsymbol{\Box} \,,   \partial\,  \boldsymbol{\Box})\times_{  \mathrm{Diff}^+(\boldsymbol{\Box}\, , \xi_0)} \mathrm{Map}(\boldsymbol{\Box} \, \setminus \xi_0(S), BG) & \xrightarrow{\cong} \mathrm{OHur}^G_{G,S} .
\end{align*}

The topology on the left-hand side (determined by the compact-open topologies) therefore defines a topology on
$\mathrm{OHur}^G_{G,S}$.  For a conjugation-invariant subset $c \subseteq G$, we topologize  $\mathrm{OHur}^c_{G, S} \subseteq \mathrm{OHur}^G_{G,S}$ with the subspace topology.

    We note that, as in Ellenberg--Venkatesh--Westerland 
 \cite[Proposition 2.3.1]{EVW2},  the map
    \begin{align*}
        \phi\colon \ohur[, S]&\longrightarrow \oconf[S]\\
        ((\lambda, \xi), f) &\longmapsto (\lambda, \xi)
    \end{align*}
    %$\ohur[, S]\to \oconf[S]$ sending $((\lambda, \xi), f)$ to $(\lambda, \xi)$ 
    is a fiber bundle, whose fiber over 
    $(\lambda, \xi)$ 
    %$(\lambda, \xi)$ 
    is the space of maps
    $$f\colon\left ([0, \lambda]\times [0,1]\setminus \im(\xi), \boldsymbol{\bigsqcup}\right) \to (BG, \text{pt})$$
    %$f\colon ([0, \lambda]\times [0,1]\setminus \xi, \sqcup) \to (BG, \text{pt})$ 
    sending a small loop around each point of $\xi$ to a loop in the conjugation-invariant subset $c$.

    \begin{definition} 
    Let $G$ be a finite group and $c \subseteq G$ be a conjugation-invariant subset.  Given a bijection of finite sets $h\colon S\to T$, we have an induced map
    $h_{*}\colon \ohur[, S]\to \ohur[, T]$
    %The symmetric group $\Sigma_{n}$ acts
    %%\footnote{I'm not sure that this action is correctly defined}
    via the induced map on configuration spaces, 
    %on $\ohur[, S]$ by its action on $\oconf[S]$, i.e. 
    %for $\tau\in \Sigma_{n}$,
    %$$h_{*}((g, \lambda, \xi), f)\colonequals (h_{*}(\lambda, \xi), f).$$
    $$h_{*}((\lambda, \xi), f)\colonequals ((\lambda, \xi \circ h^{-1}), f).$$
    %Let $\ohur$ be the $\N$-graded space $\coprod\limits_{n\geq 0}\ohur[,n]$ with $(\ohur)_{n}=\ohur[,n]$. 
Let $\ohur\in \topo^{\fb}$ denote the $\fb$-space with $\ohur(S)\colonequals \ohur[, S]$. Let $\hur$ denote the $\N$-graded space with $\hur[,n]=\ohur({\bn} )/\Sigma_{n}$ for all $n \in \N$.
 \end{definition}

Given $a, b\in G$, we will write $a^{b}\in G$ for the conjugate $ba b^{-1}$.

\begin{definition}
Let $G$ be a finite group and $c \subseteq G$ be a conjugation-invariant subset. Let $\Br_n$ denote the braid group on $n$ strands and $\PBr_n$ denote the pure braid group on $n$ strands. Define the $j$-th Artin generator of $\Br_n$ 
%The braid group $\Br_n$ has the Artin generators $\alpha_{1},\ldots, \alpha_{n-1}$ as a generating set, with $\alpha_{j}$ 
to be the braid that moves the $j$-th strand one time across the $(j+1)$-st strand and leaves all other strands unchanged.
 The braid group $\Br_n$ acts on $c^n$ on the left as follows. The action of the $j$-th Artin generator is
\begin{equation}\label{eq: brd grp action}
  (c_1,\ldots, c_j, c_{j+1},\ldots, c_n) \longmapsto (c_1,\ldots, c_{j-1},  c_{j+1}^{c_{j}}, c_j, c_{j+2},\ldots, c_n).
\end{equation}

\end{definition}
It follows from Artin's presentation for $\Br_n$ that this action is well-defined.

By Ellenberg--Venkatesh--Westerland \cite[Section 5.7 \& 5.8]{EVW2}, the space $\hur[,n]$ is homotopy equivalent to the homotopy quotient $(c^{n})_{h \Br_{n}}$ of $c^{n}$ by the action (\ref{eq: brd grp action}) of the braid group  $\Br_{n}$. See Figure \ref{ActionArtinGenerator}.
    \begin{figure}[ht!]
    \centering
    \includegraphics[scale=.7]{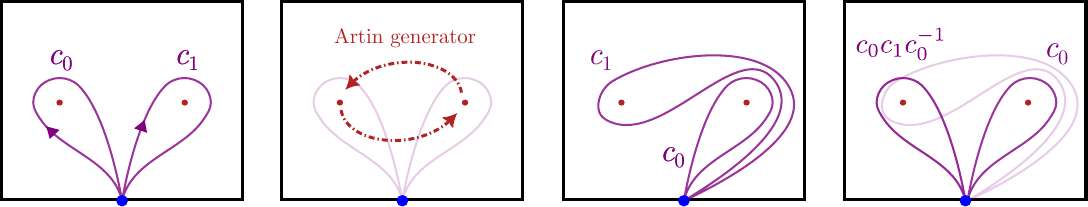}
    \caption{The action of an Artin generator of $\Br_n$ on $c^n$ from the perspective of configuration mapping spaces. Compare to Formula (\ref{eq: brd grp action}).}
    \label{ActionArtinGenerator}
\end{figure}
Similarly, there is a homotopy equivalence $\ohur[,\bn] \simeq (c^n)_{h \PBr_{n}}$.
 
    The object $\ohur\in \topo^{\fb}$ has the structure of a monoid object in $\topo^{\fb}$, where $((0, \varnothing), f_{0})$ is the unit, and 
    the product is given by the formula,
    %The space $\ohur$ is a monoid in $\topo_{\fb}$, with the unit given by $\ohur[,0]$ and the product of $(t_{0}, \xi_{0}, f_{0}), (t_{1}, \xi_{1}, f_{1})\in \ohur$, with $\xi_{i}=( (a_{i, 1}, b_{i, 1}),\ldots, (a_{i, n_{i}}, b_{i, n_{i}})$, given by the formula
   \begin{align*}
       \ohur \times \ohur & \longrightarrow \ohur \\ 
   ((\lambda_{1}, \xi_{1}), f_{1})\cdot ((\lambda_{2}, \xi_{2}), f_{2}) &\colonequals ((\lambda_{1}, \xi_{1})\cdot (\lambda_{2}, \xi_{2}), f_{3}),
   \end{align*}
    %$$((\xi_{0}, \lambda_{0}, \xi_{0}), f_{0})\cdot ((\xi_{1}, \lambda_{1}, \xi_{1}), f_{1})\colonequals ((\xi_{0}, \lambda_{0}, \xi_{0})\cdot (\xi_{1}, \lambda_{1}, \xi_{1}), f_{2}),$$ 
    where 
    %$$\xi_{2}=( (a_{0, 1}, b_{0, 1}),\ldots, (a_{0, n_{0}}, b_{0, n_{0}}), (a_{1, 1}+t_{0}, b_{1, 1}),\ldots , (a_{1, n_{1}}+t_{0}, b_{1, n_{1}})),$$ and
    $$f_{3}(x, y)=\begin{cases} f_{1}(x, y) & \text{if } 0\leq x \leq \lambda_{1},\\
    f_{2}(x-\lambda_{1}, y) & \text{if } \lambda_{1}\leq x \leq \lambda_{1}+\lambda_{2}.
    \end{cases}$$
    The monoid structure on $\ohur$ gives $\hur$ the structure of a monoid object in $\topo^{\N}$. 

It follows from the homotopy equivalence $\hur[,n] \simeq (c^{n})_{h \Br_{n}}$ that there is an isomorphism of $\bK$-modules $H_{0}(\ohur[, n]; \bK) \cong \bK[c^{n}\slash \Br_{n}]$, where $\bK[c^{n}\slash \Br_{n}]$ denotes the free $\bK$-module on the set of orbits $c^{n}\slash \Br_{n}$. Similarly, $H_{0}(\ohur[, \bn];\bK) \cong \bK[c^{n}\slash \PBr_{n}]$.

%Recall that, for $p \in \N$, we write $[p]$ for the ordered set $\{0, 1, \dots, p\}$ with the usual ordering. 
    \begin{definition} \label{Defni}
        For each $w\in c$, fix a map $$f_{w}\colon \Big([0,1]\times [0,1] \Big) \setminus \left\{\left(\textstyle\frac12, \frac12\right)\right\}\to BG$$ sending $\boldsymbol{\bigsqcup}$ to the basepoint of $BG$, and which maps a generator of the fundamental group of the boundary of $[0,1] \times [0,1]$  to $w$ in $\pi_1(BG)$, as shown in Figure \ref{FigureMapfw}. Let $i(w)\in \hur[,1]$ denote $\left(\left(1, \left(\frac12, \frac12\right)\right), f_{w}\right)$.     

        \begin{figure}[ht!]
    \centering
      \includegraphics[scale=0.65]{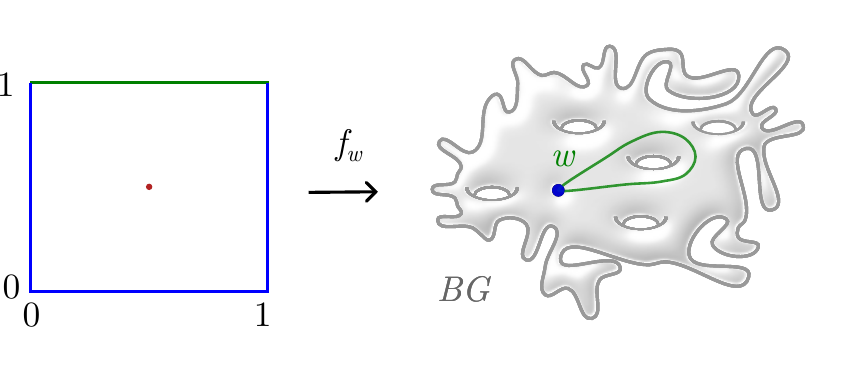}
    \caption{The point $i(w) = \left(\left(1, \left(\frac12, \frac12\right)\right), f_{w}\right) \in  \hur[,1]$.}
    \label{FigureMapfw}
\end{figure}
         Given $v=(v_{0},\ldots, v_{p})\in c^{p+1}$, let $i(v)\in \hur[, p+1]$ denote the element
        $$i(v)\colonequals i(v_{0})\cdot i(v_{1})\cdot\ldots \cdot i(v_{p})$$
        as shown in Figure \ref{Figureiv}. We denote by $D_p$ the domain of the map associated to $i(v)$, 
    $$D_p \colonequals \Big( [0, p+1] \times [0,1] \Big) \setminus \left\{{ \textstyle \left(\frac12, \frac12\right), \left(1+\frac12, \frac12\right),  \dots ,\left(p+\frac{1}{2}, \frac12\right)} \right\},$$ and we call this map $$g_v \colon D_p \to BG.$$ 

The map $g_v$ together with the configuration $\sigma^{p+1}$ gives an element of $\hur[,p+1]$. The homotopy equivalence $$ \hur[,p+1] \simeq c^{p+1} /\!\!/ \Br_{p+1}  $$ established by Ellenberg--Venkatesh--Westerland \cite{EVW2} gives a bijection $$\pi_0\left(\hur[,p+1] \right) \cong c^{p+1} / \Br_{p+1}. $$ 
Note that the element corresponding to $i(v)$ in $\pi_0(\hur[,p+1])$ agrees with the image of $v$ under the map $$c^{p+1} \to c^{p+1}/\Br_{p+1} \cong \pi_0(\hur[,p+1]).$$

        \begin{figure}[ht!]
    \centering
      \includegraphics[scale=0.5]{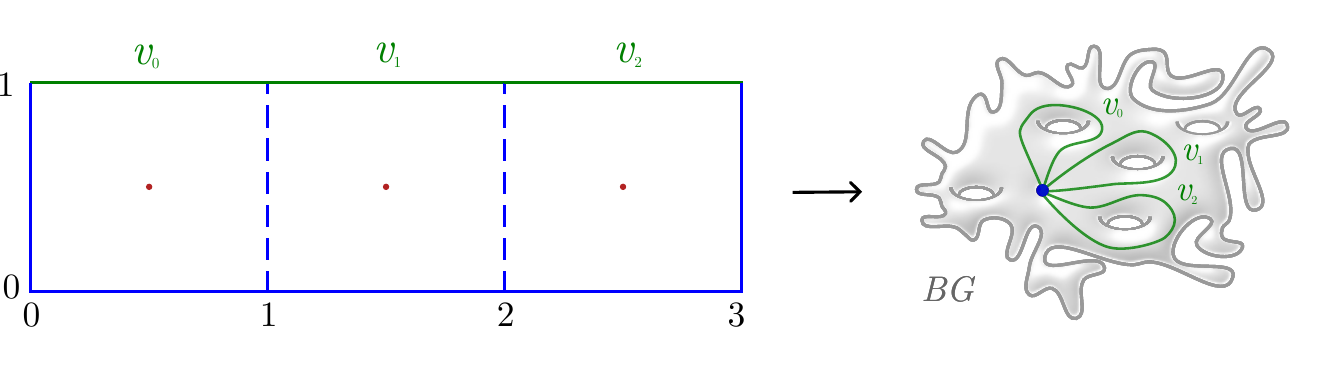}
    \caption{The element $i(v)\in \hur[, 3]$ for $v=(v_0, v_1, v_2)$.}
    \label{Figureiv}
\end{figure}
        
        Given an object $S\in \fb$ and an element $s\in S$, 
        let $\xi_{s}\colon \{s\}\to (0, \infty)\times (0, 1)$ denote the map sending $s$ to $\left(\textstyle\frac12, \frac12\right)$.
        %let $\xi_{s}\colon [1]\to \{s\}$ denote the map sending $1$ to $s$.
        Let $i(w, s)\in \ohur(\{s\})$ denote the element 
        $((1, \xi_{s}), f_{w})$ shown in Figure \ref{FigureOHuriws}. 
                    \begin{figure}[ht!]
   
    \centering
      \includegraphics[scale=0.5]{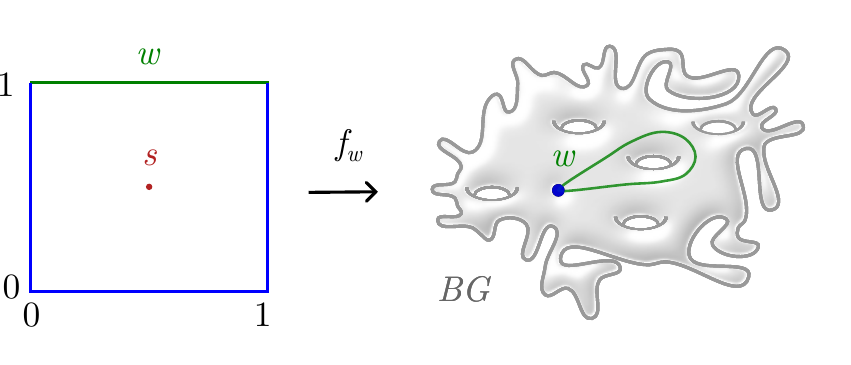}
    \caption{The element $i(w, s)=((1, \xi_{s}), f_{w})\in \ohur(\{s\})$ .}
    \label{FigureOHuriws}
\end{figure}
        
       Recall that $[p]$ denotes the ordered set $\{0,1,\dots, p\}$ with the usual ordering. Given an injection $\alpha\colon [p]\hookrightarrow S$ and $v=(v_{0},\ldots, v_{p})\in c^{p+1}$, let $i(v, \alpha)\in\ohur(\im(\alpha))$ denote the element
        $$i(v, \alpha)\colonequals i(v_{0}, \alpha(0))\cdot i(v_{1}, \alpha(1))\cdot\ldots \cdot i(v_{p}, \alpha(p))$$
        shown in Figure \ref{FigureOHuriva}.

                \begin{figure}[ht!]
    \centering
      \includegraphics[scale=0.5]{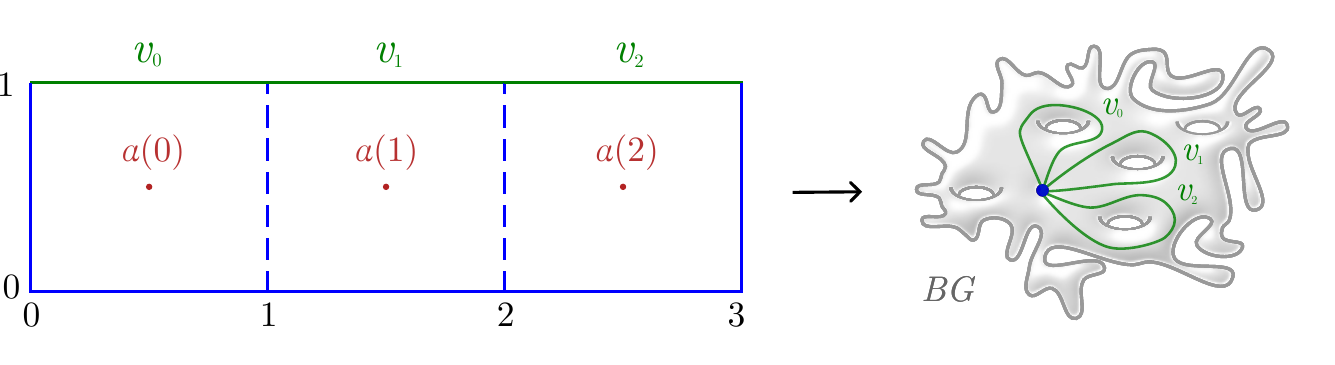}
    \caption{The element $i(v, \alpha)\in\ohur[, \im(\alpha)]$ for $v=(v_0, v_1, v_2)$.}
    \label{FigureOHuriva}
\end{figure}
    \end{definition}

\section{Variants of the semi-simplicial category} 

In this section, we first recall a topological version of the semi-simplicial category introduced by Randal-Williams \cite[Section 2]{MR4767884} to construct resolutions  of modules over configuration spaces. Then we introduce a variant of this category which we will use in \cref{ResHurSec} to construct resolutions of modules over Hurwitz spaces.

We first recall Randal-Williams' construction \cite[Definition 2.2]{MR4767884}  of a topological category equivalent to $\Delta_{\text{inj}}$.

\begin{definition}
    A \textit{Moore path} on a space $X$ is a pair $(T, \mu)$, where $T\in [0,\infty)$ and $\mu\colon [0,\infty)\to X$ is a map such that $\mu(t)=\mu(T)$ for all $t\geq T$. If $(T_{0}, \mu_{0})$ and $(T_{1}, \mu_{1})$ are Moore paths in $X$ and $\mu_{1}(0)=\mu_{0}(T_{0})$, the \textit{concatenation} $(T_{0}, \mu_{0})*(T_{1}, \mu_{1})$ of $(T_{0}, \mu_{0})$ and $(T_{1}, \mu_{1})$ is defined as the Moore path $(T_{0}+T_{1}, \mu_{2})$, where  $$\mu_{2}(t)=\begin{cases} \mu_{0}(t) & \text{if } 0\leq t \leq T_{0},\\
    \mu_{1}(t-T_{0}) & \text{if } t \geq T_{0}.
    \end{cases}$$
    We denote the space of Moore paths on $X$ by $\text{Moore}(X)$.
\end{definition}
We will often suppress notation and write a Moore path as $\mu$ instead of $(T, \mu)$.
%Throughout, we will work with Moore paths and write $*$ for concatenation of Moore paths.\footnote{define Moore paths}
%Now we describe analogues of $\delc$ and $R_{\bullet}(\bM)$ for a $\conf[]$-module $\bM$, as described in \cite{MR4767884}.\footnote{move this stuff to the beginning of the section before doing Hurwitz analogues}
\begin{definition}\label{def:tildedel}
    For $[q], [p] \in \Delta_{\text{inj}}$ let $U([q], [p])$ denote the space of pairs $(a, \mu)$, with  $a \in \conf[p-q]$ and a Moore path $\mu$ in $\conf[p]$ from $\sigma^{p+1}$ to $\sigma^{q+1}\cdot a$. See Figure \ref{FigureUqpComposition}.
    
    There is a composition law
    \begin{align*}
       U([r], [q])\times U([q], [p])&\to U([r], [p])\\
       ((a', \gamma),(a, \mu))&\mapsto (a'\cdot a, \mu* ( \gamma\cdot a))
    \end{align*}
    that gives $U$ the structure of a topologically enriched category. 
    \end{definition}
    The composition law is illustrated in Figure \ref{FigureUqpComposition}. 
        
          \begin{figure}[ht!]
    \centering
      \includegraphics[scale=0.18]{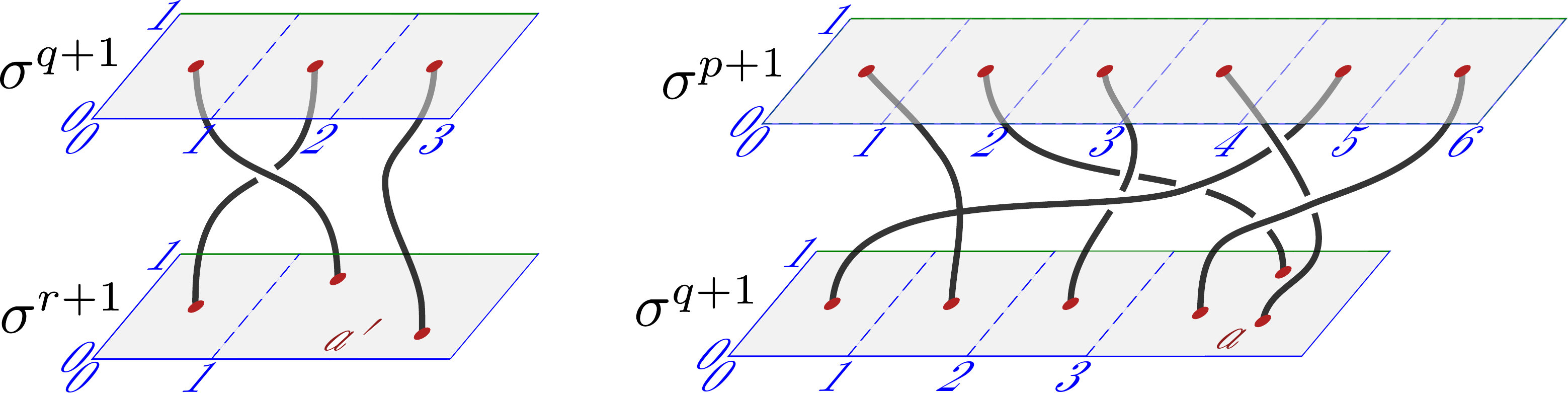} \hfill
\includegraphics[scale=0.18]{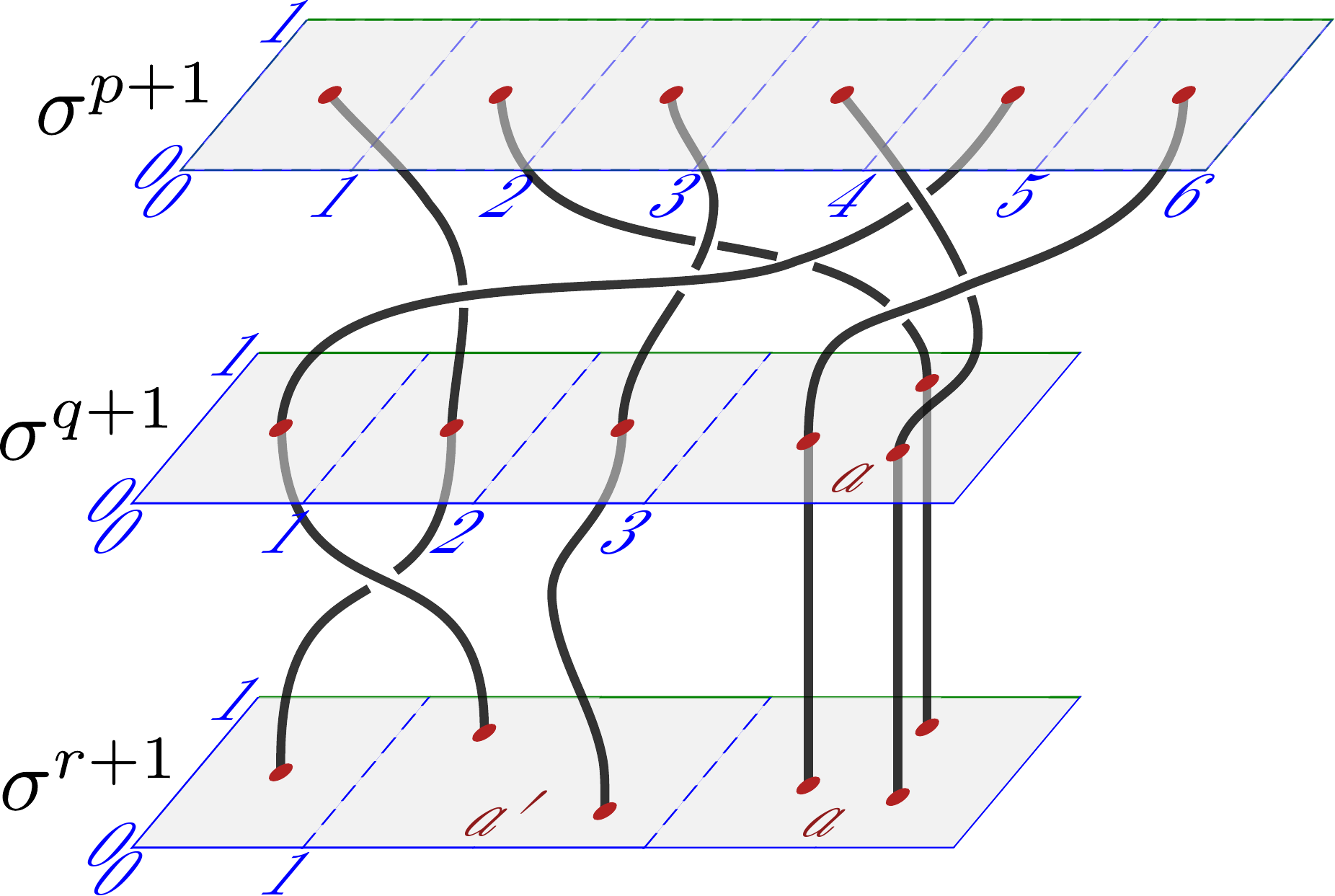}
    \caption{Morphisms in  $U([0], [2])$ and $U([2], [5])$, and their composite.}
    \label{FigureUqpComposition}
\end{figure}

%defined in exactly the same way that the composition law for $U_{c}$ was defined, that gives $U$ the structure of a topologically enriched category.

Consider a morphism $(a,\mu) \in U([q],[p])$. Observe that the Moore path $\mu$ determines the element $a$ but we keep it in the notation.

\begin{definition} \label{DefnLiftMonotone} Given a monotone injection $\iota\colon [q]\to [p]\in \Delta_{\text{inj}}$, we will specify a corresponding path component of $U([q], [p])$,  containing the pair given by the point $\sigma^{p-q} \in \conf[p-q]$ and a Moore loop satisfying the following condition. As a braid, the $(q+1)$ strands  starting at positions numbered $\iota([q])$ do not cross one another and end at positions $0, 1, \dots, q$.   These strands pass in front of the complementary $p-q$ strands, where `in front' refers to the orientation of $\R^2$ depicted in our figures. For $k=0, 1, \dots, q$, the strand starting at position $\iota(k)$ ends at position $k$.

For each monotone injection $\iota\colon [q]\to [p]\in \Delta_{\text{inj}}$, we let $\mu_{\iota}$ denote the Moore path of some fixed choice of representative of the form $(\sigma^{p-q}, \mu_{\iota})$ of the corresponding path component of $U([q], [p])$. 
\end{definition}

For example, Figure \ref{DeltaTwiddleExample} shows two morphisms of $U([2], [5])$ in the path component associated one particular monotone map. 
\begin{comment}
     \begin{align*} 
        \iota \colon [2] & \longrightarrow [5] \\ 
        0 & \longmapsto 1 \\ 
        1 & \longmapsto 3 \\ 
        2 & \longmapsto 4. 
    \end{align*}}
\end{comment}
              \begin{figure}[ht!]
    \centering
      \includegraphics[scale=0.23]{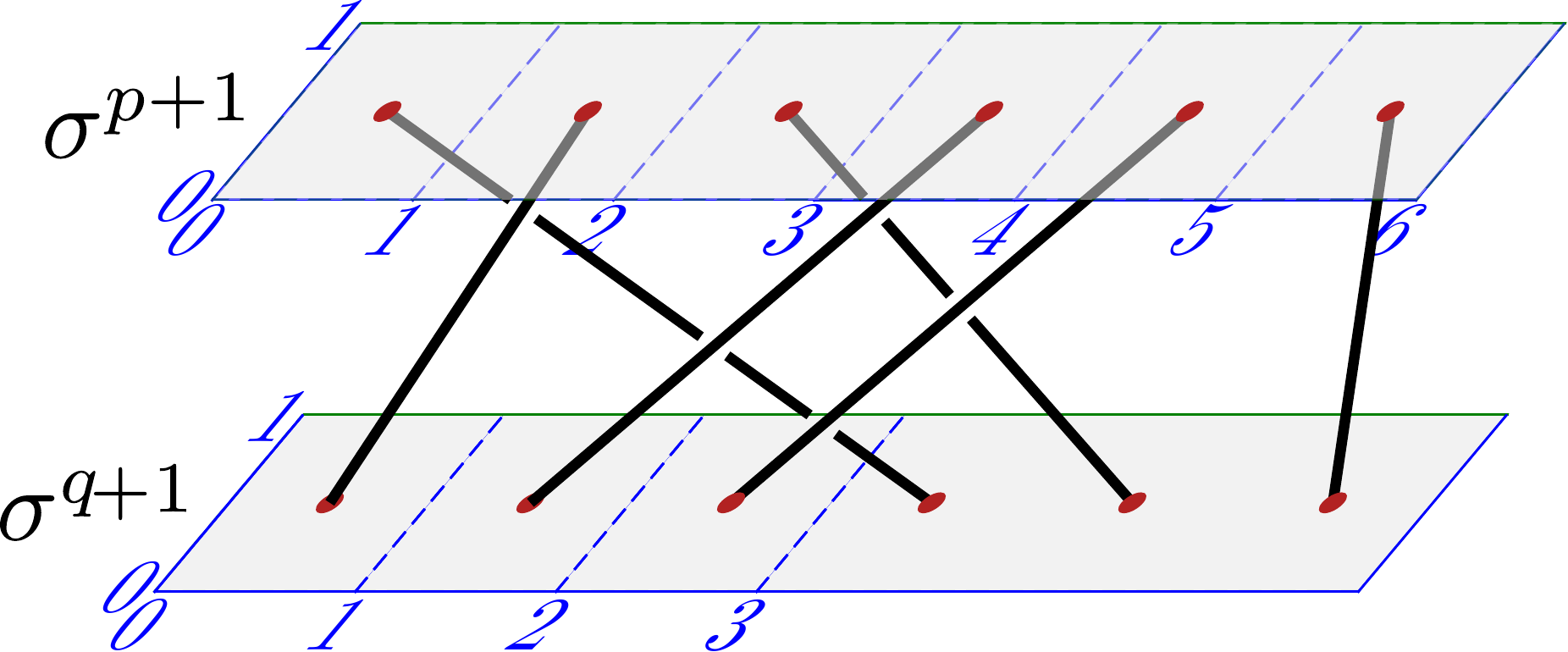} \hfill
\includegraphics[scale=0.23]{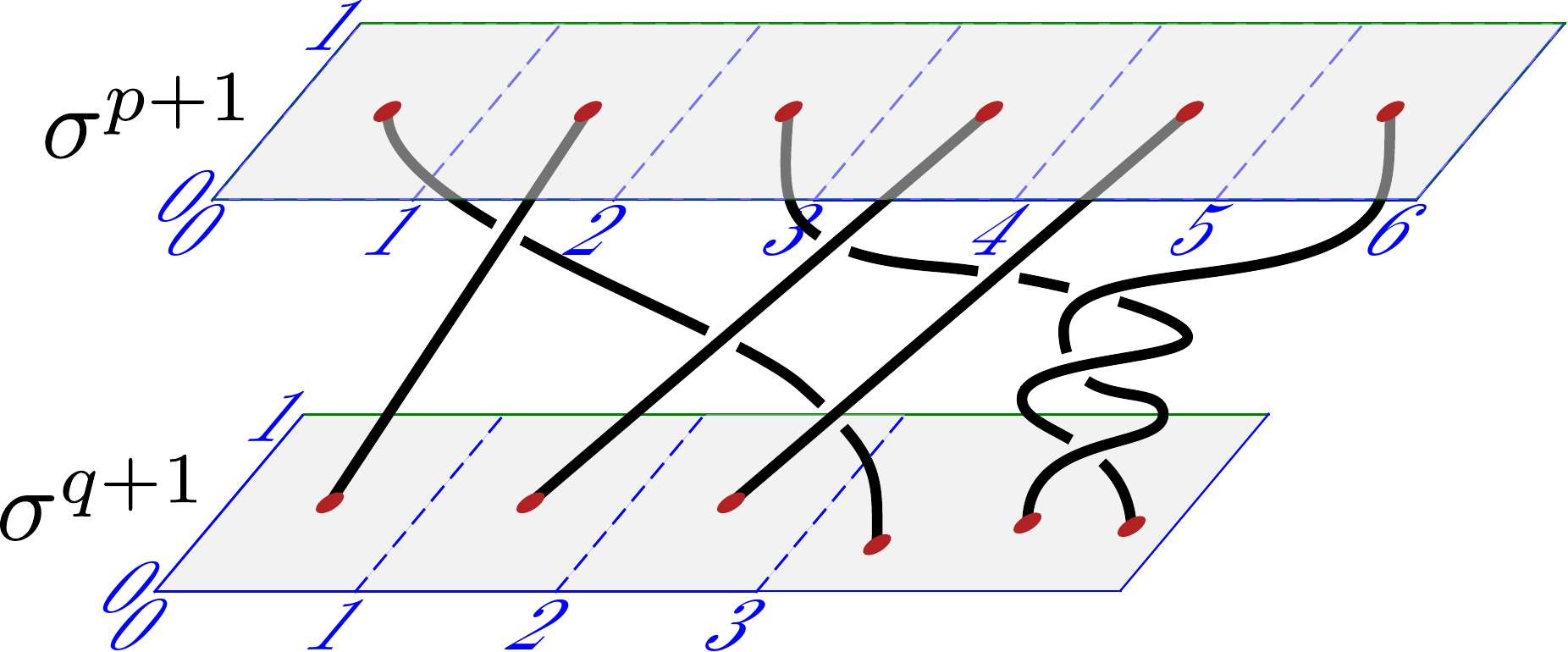}
    \caption{Two morphisms in the path component associated to the monotone injection
    \\
    $\iota \colon [2]  \longrightarrow [5]$
    \\
    $ \quad 0  \longmapsto 1 $
    \\
    $ \quad  1  \longmapsto 3 $
    \\
     $  \quad  2  \longmapsto 4 $ 
}
    \label{DeltaTwiddleExample}
\end{figure}

    Observe that the collection of such morphisms (for all $q, p,$ and $\iota$) is closed under composition.

    \begin{definition} \label{DefnDeltaTwiddleInj}
    Let $\widetilde{\Delta}_{\text{inj}}\subseteq U$ be the subcategory of $U$ with the same objects as $U$ and with morphisms $\widetilde{\Delta}_{\text{inj}}([q], [p])$ consisting of all path components of $U([q], [p])$ corresponding to the monotone injections ${\Delta}_{\text{inj}}([q], [p])$ in the sense of Definition \ref{DefnLiftMonotone}. 
    \end{definition} 
    We observe that the map   $$\widetilde{\Delta}_{\text{inj}} \to  {\Delta}_{\text{inj}}$$ defined (for each $p,q$) by the maps 
$$\widetilde{\Delta}_{\text{inj}}([q], [p])  \to \pi_0 (\widetilde{\Delta}_{\text{inj}}([q], [p])) \xrightarrow{\cong}  {\Delta}_{\text{inj}}([q], [p])$$ is functorial. 
    The morphisms of $\widetilde{\Delta}_{\text{inj}}$ are generated by the path components of $U$ corresponding to the face maps $d_i \colon [p-1] \to [p]$, as illustrated in Figure \ref{DeltaTwiddleExample-FaceMaps}.
%\end{definition}
              \begin{figure}[ht!]
    \centering
\includegraphics[scale=0.2]{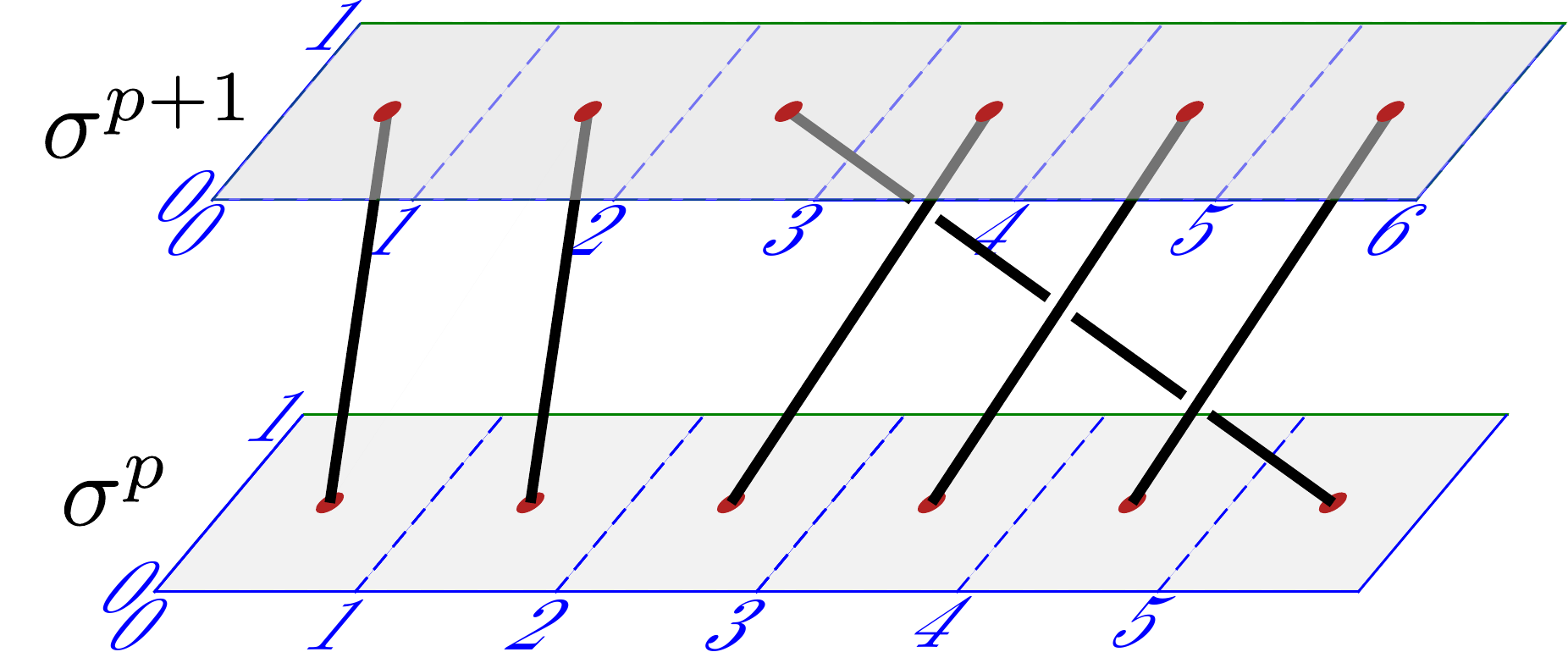}
    \caption{A morphism associated to the face map $d_2 \colon [4]  \to [5]$.}
    \label{DeltaTwiddleExample-FaceMaps}
\end{figure}  

The category $\widetilde{\Delta}_{\text{inj}}$ is a topological version of ${\Delta}_{\text{inj}}$. The functor 
$\widetilde{\Delta}_{\text{inj}} \to  {\Delta}_{\text{inj}}$
is a homotopy equivalence (see Krannich \cite[Lemma 2.11]{MR4019896} and Randal-Williams \cite[Page 3]{MR4767884}), that is, the path
components of $\widetilde{\Delta}_{\text{inj}}([q], [p])$ are contractible and in bijection with ${\Delta}_{\text{inj}}([q], [p])$.

We note that Randal-Williams \cite{MR4767884} studies right modules instead of left ones. For this reason, some of our definitions differ slightly from the definitions given there, but these distinctions do not make a material difference.  

Randal-Williams' category was defined to construct resolutions of modules over configuration spaces. Now we will describe a category adapted to construct resolutions of modules over Hurwitz spaces.

\begin{definition}
    Given $w\in c^{q+1}$ and $v\in c^{p+1}$, let $U_{c}(w, v)$ denote the space of pairs $(a, \mu)$, where $a\in \hur[, p-q]$ and a Moore path $\mu$ in $\hur[, p+1]$ from $i(v)$ to $  i(w)\cdot a$. We define a composition law on $U_{c}$ via the same formula as that on $U$ to give $U_{c}$ the structure of a topologically enriched category.
\end{definition}

A morphism in $U_c$ is shown in Figure \ref{FigureUcMorphismExample}. 

         \begin{figure}[ht!]
    \centering
      \includegraphics[scale=0.25]{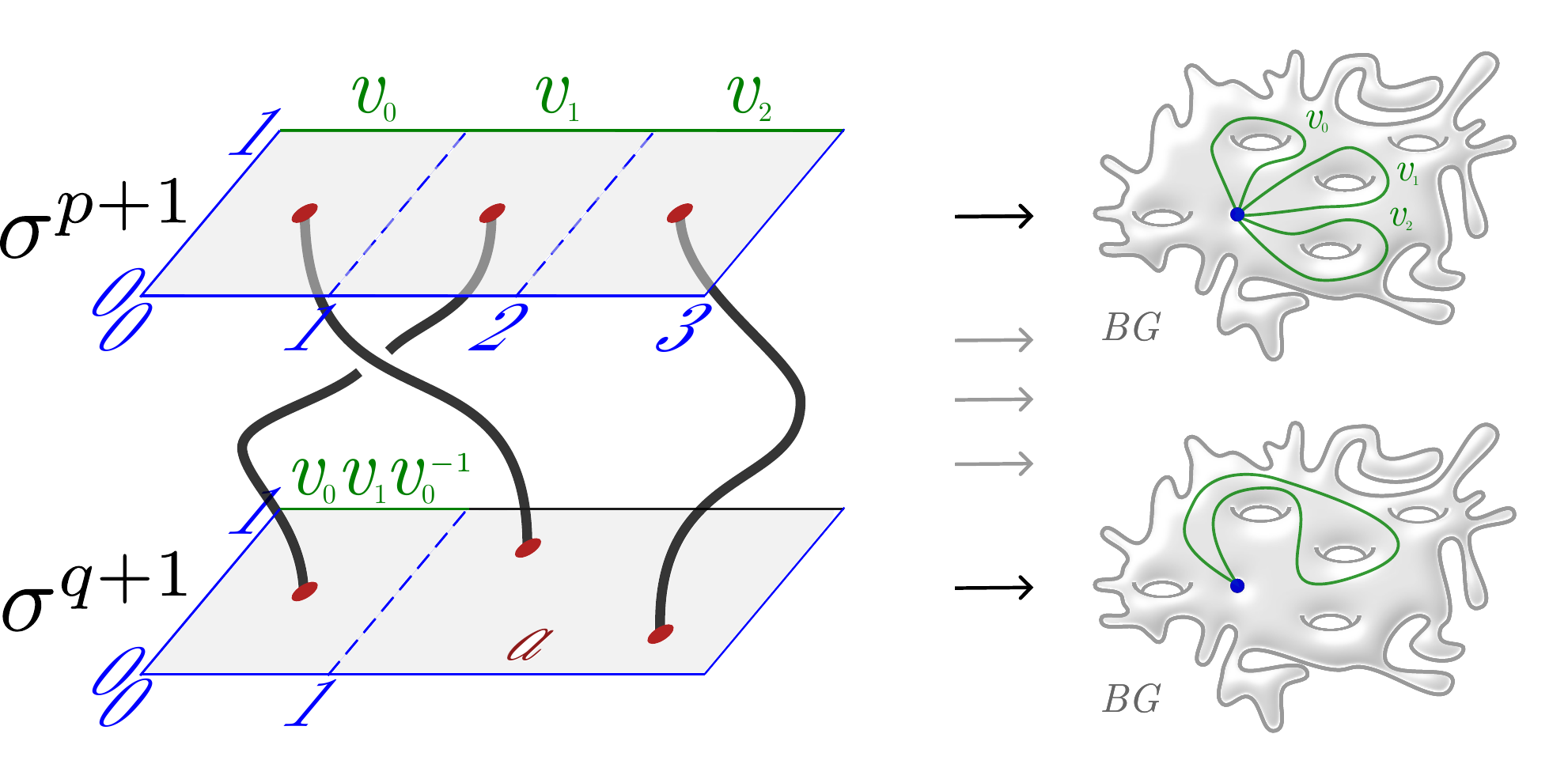} 
    \caption{A morphism in $U_c(w,v)$ where $v= (v_0, v_1, v_2)$ and $w = (v_0v_1v_0^{-1})$.}
    \label{FigureUcMorphismExample}
\end{figure}

\begin{definition}  Let $w\in c^{q+1}$ and $v\in c^{p+1}$. There is a forgetful map $U_c(w,v) \to U([q],[p])$. Define $\delc(w, v)\subseteq U_{c}(w, v)$ to be the pullback of the diagram of spaces
% https://q.uiver.app/#q=WzAsNCxbMCwwLCJcXGRlbGModywgdikiXSxbMSwwLCJVX2Modyx2KSJdLFswLDEsIlxcd2lkZXRpbGRlXFxEZWx0YV97XFx0ZXh0e2luan19KFtxXSwgW3BdKSJdLFsxLDEsIlUoW3FdLCBbcF0pIl0sWzAsMV0sWzEsM10sWzAsMl0sWzIsM11d
\[\begin{tikzcd}
	{\delc(w, v)} & {U_c(w,v)} \\
	{\widetilde\Delta_{\text{inj}}([q], [p])} & {U([q], [p]).}
	\arrow[from=1-1, to=1-2]
	\arrow[from=1-1, to=2-1]
	\arrow[from=1-2, to=2-2]
	\arrow[from=2-1, to=2-2]
\end{tikzcd}\]

Let $\delc\subseteq U_{c}$ be the subcategory of $U_{c}$ with the same objects as $U_{c}$ and the morphisms $\delc(w, v)$ as defined above. An example is shown in Figure \ref{UcMorphismDeltaExample}. 

         \begin{figure}[ht!]
    \centering
      \includegraphics[scale=0.25]{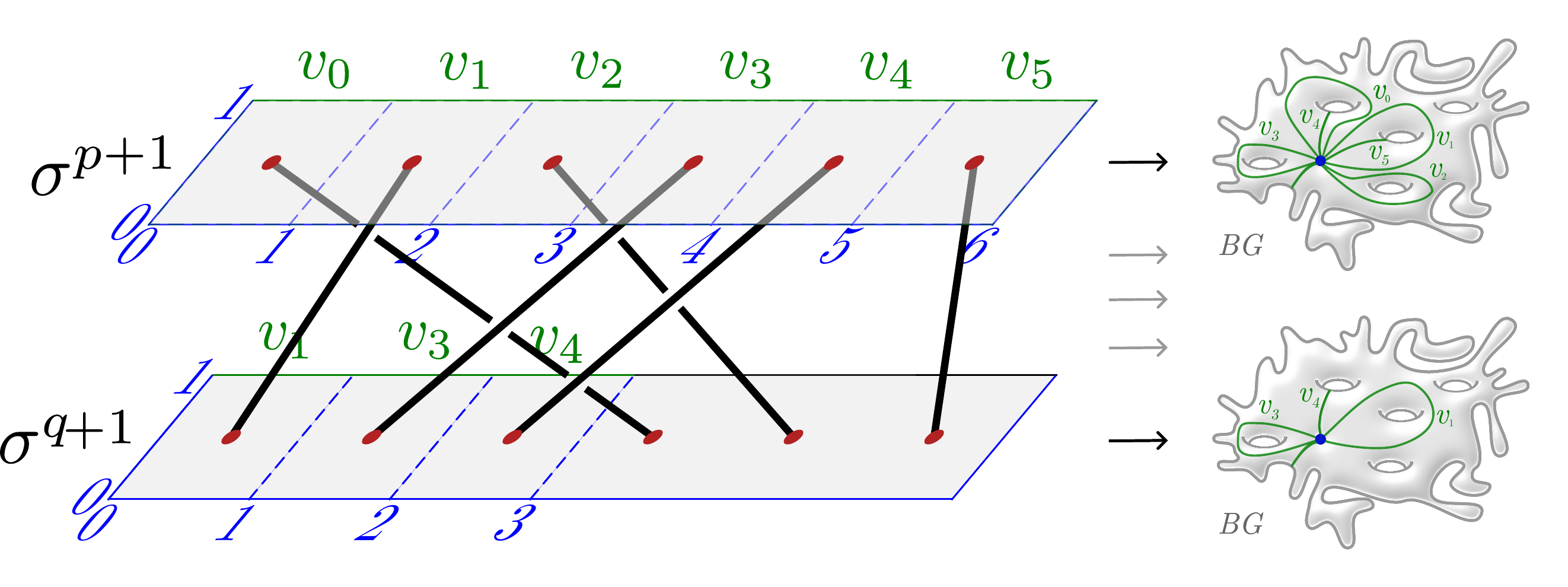} 
    \caption{A morphism in $\delc(w, v)$ where $v$ is $(v_0, v_1, v_2, v_3, v_4, v_5)$ and $w$ is $(v_1, v_3, v_4)$.}
    \label{UcMorphismDeltaExample}
\end{figure}

    %Suppose that $v\in c^{p+1}$ and $\iota\in \Delta_{\text{inj}}([q], [p])$. There is a Moore loop $\tilde{\mu}_{\iota}$ starting at $i(v)$ in $\hur[, p+1]$ obtained by picking a lift of the Moore loop $\mu_{\iota}$ from \cref{def:tildedel} to $\hur$. 
    %which corresponds to the braid on $(p+1)$ strands where the first $q$ strands of $p$ go behind the rest to end at $i([q])\subset [p]$ as in Figure\footnote{include figure}.
    %By modifying the monodromy map associated to $\tilde{\mu}_{\iota}$ if necessary, we have  $i(v)\cdot\tilde{\mu}_{\iota} = i(w)\cdot i(u) $ for a unique $u\in c^{p-q}$ and a unique $w\in c^{q+1}$ that is independent of our choice of $\mu_{\iota}$ or $\tilde{\mu}_{\iota}$.\footnote{The reason for adding the comment about possibly modifying the monodromy of $\tilde{\mu}_{\iota}$ is that the monodromy of the endpoint of $\tilde{\mu}_{\iota}$ might not agree with that of $i^{q+1}(w)\cdot i^{p-q}(u) $. I don't think this is an issue, since the monodromy of the endpoint of $\tilde{\mu}_{\iota}$ and $i^{q+1}(w) \cdot i^{p-q}(u) $ are homotopoic. I don't know how to express the first part of the sentence (``By...necessary'') more clearly.}
    %\footnote{need to change $\sigma^{p+1}(-)$ to $i(-)$.}
    %\footnote{fix stuff involving $g(\mu)$}
    
    %Define $\delc(w, v)\subset U_{c}(w, v)$ to be the path component of $U_{c}(w, v)$ containing the point $(i(u), \tilde{\mu}_{\iota})\in U_{c}(w, v)$. Let $\delc$ be the subcategory of $U_{c}$ with the same objects as $U_{c}$ and consisting of such path components.
\end{definition}
We now define a Hurwitz analog of ${\Delta}_{\text{inj}}$. 

\begin{definition} Define a category ${\Delta}_{\text{inj},c}$ as follows. The objects are the objects of $\delc$. For $w\in c^{q+1}$ and $v\in c^{p+1}$, the morphisms ${\Delta}_{\text{inj},c}(w,v)$ are the maps
$$ \{ f \colon [q] \to [p] \in {\Delta}_{\text{inj}} \; | \;  v_{f(k)} = w_{k}\text{ for all $k$} \}.$$
\end{definition}
 \begin{definition}\label{map: forgetful map to conf}
      Let $\phi\colon \hur[,n]\to\conf[n]$ denote the map that forgets the data of the map to $BG$ on the complement of the configuration.
   
    %$$\phi\colon \ohur[,S]\xrightarrow{\pi}\oconf[S]\to \conf[|S|],$$ where the map $\oconf[S]\to \conf[|S|]$ sends  $(\lambda, \xi)$ to $(\lambda, \im(\xi))$.  
    \end{definition}
    The map $\phi$ is a fibration.

\begin{definition}\label{functor to del}
    Let $$\phi_{\widetilde{\Delta}}\colon \delc\to \widetilde{\Delta}_{\text{inj}}$$ denote the enriched covariant functor of topological categories sending an object $v\in c^{p+1}$ in $\delc$ to $[p]$. Given $w\in c^{q+1}$, the map 
    $$\phi_{\widetilde{\Delta}} \colon \delc(w, v)\to \widetilde{\Delta}_{\text{inj}}([q], [p])$$
    is induced from the map $\phi\colon\hur\to \conf[]$.
\end{definition}

%\begin{definition} 
We will show that there is an enriched functor
$$ \Upsilon \colon  \widetilde{\Delta}_{\mathrm{inj},c}  \to {\Delta}_{\mathrm{inj},c} $$ 
which is the identity on objects, and defined on morphisms $\widetilde{\Delta}_{\mathrm{inj},c}(w,v)$ so as to make the following diagram of topological spaces commute. 

% https://q.uiver.app/#q=WzAsNixbMSwwLCJcXHdpZGV0aWxkZXtcXERlbHRhfV97XFxtYXRocm17aW5qfSwgY30odyx2KSJdLFsyLDAsIlVfYyh3LHYpIl0sWzEsMSwiXFx3aWRldGlsZGV7XFxEZWx0YX1fe1xcbWF0aHJte2luan19KFtxXSwgW3BdKSJdLFsyLDEsIlUoW3FdLCBbcF0pIl0sWzEsMiwie1xcRGVsdGF9X3tcXG1hdGhybXtpbmp9fShbcV0sIFtwXSkiXSxbMCwyLCJ7XFxEZWx0YX1fe1xcbWF0aHJte2luan0sIGN9KHcsdikiXSxbMCwyXSxbMSwzXSxbMCwxLCIiLDEseyJzdHlsZSI6eyJ0YWlsIjp7Im5hbWUiOiJob29rIiwic2lkZSI6InRvcCJ9fX1dLFsyLDMsIiIsMSx7InN0eWxlIjp7InRhaWwiOnsibmFtZSI6Imhvb2siLCJzaWRlIjoidG9wIn19fV0sWzIsNCwiIiwxLHsic3R5bGUiOnsiaGVhZCI6eyJuYW1lIjoiZXBpIn19fV0sWzMsNCwiIiwxLHsic3R5bGUiOnsiaGVhZCI6eyJuYW1lIjoiZXBpIn19fV0sWzUsNCwiIiwxLHsic3R5bGUiOnsidGFpbCI6eyJuYW1lIjoiaG9vayIsInNpZGUiOiJ0b3AifX19XSxbMCw1LCJcXHNpbWVxIiwyLHsiY3VydmUiOjEsInN0eWxlIjp7ImhlYWQiOnsibmFtZSI6ImVwaSJ9fX1dXQ==
\begin{equation} \label{MorphismDiagram}
\begin{tikzcd}
	& {\widetilde{\Delta}_{\mathrm{inj}, c}(w,v)} & {U_c(w,v)}      \\
	& {\widetilde{\Delta}_{\mathrm{inj}}([q], [p])} & {U([q], [p])} \\
	{{\Delta}_{\mathrm{inj}, c}(w,v)} & {{\Delta}_{\mathrm{inj}}([q], [p])}
        \arrow["{\ulcorner}"{anchor=center, pos=0.2}, phantom, from=1-2, to=2-3]
	\arrow[hook, from=1-2, to=1-3]
	\arrow["\phi_{\widetilde{\Delta}}", from=1-2, to=2-2]
	\arrow["{\Upsilon(w,v)}"', bend right=15, two heads, from=1-2, to=3-1]
	\arrow[from=1-3, to=2-3]
	\arrow[hook, from=2-2, to=2-3]
	\arrow[two heads, "\substack{\text{iso on } \\ \pi_0}"', from=2-2, to=3-2]
	\arrow[two heads, from=2-3, to=3-2]
	\arrow[hook, from=3-1, to=3-2]
    \end{tikzcd} 
\end{equation} 
%\end{definition}
Proposition \ref{lem: delc is disc} below shows that, for all $(w,v)$ the set of morphisms
${\Delta}_{\text{inj},c}(w,v)$ is equal to the image of $\delc(w,v)$ in $\Delta_{\mathrm{inj}}([q], [p])$ in Diagram (\ref{MorphismDiagram}). This shows that $\Upsilon(w,v)$ is well-defined and surjective.  Proposition \ref{lem: delc is disc} further shows that the fibres of $\Upsilon(w,v)$ are contractible, hence $\Upsilon(w,v)$ is a homotopy equivalence for all $(w,v)$.

%The definition of the map $$ \Upsilon(w,v) \colon \widetilde{\Delta}_{\mathrm{inj},c}(w,v) \to {\Delta}_{\mathrm{inj},c}(w,v)$$ is, concretely, as follows. Consider  $(a, \mu) \in \widetilde{\Delta}_{\mathrm{inj},c}(w,v)$, i.e.,  $a\in \hur[, p-q]$ and a Moore path $\mu$ in $\hur[, p+1]$ from $i(v)$ to $  i(w)\cdot a$. Then $\Upsilon(w,v)$ maps $(a, \mu)$ to the underlying injective map of sets $[q] \to [p]$ defined by braid associated to the (time-inverse $\overline{\mu}$ of the) Moore path $\mu$. 

%The functor $\Upsilon$ sends $v\in c^{p+1}$ to $[p]$ and 

\begin{proposition}\label{lem: delc is disc} There is an enriched functor
$$ \Upsilon \colon  \widetilde{\Delta}_{\mathrm{inj},c}  \to {\Delta}_{\mathrm{inj},c} $$ 
such that for each $w\in c^{q+1}$ and $v\in c^{p+1}$,  the continuous map  $$ \Upsilon(w,v) \colon \widetilde{\Delta}_{\mathrm{inj},c}(w,v) \to {\Delta}_{\mathrm{inj},c}(w,v)$$ makes Diagram (\ref{MorphismDiagram}) commute. 
The functor $ \Upsilon$ is the identity map on objects, Moreover, for all objects $w,v$, the map $ \Upsilon(w,v)$ is a homotopy equivalence. 
\end{proposition}

\begin{proof}
%[Proof of \cref{lem: delc is disc}]
Our first step is to argue the set of morphisms
$${\Delta}_{\text{inj},c}(w,v) := \{ f \colon [q] \to [p] \in {\Delta}_{\text{inj}} \; | \;  v_{f(i)} = w_{i}\text{ for all $i$} \}$$ is equal to the image of $\delc(w,v)$ in $\Delta_{\mathrm{inj}}([q], [p])$ in Diagram (\ref{MorphismDiagram}). This implies that $\Upsilon(w,v)$ is well-defined and surjective. Our second step is to argue that the fibers of the map $\Upsilon(w,v)$ are contractible. 

\,

\noindent {\bf Step 1: $\Upsilon(w,v)$ is well-defined and surjective}

Consider a morphism $(a,\mu) \in \widetilde{\Delta}_{\text{inj},c}(w,v)$. By definition, $a \in \hur[,p-q]$ and $\mu$ is  a Moore path in $\hur[, p+1]$ from $i(v)$ to $  i(w)\cdot a$. %Note that $a$ is uniquely determined by $\mu$.

Let $f:[p] \to [q]$ be the map in $\Delta_{\mathrm{inj}}([q], [p])$ corresponding to the path component containing $(a', \gamma)\colonequals \phi_{\widetilde{\Delta}}(a, \mu))\in \widetilde{\Delta}([q], [p])$. 
%Suppose that $(a, \mu)$ maps in $\widetilde{\Delta}_{\mathrm{inj}}([q], [p])$ to the path component corresponding to the injective map $f:[p] \to [q]$. 
We will verify that $f$ satisfies $v_{f(k)} = w_k$ for all $k$. To be concrete, we illustrate the argument in the case of the particular map 
\begin{align*} 
f: \{0,1,2\} & \longrightarrow \{0,1,2,3,4,5\} \\ 
0 & \longmapsto 1 \\ 
1 & \longmapsto 3 \\ 
2 & \longmapsto 4 
\end{align*} 
A Moore path $\gamma$ of configuration spaces corresponding to this sample map $f$ is shown in Figure \ref{DeltaTwiddleExample-SampleBraid}.
       \begin{figure}[ht!]
    \centering
      \includegraphics[scale=0.19]{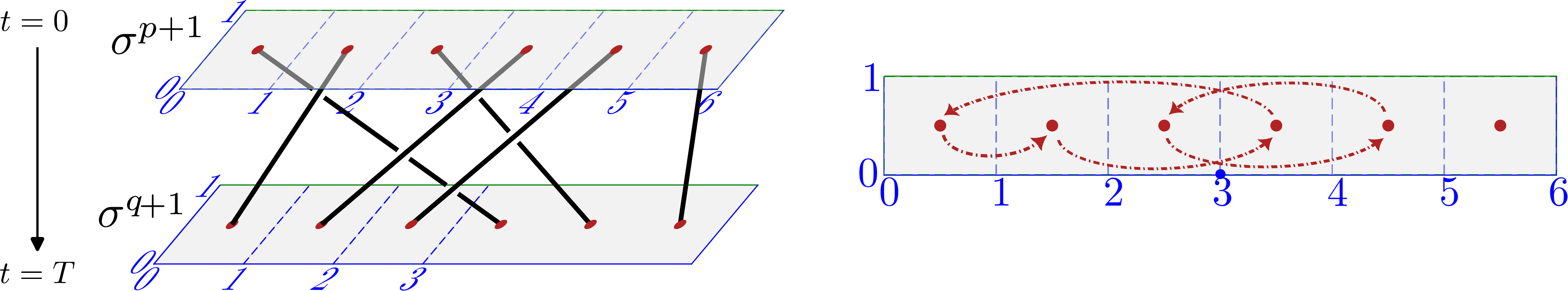} 
    \caption{Two perspectives on (the path-homotopy class of) a Moore path $\gamma$ with Moore parameter $T$ lifting our sample injective map $f: [2] \to [5]$. The righthand image shows the motion of the configuration during the reverse path $\overline{\gamma}$ as the particles in positions $0,1,2$ move to positions $f(0), f(1), f(2)$, respectively.}
    \label{DeltaTwiddleExample-SampleBraid}
\end{figure} 

%Consider the map $(a_{0}, \mu_{0})\colonequals \phi_{\widetilde{\Delta}}(a, \mu))\in \widetilde{\Delta}([q], [p])$ given by forgetting the map to $BG$ and let $f\colon [q]\to [p]$ be the injection corresponding to the path component of $\widetilde{\Delta}([q], [p])$ containing $(a_{0}, \mu_{0})$. 

Our construction---\cref{DefnDeltaTwiddleInj}---of $\widetilde{\Delta}_{\mathrm{inj}}([q], [p])$, determines (up to path isotopy) the reverse Moore path $\bar{\gamma}$ from $\sigma^{q+1}\cdot a'$ to $\sigma^{p+1}$. For each $k=0, 1, \dots, q$, the $k$-th strand passes from the $k$-th point of $\sigma^{q+1}$ to the $f(k)$-th point of $\sigma^{p+1}$. The key feature, which we now illustrate, is that because (up to path isotopy) these $(q+1)$ strands pass \emph{in front} of the complementary $(p-q)$ strands and \emph{do not cross} one another, the label $w_k$ of the $k$-th strand of $\mu$ at time $t=T$ must agree with its label $v_{f(k)}$ at time $t=0$.

%As a result, since the strands of $\mu$ starting at $i(w)$ pass in front of the complementary $p-q$ strands and end at the points in $i(v)$ determined by $f$, we have $v_{f(i)}= w_{i}$ for all $i\in [q]$.

Recall that, by our convention established in \cref{Defni}, the tuple  $v \in c^{p+1}$  is defined by $\mu(0) \in \hur[p+1]$ as follows. We have chosen a distinguished set of generators of the fundamental group of the complement $D_{p}$ of the configuration $\sigma^{p+1}$; representatives of these generators are shown in the case of our running example in Figure \ref{DeltaTwiddleExample-DistinguishedPi1Generators}. 
       \begin{figure}[ht!]
    \centering
      \includegraphics[scale=0.3]{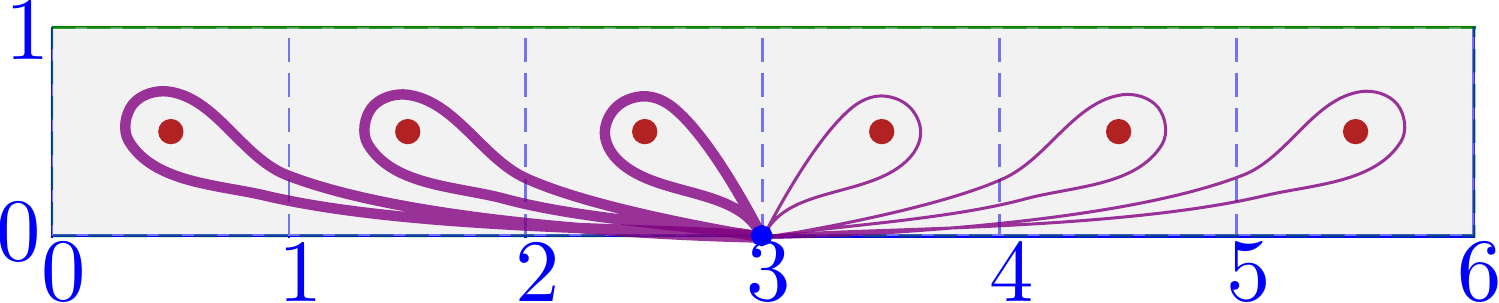} 
    \caption{Our choice of basepoint and distinguished generators of the fundamental group of $D_{5}$. The generators for $D_{2}\subseteq D_{5}$ are bolded.}
    \label{DeltaTwiddleExample-DistinguishedPi1Generators}
\end{figure} 
The entries of the tuple $v \in c^{p+1}$ correspond to the images of these distinguished generators in $G=\pi_1(BG, *)$ as determined by $\mu(0) \in \hur[, p+1]$. Similarly, the tuple $w$ encodes the images of the distinguished generators for the fundamental group of $D_{q}$ under the map specified by $\mu(T) \in \hur[, p+1]$.  

 Figure \ref{DeltaTwiddleExample-BraidInducedMap} shows the map induced on the fundamental group of $D_p$ in the case of our specific example. Observe that because the particles $k=1, 2, \dots, q+1$ only cross in front of other particles---the side closer to the basepoint---the braid $\overline{\gamma}$ maps the $k$-th distinguished generator of the fundamental group of $D_p$ to the $f(k)$-th distinguished generator. 
 Consequently, the path $\mu$ satisfies $v_{f(k)} = w_k$ for $k \in [q]$. 
       \begin{figure}[ht!]
    \centering
      \includegraphics[scale=0.25]{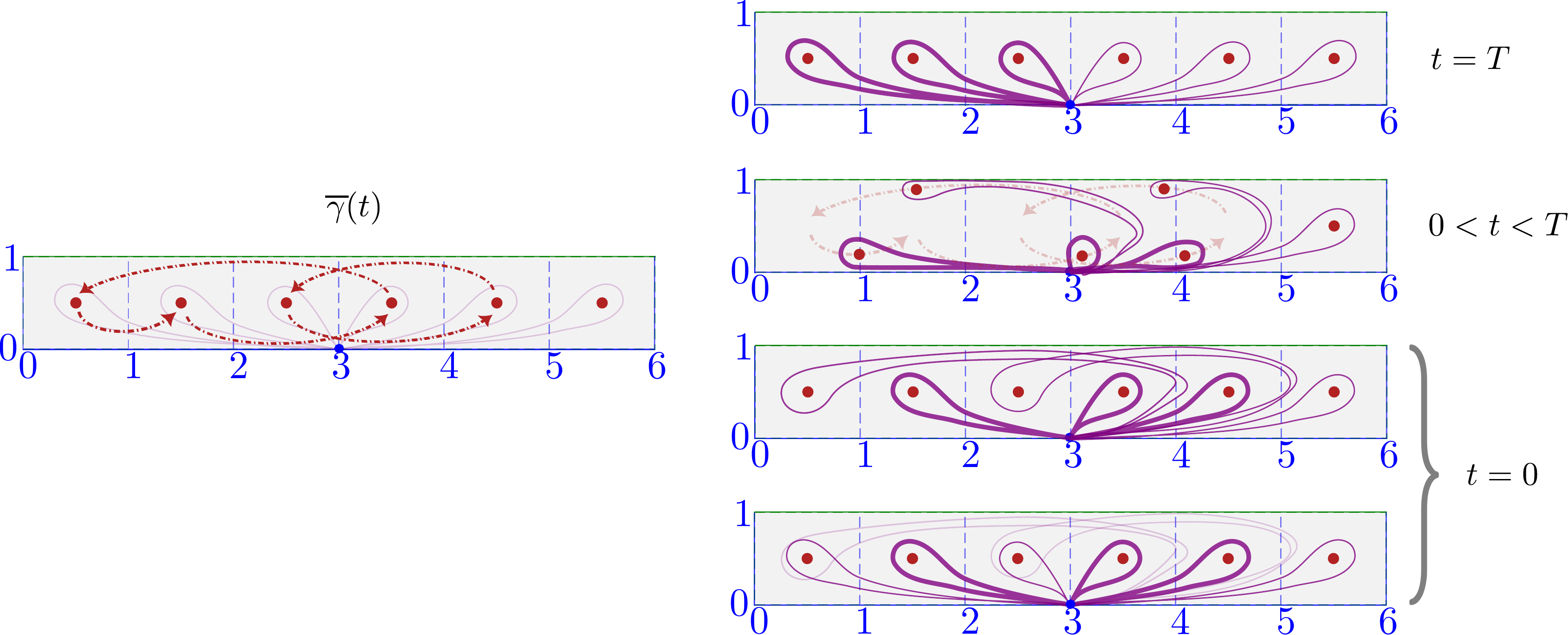} 
    \caption{The map induced by $\overline{\gamma}$ on the fundamental group of the complement of the configurations. The generators corresponding to the configuration $\sigma^{q+1} \subseteq \sigma^{p+1}$ are bolded.}
    \label{DeltaTwiddleExample-BraidInducedMap}
\end{figure} 
We have shown that the morphism $f$ is an element of ${\Delta}_{\mathrm{inj, c}}(w,v) \subseteq \Delta_{\mathrm{inj}}([q],[p])$ as claimed. 

Conversely, let $g \in \Delta_{\mathrm{inj, c}}(w, v)$ be any morphism. Let $\eta$ be the Moore path of a morphism in $\widetilde{\Delta}_{\mathrm{inj}}([q],[p])$ corresponding to $g$. Lift $\eta$ to a path in $\hur[, p+1]$ starting at the element $i(v)$ of \cref{Defni}.  Then necessarily $w_k  = v_{g(k)}$ for all $k$, and this lift is (up to homotopy) a morphism of $\delc(w,v)$ that maps to $g$. We have therefore verified that the image of $\delc(w,v)$ in $\Delta_{\mathrm{inj}}([q], [p])$ in Diagram (\ref{MorphismDiagram}) is precisely $\Delta_{\mathrm{inj, c}}(w, v)$. 

\,

\noindent {\bf Step 2: The fibers of $\Upsilon$ are contractible}

To complete the proof that $\Upsilon(w,v)$ is a homotopy equivalence, it remains to show its fibers are contractible. Note that, since the map $U_{c}(w, v)\to U([q], [p])$ is a fibration and the pullback along a fibration is also a fibration, the map
$$\phi_{\widetilde{\Delta}}\colon \delc(w, v)\to \widetilde{\Delta}_{\text{inj}}([q], [p])$$
%\begin{align}\label{map from delc to del}
%   \delc(w, v)\to \widetilde{Delta}_{\text{inj}}([q], [p]) 
%\end{align}
is a fibration. Since $\phi_{\widetilde{\Delta}}$ is a fibration,  the homotopy fiber over any  element in $\widetilde{\Delta}_{\text{inj}}([q], [p])$ is weakly homotopy equivalent to the literal fiber over that point. As a result, since $\widetilde{\Delta}_{\text{inj}}([q], [p])$ is homotopy discrete, to show that $\delc(w, v)$ is homotopy discrete,   it suffices to show that for any $(b, \gamma)\in \widetilde{\Delta}_{\text{inj}}([q], [p])$ in the image of $\phi_{\widetilde{\Delta}}$, the fiber  over $(b, \gamma)$ is contractible. 

Fix $(b, \gamma)\in \widetilde{\Delta}_{\text{inj}}([q], [p])$ in the image of $\phi_{\widetilde{\Delta}}$ and let $F \subset \delc(w, v)$ be the fiber. A point $(a,\mu)$ in $F$ is a Moore path $\mu$ in $ \hur[,p+1]$ and an element $a \in \hur[,p-q]$ such that the underlying configuration and Moore path of configurations are $(b,\gamma)$. Recall that $a$ is determined by $\mu$. Let $T \in [0,\infty)$ be the number associated to $\gamma$ that governs when $\gamma$ is required to be constant.

 Let  $$C = \bigcup_t \Big( \left( [0,\lambda(t) ] \times [0,1] \right) \setminus \xi(t) \Big) \times \{t\}  \subset \R^3 \qquad \text{for } t \in [0, \infty)$$ 
     and $H \subseteq C $ be  
     $$ H = \bigcup_{t} \left({\textstyle \boldsymbol{\bigsqcup}_{\lambda(t)}} \times \{t\} \right) \subseteq \R^3 \qquad \text{for } t \in [0, \infty).$$ 
%There is a map $$\pi_0 \left (\mathrm{Map}\left( \left( D_p,{\textstyle \boldsymbol{\bigsqcup}_{p+1}}\right), (BG, pt)\right) \right) \to G^{p+1}$$ given by restricting maps to $\{  \}$
     
Observe that $F$ is homeomorphic to the subspace of $$\mathrm{Map}\left( ( C,H), (BG, \text{pt})\right)$$ of maps $f$ satisfying the following conditions: 
\begin{enumerate}
    \item \label{f=v}  $f_0 = g_v$,
    \item \label{f=w} $f_t$ restricted to $\{t\} \times [0,q+1] \times [0,1]$ is $g_w$ for all $t \geq T$.
    \end{enumerate}
Consider the restriction maps $$\text{res}_1: \mathrm{Map}\left( ( C,H), (BG, \text{pt})\right) \to \mathrm{Map}\left(  \left( D_p,{\textstyle \boldsymbol{\bigsqcup}_{p+1}}\right), (BG, \text{pt})\right), $$
$$\text{res}_2: \mathrm{Map}\left( ( C,H), (BG, \text{pt})\right) \to \mathrm{Map}\left( \left( D_q,{\textstyle \boldsymbol{\bigsqcup}_{q+1}}\right), (BG, \text{pt})\right) $$ where the first map evaluates at $0$ and the second map evaluates at $T$.

Condition (\ref{f=v}) is the statement $\text{res}_{1}(f)=g_v$ and Condition (\ref{f=w}) is the statement that $\text{res}_{2}(f)=g_w$. The maps $\text{res}_{1}$, $\text{res}_{2}$, and $\text{res}_{1} \times \text{res}_{2}$ are adjoint to CW inclusions. Since CW inclusions are cofibrations and the mapping space functor takes cofibrations to fibrations, we conclude that $\text{res}_{1}$, $\text{res}_{2}$, and $\text{res}_{1} \times \text{res}_{2}$ are fibrations. Observe that 
%In case it is needed: the spaces $D_p$, $D_q$, $C$ (and the boundary spaces), are all smooth manifolds with boundary. Such spaces are absolute neighborhood retracts. Hence, the relevant inclusions are all Hurewicz cofibrations. 
\begin{align*}
 &  \mathrm{Map}\left( (C,H), (BG, \text{pt})\right) \\ 
  \simeq \; & \mathrm{Map} \left(\left(\bigvee_{p+1} S^1, \text{pt} \right), (BG, \text{pt})\right)  \\
  \cong \; &(\Omega BG)^{p+1} \\
   \simeq \; & G^{p+1}.
\end{align*}
Similarly, $$\mathrm{Map}\left(  \left( D_p,{\textstyle \boldsymbol{\bigsqcup}_{p+1}}\right), (BG, \text{pt})\right) \simeq G^{p+1}$$ and $$\mathrm{Map}\left(  \left( D_q,{\textstyle \boldsymbol{\bigsqcup}_{q+1}}\right), (BG, \text{pt})\right) \simeq G^{q+1}.$$

Suppose $\text{res}_{1}(f) \simeq g_v$. By the discussion in {\bf Step 1}, it follows that $\text{res}_{2}(f) \simeq g_w$. This can be seen by considering how the map to $BG$ changes along the path $\gamma$; the $q+1$ strands connecting to $\sigma^{q+1}$ all  go in front of the remaining $p-q$ strands. See Figure \ref{Res1Res2Proof}.
         \begin{figure}[ht!]
    \centering
      \includegraphics[scale=0.25]{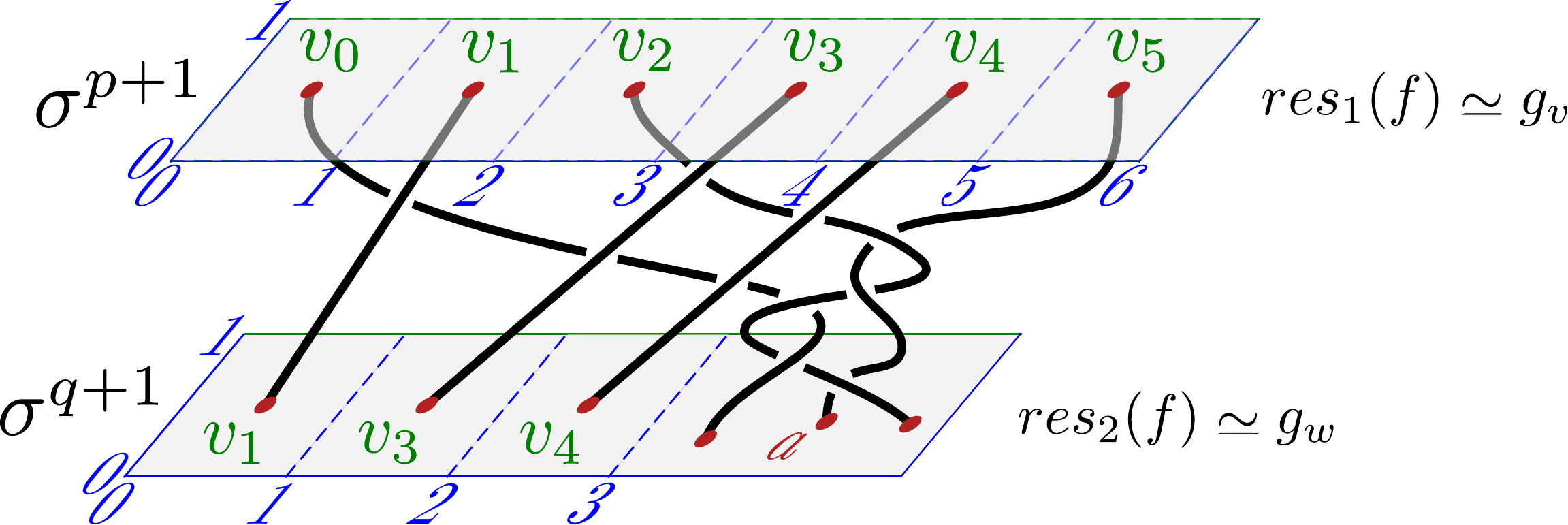} 
    \caption{}
    \label{Res1Res2Proof}
\end{figure}  

The map $\text{res}_{1}$ is a homotopy equivalence since it is adjoint to a homotopy equivalence. Hence its fibers are contractible. By the above paragraph, the image of $$\text{res}_{2}:\text{res}_{1}^{-1}(g_v)  \to \mathrm{Map}\left(  \left( D_q,{\textstyle \boldsymbol{\bigsqcup}_{q+1}}\right), (BG, \text{pt})\right)$$ is the path component corresponding to $w \in G^{q+1}$. This connected component is contractible. Thus, the preimage of $g_w$ under $$\text{res}_{2}:\text{res}_{1}^{-1}(g_v)  \to \mathrm{Map}\left(  \left( D_q,{\textstyle \boldsymbol{\bigsqcup}_{q+1}}\right), (BG, \text{pt})\right)$$ is contractible. Thus $$F \cong (\text{res}_{1} \times \text{res}_{2})^{-1}(g_v,g_w)$$ is contractible.
\end{proof}

%There is a functor $\delc\to \Delta_{\text{inj}}$ sending an object $v\in c^{p+1}$ to $[p]$ and a map $f\colon w\to v$ in $\delc(w, v)$, with $w\in c^{q+1}$, to $f\in \Delta([q], [p])$.

%For any $w\in c^{q+1}$ and $v\in c^{p+1}$, the forgetful
%\footnote{Should introduce the map $\phi$ earlier. It's possible the source should be $\hur$} 
%map $$\phi\colon\hur\to \conf[]$$ induces a map
%    $$\widetilde{\Delta}(\phi)(w, v)\colon \delc(w, v)\to \widetilde{\Delta}_{\text{inj}}([q], [p]).$$ As a result, we have an enriched covariant forgetful functor of topological categories, $$\widetilde{\Delta}(\phi)\colon \delc\to \widetilde{\Delta}_{\text{inj}}.$$ 
    %{\color{cyan} ZZZZZ Need different notation for $\phi$ ZZZZZZZ}    
    %The space $\delc(w, v)$ corresponds to the path component of $\widetilde{\Delta}_{\text{inj}}([q], [p])$ corresponding to $\iota\colon [q]\to [p]$. Since $\widetilde{\Delta}_{\text{inj}}([q], [p])$ is homotopy discrete
%(by the argument in \cite[Lemma 2.11]{MR4019896}) we obtain the following lemma:
\begin{lemma}\label{lem: delc is homotopy discrete}
    If $\delc(w, v)$ is non-empty, then it is contractible. In addition, for each $v\in c^{p+1}$, the forgetful map $\phi\colon\hur[,q+1]\to \conf[q+1]$ induces a homotopy equivalence $$\bigsqcup_{w\in c^{q+1}} \delc(w, v)\to \widetilde{\Delta}_{\text{inj}}([q], [p]).$$ 
\end{lemma}
\begin{proof}
    We have a commutative diagram
    \[\begin{tikzcd}
	{\bigsqcup_{w\in c^{q+1}}\delc(w, v)} & {\widetilde{\Delta}_{\text{inj}}([q], [p])} \\
	{\bigsqcup_{w\in c^{q+1}} \Delta_{\text{inj,c}}(w, v)} & {\Delta_{\text{inj}}([q], [p]).}
	\arrow[from=1-1, to=1-2]
	\arrow[from=1-1, to=2-1, "j_{1}"]
	\arrow[from=1-2, to=2-2, "j_{2}"]
	\arrow[from=2-1, to=2-2, "\psi"]
\end{tikzcd}\]
The maps $j_{1}, j_{2}$ are homotopy equivalence by \cref{lem: delc is disc} and the argument in \cite[Lemma 2.11]{MR4019896} respectively. Therefore, to show that the top horizontal map is a homotopy equivalence, it suffices to show that the map $\psi$ is a bijection. 

By definition,
$$ \Delta_{\text{inj,c}}(w,v)=\left\{f \colon [q] \to [p]  \; \middle| \;  f\text{ is order preserving and } v_{f(i)}=w_{i}   \right\} $$ 
Hence,
$$ \bigsqcup_{w\in c^{q+1}} \Delta_{\text{inj,c}}(w,v) =
 \left\{(w,f) \; \middle| \; f\in \Delta_{\text{inj}}([q], [p]), \; w \in c^{q+1}, \text{ and } v_{f(i)}=w_{i}   \right\}$$ 
Note that $w$ is determined by $f$ so the forgetful map induces a bijection between this disjoint union and the set
$$ \left\{f \; \middle| \;  f \colon [q] \to [p]  , \;  f \text{ is order preserving } \right\}
=\Delta_{\text{inj}}([q],[p]) $$ 
as claimed.
% Since $v$ is fixed, $\psi$ is injective.  Now we show that $\psi$ is surjective.Given $f\in \Delta_{\text{inj}}([q], [p])$, pick an element $(e, \mu)\in j^{-1}_{2}(f)$. Since the map $\hur[,p+1]\to \conf[p+1]$ is a fibration, there is a lift $\tilde{\mu}$ of $\mu$ to a Moore path in $\hur[,p+1]$ starting at $i(v)$ and ending at $i(w_{0})\cdot \tilde{e}$ for some $w_{0}\in c^{q+1}$ and $\tilde{e}\in \hur[,p-q]$. We have that $(\tilde{e}, \tilde{\mu})$ is in $\delc(w_{0}, v)$ and $j_{1}(\tilde{e}, \tilde{\mu})= f\in \Delta_{\text{inj}}(w_{0}, v)$, and so $\psi$ is surjective. 
\end{proof}

\section{A functorial resolution of modules over ordered Hurwitz space} \label{ResHurSec}
In this section, we  adapt ideas from Randal-Williams \cite[Section 5]{MR4767884} (see also Krannich \cite[Section 2.2]{MR4019896}) to construct a resolution $R_{\blacktriangle}(\bM)$ of a module $\bM$ over  ordered Hurwitz space. We show that the homotopy type of this resolution computes derived $\ohur$-module indecomposables. In fact, the resolution may be viewed as a Koszul resolution. We first recall the construction from Randal-Williams 
\cite{MR4767884} for resolving modules over configuration spaces.

\begin{definition} \label{DefnRConfModule}
    For a left $\conf[]$-module $\bM$, let $R_{p}(\bM)$ denote the object in $\topo^{\N}$ with $R_{p}(\bM)(n)$ consisting of the space of pairs $(b, \gamma)$, where $b$ is a point in $\bM(n-p-1)$ and $\gamma$ is a Moore path in $\bM(n)$ ending at $\sigma^{p+1} \cdot b$. An example is shown in Figure \ref{FigureR(N)PointExample}. 
    
    Consider $R_{\bullet}(\bM)$ as an enriched covariant functor $\widetilde{\Delta}_{\text{inj}}^{\text{op}}\to \text{Top}$ via the maps
    \begin{align*}
        \widetilde{\Delta}_{\text{inj}}([q], [p])\times R_{p}(\bM)&\to R_{q}(\bM)\\
        ((e, \mu), (b,\gamma))&\mapsto (e\cdot b, \gamma*(\mu\cdot b)).
    \end{align*}
   as illustrated in Figure \ref{FigureR(N)ActionExample}.

     The maps $\epsilon_{p}\colon R_{p}(\bM)\to \bM$ given by evaluating the Moore path at $0$ assemble to give an augmentation $\epsilon_{\bullet}\colon R_{\bullet}(\bM)\to \bM$.

\end{definition}
         \begin{figure}[ht!]
    \centering
      \includegraphics[scale=0.25]{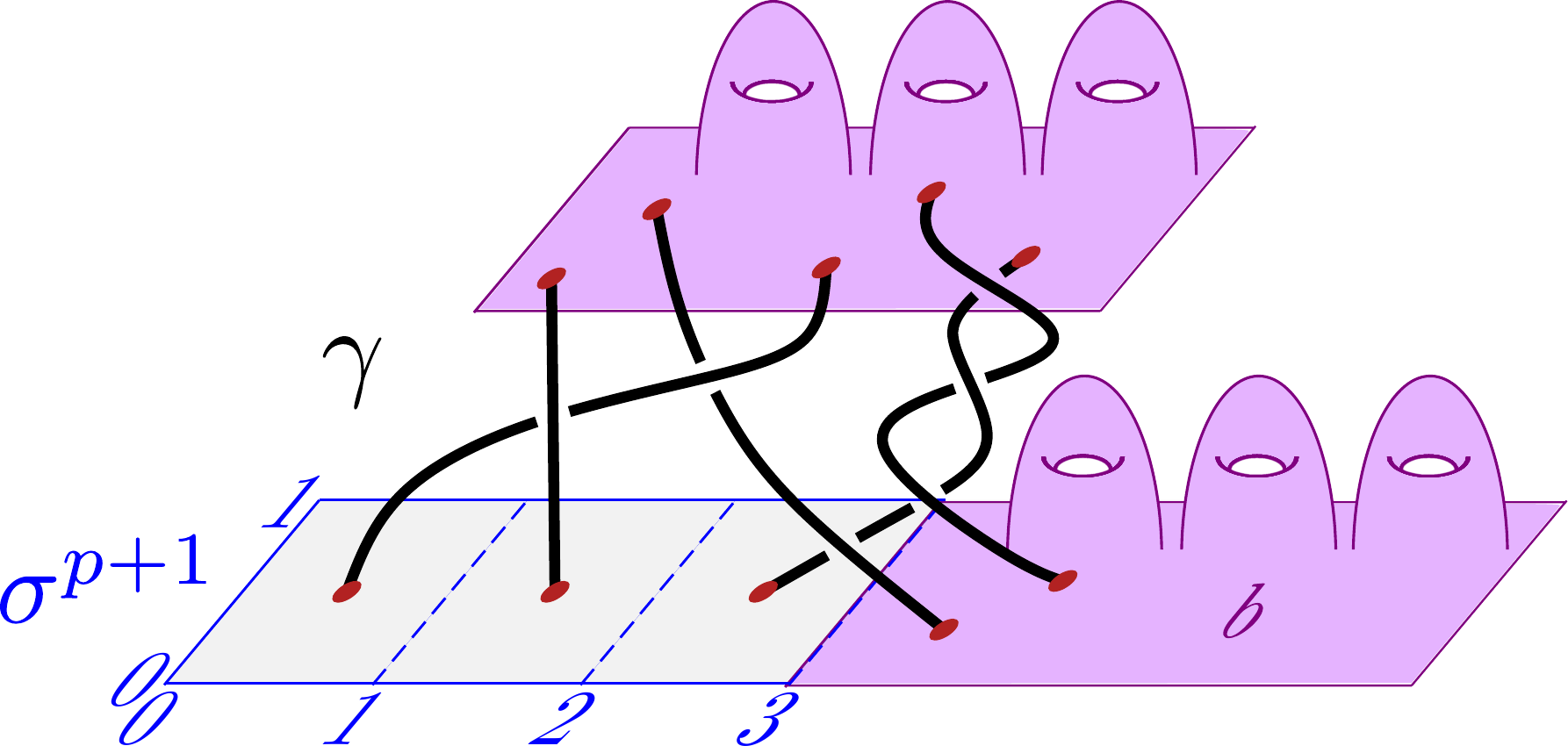} 
    \caption{An element $(b,\gamma)$ of $R_{p}(\bM)(n)$.}
    \label{FigureR(N)PointExample}
\end{figure}

         \begin{figure}[ht!]
    \centering
      \includegraphics[scale=0.25]{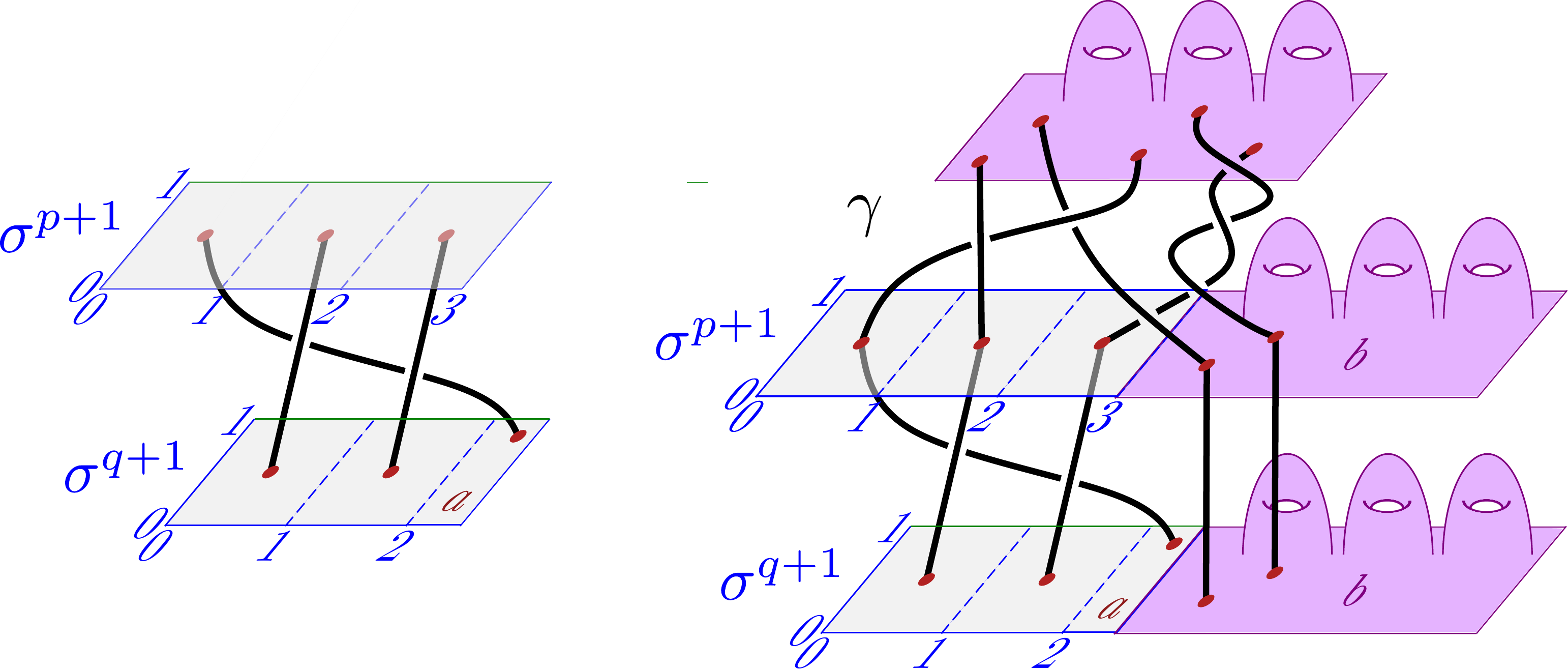} 
    \caption{An element $(a,\mu) \in  \widetilde{\Delta}_{\text{inj}}([q], [p])$ and its action on the element $(b,\gamma) \in R_{p}(\bM)(n)$ of Figure \ref{FigureR(N)PointExample}.}
    \label{FigureR(N)ActionExample}
\end{figure}

The following construction can be thought of as a Hurwitz version of cenetral stability homology as in \cite{MR3703435, MR4125676, MR4671766}.

\begin{definition}
    Let $\bM$ be an $\fb$-space. Given $p\geq -1$ and $v \in c^{p+1}$,  let $L_{v}(\bM)\colon \fb\to \topo$ be the functor with
    %which sends $S\in \fb$ to the space 
    $$L_{v}(\bM)(S)\colonequals \bigsqcup_{h\colon [p]\hookrightarrow S}\bM(S\setminus\im(h)).$$
    %of pairs $(g, b)$, where $g\colon [p]\hookrightarrow S$ and $b\in \bM(S\setminus\im(g))$. 
    We denote an element of $L_{v}(\bM)(S)$ by $(h, b)$ with $h\colon [p]\hookrightarrow S$ and $b\in \bM(S\setminus\im(h))$. Given a map $f\colon S_{0}\to S_{1}$ in $\fb$, the induced map $$L_{v}(\bM)(f)\colon L_{v}(\bM)(S_{0})\to L_{v}(\bM)(S_{1})$$ sends $(h, b)$ to $(f\circ h, (f|_{S_{0}\setminus \im(h)})_{*}(b))$.
\end{definition}

If $\bM$ is a module over $\pi_0(\ohur)$, then the FB-spaces $L_v(M)$ assemble to form an augmented co-$\Delta_{\text{inj,c}}$-FB-space, as follows. 

\begin{definition}  \label{LDelStructure}
    Let $\bM$ be a $\pi_0(\ohur)$-module and $S$ a finite set.  Let $g \in \Delta_{\text{inj,c}}(w,v)$, with $v\in c^{p+1}$ and $w\in c^{q+1}$, and let $(h, b)\in L_{v}(\bM)(S)$. Let $$\{r_{0}<\dots < r_{p-q-1}\}$$ denote $[p]\setminus  g([q])$. For each $k=0,\ldots, p-q-1$, let $u_{k}\in c$ denote the following element. Let $j_k=\min\{ j \; | \; g(j) >i_k\}$. Then we let

       $$u_{k}= \left(\prod_{j =j_k}^q v_{g(j)}\right)^{-1}v_{i_{k}}\left(\prod_{j =j_k}^q v_{g(j)}\right).$$
    Let $u(v, w, g)\in c^{p-q}$ denote the element $(u_{0},\ldots, u_{p-q-1})$. See Figure \ref{UcMorphismDeltaExample-UTuple}.

             \begin{figure}[ht!]
    \centering
      \includegraphics[scale=0.25]{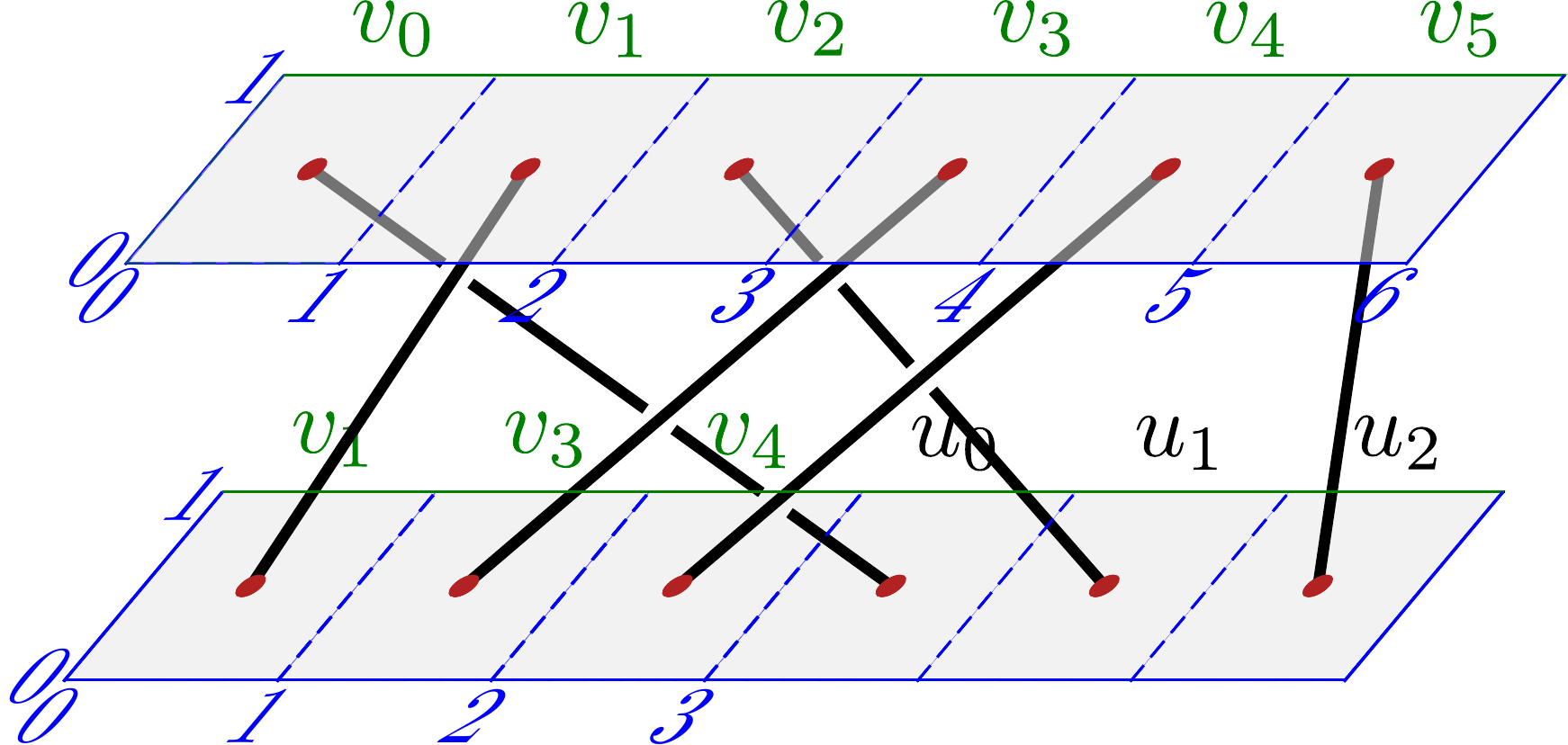} 
    \caption{For $g \in \Delta_{\text{inj,c}}(w,v)$, define a braid $\beta$ from $\sigma^{p+1}$ to $\sigma^{p+1}$ such that (read from bottom to top) the strands originating in positions $1, 2, \dots, q+1$ pass in front of the complementary $p-q$ strands, do not cross one another, and end at positions $g([q])$. The complementary strands pass behind, do not cross one another, and end at positions $[p] \setminus g([q])$. The elements $u_k \in c$ are defined so that the tuple $(v_{g(0)}, v_{g(1)}, \dots, v_{g(q+1)}, u_0, \dots, u_{p-q-1})$ is the image of $(v_0, v_1, \dots, v_{p})$ under the action of this braid. This relationship is illustrated in the case that $g$ is the monotone injection from Figure \ref{DeltaTwiddleExample}.}
    \label{UcMorphismDeltaExample-UTuple}
\end{figure}

Let $r$ be the order-preserving bijection 
\begin{align*}
    r \colon [p-q-1] &\longrightarrow [p]\setminus g([q])\\
    j& \longmapsto r_j
\end{align*}
Recall that $i(u(v,w,g), h\circ r)$ is an element of $\ohur\big((h\circ r)({[p-q-1]})\big)$ as defined in Definition \ref{Defni}, and recall that $b \in \bM(S \setminus h([p]))$. Note that 
$$(h\circ r)([p-q-1]) \sqcup \Big( S \setminus  h([p]) \Big) \; = \; \Big( S \setminus  (h \circ g)([q]) \Big) $$
and hence $i(u(v,w,g), h\circ r) \cdot b$ is in $\bM(S\setminus (h \circ g)([q]))$. Define 
 \begin{align*}
        \Delta_{\text{inj,c}}(w, v)\times L_{v}(\bM)(S)&\to L_{w}(\bM)(S)\\
        (g, (h, b))&\mapsto (h\circ g, i(u(v, w, g), h\circ r)\cdot b).
    \end{align*}

The maps
    %\footnote{Need to explain map $d_{i}$.} 
    $\delc(w, v)\times L_{v}(\bM)(S)\to L_{w}(\bM)(S)$ assemble to give a map of $\fb$-spaces $\delc(w, v)\times L_{v}(\bM)\to L_{w}(\bM)$.

\end{definition}

Observe that if $\bM$ is a module $\ohur$ that does not descend to a module over $\pi_0(\ohur)$, then the FB-spaces $L_v(M)$ do not assemble to form an augmented co-$\Delta_{\text{inj,c}}$-FB-space. To rectify this issue, we will add additional Moore loop space parameters, and work with $\delc$ instead of $\Delta_{\mathrm{inj}, c}$. 
For $\ohur$-modules we use the following construction instead of $L$. 

We will formally describe a relationship between $L_{\blacktriangle}$ and $R_{\blacktriangle}$ in Section \ref{DiscreteModelsOHur}: when $\bM \colon \fb \to \sets$ is a module over $\pi_0(\ohur)$, we will see that these constructions give level-wise weakly homotopy equivalent resolutions.

\begin{definition}\label{def: canonical res for OHur-mod}
    For a left $\ohur$-module
    %\footnote{need to define what an $\ohur$-module is} 
    $\bM$ and $v\in c^{p+1}$, let $R_{v}(\bM)(S)$ denote the following pullback of spaces
    %in the category of $\Sigma_{S}$-spaces\footnote{include upside down left thing}
    \[
    \begin{tikzcd}
        R_{v}(\bM)(S)\arrow[r]\arrow[d]\arrow[dr, phantom, "\ulcorner", very near start]&\text{Moore}(\bM(S))\arrow[d,"ev"]\\
        L_{v}(\bM)(S)\arrow[r,"i(v{,} -)\cdot -"]& \bM(S)
    \end{tikzcd}
    \]
    where $ev\colon \text{Moore}(\bM(S))\to \bM(S)$ is the map that sends a Moore path $\mu$ in $\bM(S)$ to the endpoint of $\mu$, and $i(v,-)\cdot-\colon L_{v}(\bM)(S)\to \bM(S)$ is the map
    \begin{align*}
        i(v,-)\cdot -\colon L_{v}(\bM)(S)&\to \bM(S).\\
        (h, b)&\mapsto i(v, h)\cdot b
        %[((T,\tau), b)]&\mapsto i(v, \tau|_{[p]})\cdot (\tau|_{T})_{*}(b).
    \end{align*}    
    %$\N$ graded space which in grading $n$ consists of pairs $(b, \gamma)$, where $b$ is a point in $\bM(n-p-1)$ and $\gamma$ is a Moore path in $\bM(n)$ ending at $b\cdot \sigma^{p+1}(v)$. 
    Given a map $f\colon S_{0}\to S_{1}$ in $\fb$, we have a map 
    \begin{align*}
        R_{v}(\bM)(S_{0})(f)\colon R_{v}(\bM)(S_{0})&\to R_{v}(\bM)(S_{1})\\
        ((h, b), \gamma)&\mapsto ((f|_{\im(h)}\circ h, (f|_{S_{0}\setminus \im(h)})_{*}(b)), f_{*}(\gamma)),
    \end{align*}
    i.e. the one induced by postcomposing in $L_{v}(\bM)(S)$ and $\text{Moore}(\bM(S))$ with $f$.
    We consider $R_{\blacktriangle}(\bM)$ as an enriched covariant functor $\delc^{\text{op}} \to\topo^{\fb}$ as follows. 

    Given $(e, \mu)\in \delc(w, v)$, let $[(e, \mu)]\colon [q]\to [p]$ denote the map in $\Delta_{\text{inj,c}}(w, v)$ corresponding to the 
    path component of $(e, \mu)$ in $\delc(w, v)$.
    Given $((h, b), \gamma)\in R_{v}(\bM)(S)$, since $\ohur(S)\to \hur(|S|)$ is a covering space, we can lift the Moore path $\mu$ in $\hur(|S|)$ uniquely to a Moore path $\widetilde{\mu}$ in $\ohur(S)$ starting at $i(v, h)$.  We can express the endpoint of $\widetilde{\mu}$ uniquely as $i(w, h\circ [(e, \mu)])\cdot \tilde{e}$, 
    %for a unique choice of a map $g\colon [q]\hookrightarrow S$ that factors through $h\circ f$, 
    %where $[(e, \mu)]\colon [q]\to [p]$ is the map in $\Delta_{\text{inj,c}}(w, v)$ corresponding to the 
    %path component of $(e, \mu)$ in $\delc(w, v)$,
    %braid $\mu$ in $\delc(w, v)$, 
   % and 
    where 
    $\tilde{e}\in \ohur(\im(h)\setminus\im(g))$. 
    %Our choice of  $\widetilde{\mu}$ and $\tilde{e}$ are independent of our choice of representative for $[((T,\tau), b)]$. 
    We have a map 
    %(see \cref{fig:map of functor R(M) for OHur mod})
    \begin{align*}
       \delc(w, v)\times R_{v}(\bM)(S)&\to R_{w}(\bM)(S)\\
       ((e, \mu), ((h, b), \gamma))&\mapsto ((h\circ [(e, \mu)], \tilde{e}\cdot b), \gamma *(\widetilde{\mu}\cdot b)).
    \end{align*}
    %$$\delc(w, v)\times R_{v}(\bM)(S)\to R_{w}(\bM)(S)$$ by sending $((e, \mu), ((h, b), \gamma))$ to $((g, \tilde{e}\cdot b), \gamma *(\widetilde{\mu}\cdot b))$.
    %$$([((T\sqcup (\im(\tau|_{[p]})\setminus\im(g),\tau\circ g), \tilde{e}\cdot (\tau|_{T})_{*}(b))], \gamma*(\widetilde{\mu}\cdot (\tau|_{T})_{*}(b)))).$$

    The maps
    %\footnote{Need to explain map $d_{i}$.} 
    $\delc(w, v)\times R_{v}(\bM)(S)\to R_{w}(\bM)(S)$ assemble to give a map of $\fb$-spaces $\delc(w, v)\times R_{v}(\bM)\to R_{w}(\bM)$.
    %\begin{align*}
    %    \delc(w, v)\times R_{v}(\bM)(S)&\to R_{w}(\bM)\\
    %    ((d, \mu), ([((T,\tau), b)], \gamma)))&\mapsto 
    %    ([((T\sqcup (\im(f)\setminus\im(g),\tau\circ g), (\tau|_{T})_{*}(b)\cdot \tilde{d})], \gamma*((\tau|_{T})_{*}(b)\cdot \widetilde{\mu})))
        %(b\cdot d, \gamma*(b\cdot \mu)).
    %\end{align*}
    The maps $$\epsilon_{v}\colon R_{v}(\bM)\to \bM$$ given by sending $((h, b), \gamma)\in R_{v}(\bM)(S)$  to $\gamma(0)$  assemble to give an augmentation map $\epsilon_{\blacktriangle}\colon R_{\blacktriangle}(\bM)\to \bM$.
    %%\footnote{define augmentation map.}
\end{definition}
%Equivalently,  $R_{v}(\bM)(S)$ is the space of triples $(g, b, \gamma)$, where $g\colon [p]\hookrightarrow S$, $b\in \bM(S\setminus \im(g))$, and $\gamma$ is a Moore path in $\bM(S)$ ending at $b\cdot i(v, g)$.
%The augmentation map $\epsilon_{\bullet}\colon R_{\bullet}(\bM)\to \bM$ induces a map $|\epsilon_{\bullet}|\colon |R_{\bullet}(\bM)|\to \bM$.
\begin{lemma}\label{lem: canon res map is a fibration}
    The map $\epsilon_{v}\colon R_{v}(\bM)(S)\to \bM(S)$ is a Serre fibration for all $v\in \delc$ and $S\in \fb$.
\end{lemma}
\begin{proof}
    Suppose we have the following commutative diagram
    \begin{center}
    \begin{tikzcd}
        D^{k}\times\{0\}\arrow[r, "f_{0}"]\arrow[d] & R_{v}(\bM)(S)\arrow[d, "\epsilon_{v}"]\\
        D^{k}\times[0,1]\arrow[r, "f"]& \bM(S).
    \end{tikzcd}
    \end{center}
    with $v\in c^{p+1}$.
    We want to construct a lift $\widetilde{f}\colon D^{k}\times[0,1]\to R_{v}(\bM)(S)$ of $f$ such that $\widetilde{f}(x,0)=f_{0}(x)$ for all $x\in D^{k}$ and $\epsilon_{v}\circ \widetilde{f}=f$. Let $f_{0}(x)=((h(x), b(x)), \gamma(x))$. Recall $h(x) \in \Hom_{\mathrm{FI}}([p], S)$, $b(x) \in \bM(S \setminus \im(h(x)))$, and $\gamma(x)$ is a Moore path in $\bM_S$ ending at $i(v, h(x))\cdot b(x) \in \bM_{S}$.

    For each $x\in D^{k}$ and $t\in [0,1]$, let $\widetilde{\gamma}(x, t)\in\text{Moore}(\bM(S))$ be the Moore path with
    $$\widetilde{\gamma}(x, t)(s)=\begin{cases} f(x,t-s) & \text{if } s<t,\\
    \gamma(x)(s-t) & \text{if } s\geq t.
    \end{cases}$$
    The function $\widetilde{\gamma}(x, t)$ is continuous at $s=t$, since $f(x, 0)=\epsilon_{v}\circ f_{0}(x)=\gamma(x)(0)$.
    Let $\widetilde{f}(x,t)=((h(x), b(x)), \widetilde{\gamma}(x, t))$. Since the endpoint of $\widetilde{\gamma}(x,t)$ is $i(v, h(x))\cdot b(x)$ for all $(x,t)\in D^{k}\times[0,1]$, $\widetilde{f}(x,t)$ is in $R_{v}(\bM)(S)$. We have that $\widetilde{f}(x,0)=f_{0}(x)$ for all $x\in D^{k}$ since $\widetilde{\gamma}(x, 0)=\gamma(x)$ and  $\epsilon_{v}\circ \widetilde{f}=f$ since $\epsilon_{v}\circ \widetilde{f}(x,t)=\widetilde{\gamma
    }(x,t)(0)=f(x, t-0)=f(x, t)$.
\end{proof}
Now we will relate $ \hocolim\limits_{\blacktriangle\in \delc^{\op}}R_{\blacktriangle}(\ohur)$ to $\ohur$. Recall that we describe in Definition \ref{DefnModelHocolim} an explicit model for the homotopy colimit. 
\begin{definition}
    Let $\ohur[, +]$ denote the sub-$\fb$-space of $\ohur$ which is empty when $S=\varnothing$ and agrees with $\ohur$ otherwise.
\end{definition}
\begin{lemma}\label{lem:res of OHur}
The map $$\hocolim\limits_{\blacktriangle\in \delc^{\op}}\epsilon_{\blacktriangle}\colon \hocolim\limits_{\blacktriangle\in \delc^{\op}}R_{\blacktriangle}(\ohur)\to \ohur$$ is a weak homotopy equivalence onto $\ohur[, +]$.
\end{lemma}
\begin{proof}
    This result is a Hurwitz space analogue of Krannich \cite[Example 2.18]{MR4019896} and Randal-Williams \cite[Lemma 3.2]{MR4767884}.

    For $v\in c^{p+1}$, the space $R_{v}(\ohur)(\varnothing)$ consists of a triple $(g, b, \gamma)$, where $g\colon [p]\hookrightarrow  \varnothing$, $b\in \ohur(\varnothing\setminus\im(g))$, and $\gamma$ a Moore path  to $i(v, g)\cdot  b$ in $(\ohur)(\varnothing)$. The set $[p]$ is nonempty, so the set of injections $g$ is empty, and hence the space $R_{v}(\ohur)(\varnothing)$ is empty. Therefore, the fiber of $\hocolim\limits_{\blacktriangle\in \delc^{\op}}\epsilon_{\blacktriangle}$ over $S=\varnothing$ is empty.

    %For $(v, \rho)\in\csig$, the space $R_{(v, \rho)}(\ohur)(0)$ consists of a point $b\in \hur(0-p-1)$ and a Moore path to $b\cdot \sigma^{p+1}(v, \rho)$, so is empty. Therefore, the fiber of $|\epsilon_{\bullet}|$ over the point of grading $0$ is empty.

    We want to show that the homotopy fiber of $\hocolim\limits_{\blacktriangle\in \delc^{\op}}\epsilon_{\blacktriangle}$ over any $a\in \ohur(S)$, with $S\neq \varnothing$, is weakly contractible. First, we will describe the (literal) fiber of $\hocolim\limits_{\blacktriangle\in \delc^{\op}}\epsilon_{\blacktriangle}$ over $a$ and then we will relate it to  the homotopy fiber of $\hocolim\limits_{\blacktriangle\in \delc^{\op}}\epsilon_{\blacktriangle}$ over $a$. The fiber over $a$ is equal to the geometric realization of  the sub-semi-simplicial space $$BF(a)_{\bullet}\subseteq B_{\bullet}(*, \delc, R_{\blacktriangle}(\bM))$$ given by taking the  levelwise fiber  over $a$ of the map $$B(\epsilon)_{\bullet}\colon B_{\bullet}(*, \delc, R_{\blacktriangle}(\ohur))\to \ohur,$$ that is,  
    $$BF(a)_{p}\colonequals\{ x\in B_{p}(*, \delc, R_{\blacktriangle}(\bM))\colon B(\epsilon)_{p}(x)=a\}.$$ 
    We can describe $BF(a)_{\bullet}$ more explicitly as follows.
    Let $F(a)_{\blacktriangle}\colon \delc^{\text{op}}\to \text{Top}$ be the enriched covariant functor which assigns to $v\in \delc$ the (literal) fiber of the map $\epsilon_{v}\colon R_{v}(\ohur)(S)\to \ohur(S)$ at $a$.  The $\delc^{\op}$ structure maps on  $F(a)_{\blacktriangle}$ are induced by those on $R_{\blacktriangle}(\ohur)(S)$.
      For $v\in c^{p+1}$, 
    %by \cref{lem: different desc of L(M)}, 
    as a set $F(a)_{v}$ consists  of all tuples $((g, b), \gamma)$ with 
    $g\colon [p]\hookrightarrow  S$ an injective map, $b\in \ohur(S\setminus\im(g))$, and $\gamma$ a Moore path  from $a$ to $i(v, g)\cdot b$. Observe that we therefore have $$BF(a)_{\bullet}=B_{\bullet}(*, \delc, F(a)_{\blacktriangle}).$$
    %an element $b\in \ohur[, n-p-1]$ and a Moore path from $d$ to $b\cdot i(v, \rho)$.
    %We have that $F(d)_{v}$ is homotopy discrete for all $v$, since it is the homotopy fiber of the map
    %\begin{align*}
    %    \hur[,n-p-1]&\to \hur[,n]\\
    %    d&\mapsto d\cdot i(v)
    %\end{align*}
    %and this is a map of $K(\pi, 1)$'s which is injective on fundamental groups.\footnote{Need a different justification. I don't know if this map actually is injective on $\pi_1$. Also, it isn't a map of Eilenberg--MacLane spaces}
    
   The map $\epsilon_{v} \colon R_{v}(\ohur)(S)\to \ohur[,S]$ is a Serre fibration for every $v\in \delc$ by \cref{lem: canon res map is a fibration}.
    %\footnote{I don't have a rigorous proof for this claim, but I think the following idea is sufficient. The analogous map $R_{p+1}(\oconf)\to \oconf$ is a fibration, so given maps $D^{m}\times [0,1]\to \ohur$ and  $D^{m}\to R_{v}(\ohur)$, we can assign the non-monodromy data in $R_{v}(\ohur)$. Then we use that $\ohur\to \oconf$ is a fibration to assign the monodromy data in $R_{v}(\ohur)$. Honestly, I'm not sure this argument is totally correct/rigorous though.}
Since $B_{p}(*, \delc^{\text{op}}, R_{\blacktriangle}(\ohur)(S))$ is equal to 
    $$\bigsqcup_{\substack{v_{0},\ldots, v_{p} \\ \in \text{ob}(\delc^{\text{op}})}}\delc^{\text{op}}(v_{0}, v_{1})\times\cdots \times\delc^{\text{op}}(v_{p-1}, v_{p})\times  R_{v_{p}}(\ohur)(S),$$ the map $$B(\epsilon)_{p}\colon B_{p}(*, \delc^{\text{op}}, R_{\blacktriangle}(\ohur)(S))\to \ohur[,S]$$ is therefore a Serre fibration for each $p\geq 0$. As a result, the augmented semi-simplicial space $B(\epsilon)_{\bullet}$ satisfies the hypotheses of 
   Ebert--Randal-Williams \cite[Lemma 2.14]{MR3995026}, and so the homotopy fiber of $\hocolim\limits_{\blacktriangle\in \delc^{\op}}\epsilon_{\blacktriangle}$ over $a$ is weakly homotopy equivalent to $\hocolim\limits_{\blacktriangle\in \delc^{\op}}F(a)_{\blacktriangle}$.
    %\footnote{Need a proper reference for this because the reference Oscar cites does not apply here (since Hurwitz spaces are not connected).} 

     %We have a map 
     %$$|F(d)_{\bullet}|\to |\pi_{0}F(d)_{\bullet}|$$ induced by the map $F(d)_{\bullet}\to \pi_{0}F(d)_{\bullet}$ and it is is a weak homotopy equivalence since $F(d)_{v}$ is homotopy discrete for all $v$.\footnote{maybe drop going to $\pi_0$}
    
 Now let $R_{\bullet}(\conf[])$ be as in Definition \ref{DefnRConfModule}, that is, we view $\conf[]$ as a module over itself, and \emph{not} as $\ohur$-module. Consider the augmentation $$\epsilon'\colon R_{\bullet}(\conf[])\to \conf[].$$ 
 Recall the forgetful map $$\phi\colon\hur[,|S|] \to \conf[|S|]$$ from \cref{map: forgetful map to conf} given by forgetting the map to $BG$. Let $$\zeta\colon \ohur[, S]\to \hur[,|S|]$$ be the map given by forgetting the order of the points.
 Let $a'=\phi\circ \zeta(a)$ and let $F(a')_{\bullet}\subset R_{\bullet}(\conf[])$ be the $\widetilde{\Delta}_{\text{inj}}$-space with $p$-simplices given by the literal fibers over $d'$ of the maps $\epsilon'_{p}\colon R_{p}(\conf[])\to \conf[]$. This means that elements of $F(a')_{p}$ consist of a configuration $b\in \conf[n-p-1]$ and a Moore path from $a'$ to $\sigma^{p+1}\cdot b$. The $\widetilde{\Delta}^{\op}_{\text{inj}}$ structure maps on  $F(a')_{\bullet}$ are induced by those on $R_{\bullet}(\conf)(|S|)$. For every $v\in c^{p+1}$, the function $\phi$ 
 %\begin{align*}
%     \widetilde{\Delta}_{\text{inj}}\colon \delc & \longrightarrow \widetilde{\Delta}_{\text{inj}} \\
%     v  & \longmapsto [p] \\
%     \delc(w, v)& \longrightarrow \widetilde{\Delta}_{\text{inj}}([q], [p]) \qquad \text{ for $w \in c^{q+1}, v \in c^{p+1}$}\\
%     f & \longmapsto f
% \end{align*}
 induces a map 
 $$\phi_{F_v}\colon F(a)_{v}\to F(a')_{p}$$ that is natural in the sense that for any $w\in\delc$, the following diagram commutes

    \begin{center}
    \begin{tikzcd}
        \delc(w, v)\times F(a)_{v}\arrow[r]\arrow[d,swap, "\phi_{\widetilde{\Delta}}\times \phi_{F_v}"] & F(a)_{w}\arrow[d, "\phi_{F_w}"]\\
        \widetilde{\Delta}_{\text{inj}}([q], [p])\times F(a')_{p}\arrow[r]& F'(a')_{q}
    \end{tikzcd}
    \end{center}
    Therefore, 
    %by applying the functor $\pi_{0}$ to the maps $F(d)\to F(d')$, 
    these maps assemble to give a map of semi-simplicial spaces
    $$\phi_{B_{\bullet}}\colon B_{\bullet}(*, \delc, F(a)_{\blacktriangle})\to B_{\bullet}(*, \widetilde{\Delta}_{\text{inj}}, F(a')_{\bullet})$$
    and an induced map 
    \begin{align}%\label{map on hocolims}
        \Vert \phi_{B_{\bullet}}\Vert \colon\hocolim\limits_{\blacktriangle\in \delc^{\op}}F(a)_{\blacktriangle}\to \hocolim\limits_{\bullet\in \widetilde{\Delta}_{\text{inj}}^{\op}}F(a')_{\bullet}
    \end{align}
    %$\hocolim\limits_{v\in \delc^{\op}}F(a)_{v}\to \hocolim\limits_{[p]\in \widetilde{\Delta}_{\text{inj}}^{\op}}F(a')_{p}$. 
    Using Damiolini \cite[Theorem 2.48]{damiolini2013braid} (see also Hatcher--Vogtmann \cite[Proposition 3.2]{HV-tethers} for a published reference), Randal-Williams shows that $\hocolim\limits_{\bullet\in \widetilde{\Delta}_{\text{inj}}^{\op}}F(a')_{\bullet}$ is weakly contractible in the proof of \cite[Lemma 3.2]{MR4767884}. 
    As a result, to show that $\hocolim\limits_{\blacktriangle\in \delc^{\op}}F(a)_{\blacktriangle}$ is weakly contractible, it suffices to show that the map 
   $ \Vert \phi_{B_{\bullet}}\Vert $ 
    %$\hocolim\limits_{v\in \delc^{\op}}F(a)_{v}\to \hocolim\limits_{[p]\in \widetilde{\Delta}_{\text{inj}}^{\op}}F(a')_{p}$ 
    is a weak homotopy equivalence. 
A map of semi-simplicial spaces that induces a weak homotopy equivalence between spaces of $p$-simplices for all $p$ induces a weak homotopy equivalences on the geometric realization; see, for example, Ebert--Randal-Williams \cite[Theorem 2.2]{MR3995026}. Thus, it suffices to show that $\phi_{B_{\bullet}}$ is a level-wise weak homotopy equivalence.

    First, we will show that the the forgetful map $$\phi_F \colonequals \bigsqcup_{v\in c^{p+1}}\phi_{F_v} \colon\bigsqcup_{v\in c^{p+1}}F(a)_{v}\to F(a')_{p}$$ is a weak homotopy equivalence. Fix a point $(b, \gamma)$ in $F(a')_{p}$. It suffices to show that the homotopy fiber of $\phi_F$ over $(b, \gamma)$ is contractible. Since $\phi_F$ is a fibration, the homotopy fiber over $(b, \gamma)$ is the same as the literal fiber over $(b, \gamma)$. We can write $\gamma\colon [0,\infty) \to\conf[|S|]$ as $\gamma(t)=(\lambda(t), \xi(t))$,
    %\footnote{specify what this notation means when defining $\conf$} 
    with the conditions that $\gamma(0)=a'$ and there is some $T_{0}\geq 0$ such that $\gamma(t)=\gamma(T_{0})=\sigma^{p+1}\cdot b$ for all $t\geq T_{0}$. 
    
    Our current goal is to show $\phi_F^{-1}(b, \gamma)$ is weakly contractible (compare with
    Step 2  of \cref{lem: delc is disc}).   
    Qualitatively, a point in $\phi_F^{-1}(b, \gamma)$ is a continuous choice of labeling by elements of $S$ on the configurations $\xi(t)$ and choice of map $f(t)$ to $BG$ on the complement of $\xi(t)$. This choice of labeling on $\xi(0)$ is required to be the labeling on $a$. By continuity, this determines the labeling on $\xi(t)$ for all $t$.

    %Let $\lambda(t)$ be the Moore parameter associated with $\gamma(t)$. 
    Let $$g  \colon ([0,\lambda(0) ] \times [0,1])  \setminus \xi(0)  \to BG$$ be the map associated to $a$. Recall that  
    $$g_v \colon  D_p    \to BG$$
    is the map associated to $i(v)$ as defined in Definition \ref{Defni}. The maps $$f(t) \colon ([0,\lambda(t) ] \times [0,1])  \setminus \xi(t)  \to BG$$ are required to satisfy three conditions: 
    \begin{enumerate}
     \item \label{ConditionInitialCondition} $f(0)=g$,
      \item \label{ConditionTerminalCondition} For all $t \geq T_0$, the restriction of $f(t)$ to $D_p$ is $g_v$ for some $v$.  
     \item \label{HorseshoeCondition} For all $t$, $f(t)$ maps the subspace ${\textstyle \boldsymbol{\bigsqcup}_{\lambda(t)}}$ to the basepoint of $BG$. 
      \end{enumerate}
     Let 
     $$C = \bigcup_t \Big( \left( [0,\lambda(t) ] \times [0,1] \right) \setminus \xi(t) \Big) \times \{t\}  \subset \R^3 \qquad \text{for } t \in [0, \infty)$$ 
     and $H \subseteq C $ be  
     $$ H = \bigcup_{t} \left({\textstyle \boldsymbol{\bigsqcup}_{\lambda(t)}} \times \{t\} \right) \subseteq \R^3 \qquad \text{for } t \in [0, \infty).$$ 

We view the family of maps $f(t)$ as a function $f \colon C \to BG$. The space of such functions satisfying Condition (\ref{HorseshoeCondition}) is $\mathrm{Map((C,H), (BG, \text{pt}))}$. There is a restriction map defined by setting $t=0$, 
$$\text{res}_{1} \colon \mathrm{Map}((C,H), (BG, \text{pt})) \to \mathrm{Map}((([0,\lambda(0) ] \times [0,1]) \setminus \xi(0)),{\textstyle \boldsymbol{\bigsqcup}_{\lambda(0)}}), (BG, \text{pt})).$$
The fiber of this map over $g$ are the functions that additionally satisfy Condition (\ref{ConditionInitialCondition}).

To understand Condition (\ref{ConditionTerminalCondition}), we will consider an additional restriction map, which is defined by setting $t=T_0$ and restricting to $D_p$, 
$$ \text{res}_{2} \colon \mathrm{Map}((C,H), (BG, \text{pt})) \to \mathrm{Map}\left(\left( D_p,{\textstyle \boldsymbol{\bigsqcup}_{p+1}}\right), (BG, \text{pt})\right).$$ 
Condition (\ref{ConditionTerminalCondition}) is the statement that $f$ is in the preimage of $\{g_v \; | \; v \in c^{p+1} \}$. 
Note that the pair $( D_p,{\textstyle \boldsymbol{\bigsqcup}_{p+1}})$ is homotopy equivalent to the pair $(\bigvee_{p+1} S^1, \text{pt})$. As in Step 2  of \cref{lem: delc is disc}, we conclude that $$\mathrm{Map}\left( \left( D_p,{\textstyle \boldsymbol{\bigsqcup}_{p+1}}\right), (BG, \text{pt})\right) \simeq G^{p+1}.$$

%Now, 
%\begin{align*}
% &  \mathrm{Map}\left( \left( D_p,{\textstyle \boldsymbol{\bigsqcup}_{p+1}}\right), (BG, \text{pt})\right) \\ 
%  \simeq \; & \mathrm{Map} \left(\left(\bigvee_{p+1} S^1, \text{pt} \right), (BG,\text{pt})\right)  \\
%  \cong \; &(\Omega BG)^{p+1} \\
%   \simeq \; & G^{p+1}.
%\end{align*}
\noindent Hence the mapping space $\mathrm{Map}\left( \left( D_p,{\textstyle \boldsymbol{\bigsqcup}_{p+1}}\right), (BG, \text{pt})\right) $ is homotopy discrete. Let $$ \mathrm{Map}_{c}\left( \left( D_p,{\textstyle \boldsymbol{\bigsqcup}_{p+1}}\right), (BG, \text{pt})\right) $$ be the union of the components containing the maps $\{g_v \; | \; v \in c^{p+1}\}$. The homotopy equivalence described above restricts to a homotopy equivalence 
\begin{align*}
 &  \mathrm{Map}_c\left( \left( D_p,{\textstyle \boldsymbol{\bigsqcup}_{p+1}}\right), (BG, \text{pt})\right) \simeq c^{p+1}.
\end{align*}

To summarize, we see that the set $\phi_F^{-1}(b, \gamma)$ is homeomorphic to the preimage $$(\text{res}_{1} \times \text{res}_{2})^{-1} \left( \{g\} \times \{g_v \; | \; v \in c^{p+1} \}\right).$$ The map $\text{res}_{1}$ is a fibration since it is obtained by applying the mapping space functor to a cofibration. Additionally, $\text{res}_{1}$ is a homotopy equivalence because the inclusion
$$ \left(([0,\lambda(0) ] \times [0,1]) \setminus \xi(0)),\; {\textstyle \boldsymbol{\bigsqcup}_{\lambda(0)}}\right) \quad \hookrightarrow \quad (C,H)$$ is a homotopy equivalence. Therefore $\text{res}_{1}^{-1}(g)$ is contractible.

Note that $\text{res}_{2}$ maps $\text{res}_{1}^{-1}(g)$ to  $\mathrm{Map}_{c}\left( \left( D_p,{\textstyle \boldsymbol{\bigsqcup}_{p+1}}\right), (BG, \text{pt})\right)$ since $g$ maps a small loop around each point in the configuration to a conjugacy class in $c$. The space $(\text{res}_{1} \times \text{res}_{2})^{-1} \left( \{g\} \times \{g_v \; | \; v \in c^{p+1} \}\right)$ is the pullback under the map 
\begin{align*}
    c^{p+1} & \longrightarrow \mathrm{Map}_{c}\left( \left( D_p,{\textstyle \boldsymbol{\bigsqcup}_{p+1}}\right), (BG, \text{pt})\right) \\ 
    v & \longmapsto g_v
\end{align*}
of the map $$\text{res}_{2}|_{\text{res}_{1}^{-1}(g)} \colon \text{res}_{1}^{-1}(g) \longrightarrow \mathrm{Map}_{c}\left( \left( D_p,{\textstyle \boldsymbol{\bigsqcup}_{p+1}}\right), (BG, \text{pt})\right).$$
Since $\text{res}_{2}|_{\text{res}_{1}^{-1}(g)}$ is a fibration,  $c^{p+1} \to \mathrm{Map}_{c}\left( \left( D_p,{\textstyle \boldsymbol{\bigsqcup}_{p+1}}\right), (BG, \text{pt})\right)$ is a homotopy equivalence, and $\text{res}_{1}^{-1}(g)$ is contractible,  $(\text{res}_{1} \times \text{res}_{2})^{-1} \left( \{g\} \times \{g_v \; | \; v \in c^{p+1} \}\right)$ is weakly contractible. 
We have therefore proven that $\phi_F^{-1}(b, \gamma)$ is weakly contractible as desired.

    By \cref{lem: delc is homotopy discrete}, for every $v\in c^{p+1}$,
    %From the construction of $\delc$, 
     the map  $$\bigsqcup_{w\in c^{q+1}}  \phi_{\widetilde{\Delta}}\colon \bigsqcup_{w\in c^{q+1}} \delc(w, v)\to \widetilde{\Delta}_{\text{inj}}([q], [p])$$ is a weak homotopy equivalence. Therefore,
     %.\footnote{maybe try to say more here or in the homotopy discrete lemma}  Therefore, 
    %$G$ induces a weak equivalence from the nerve of $\delc$ to the nerve of $\widetilde{\Delta}_{\text{inj}}$--equivalently, 
    $\bigsqcup_{w\in c^{q+1}}  \phi_{\widetilde{\Delta}}$ induces a weak homotopy equivalence from  
    $$\bigsqcup_{v_{0},\ldots, v_{r}\in \delc}\delc(v_{0}, v_{1}, )\times\cdots \times \delc(v_{r-1}, v_{r}) $$
    to 
    $$\bigsqcup_{p_{0},\ldots, p_{r}\in \widetilde{\Delta}_{\text{inj}}}\widetilde{\Delta}_{\text{inj}}(p_{0}, p_{1})\times\cdots \times\widetilde{\Delta}_{\text{inj}}(p_{r-1}, p_{r}),$$ where $p_{i}=\phi_{\widetilde{\Delta}}(v_{i})$.
    As a result, the map
    $$\phi_{B_r} \colon B_{r}(*, \delc, F(a)_{\blacktriangle})\to B_{r}(*, \widetilde{\Delta}_{\text{inj}}, F(a')_{\bullet})$$
    is a weak homotopy equivalence for all $r\geq 0$.    
\end{proof}
\begin{proposition}\label{prop: relating canonical res and derived indecomposables}
    Suppose that $\bM\in \topo^{\fb}$ is cofibrant and a left $\ohur$-module. There is a weak homotopy equivalence of $\fb$-spaces between the homotopy cofiber of $$\hocolim\limits_{\blacktriangle\in\delc}\epsilon_{\blacktriangle}\colon \hocolim\limits_{\blacktriangle\in\delc} R_{\blacktriangle}(\bM)\to \bM$$ and $Q^{\ohur}_{\mathbb{L}}(\bM)$.
\end{proposition}

\begin{proof}
    Compare with \cite[Theorem 3.1]{MR4767884}. The augmented $\delc$-object $$(R_{\blacktriangle}(\ohur), \ohur, \epsilon_{\blacktriangle})$$ is constructed from the left $\ohur$-module structure on $\ohur$, so it has a compatible right $\ohur$-module structure via the following maps. Let $v \in c^{p+1}$. 
    \begin{align*}
         \text{Ind}^{\Sigma_{S\sqcup W}}_{\Sigma_{S}\times \Sigma_{W}} R_{v}(\ohur)(S) \times \ohur(W) &\to R_{v}(\ohur)(S\sqcup W)\\
        \big(((h, b), \gamma), a\big)
        %(([((T, \tau), b)], \gamma), a)
        &
        \mapsto ((\tilde{h}, b\cdot a), \gamma\cdot a),
        %\mapsto ([((W\sqcup T, \text{id}_{W}\sqcup\tau), b\cdot a)],  \gamma \cdot a)
        %\mapsto ([((W\sqcup T, \text{id}_{W}\sqcup\tau), b\cdot a)],  \gamma \cdot a)
    \end{align*}
    where $\tilde{h}\colon [p]\hookrightarrow S\sqcup W$ is the composition of $h\colon [p]\hookrightarrow S$ with the inclusion $S\hookrightarrow S\sqcup W$. 
    Contracting Moore paths in $R_{v}(\ohur)(S)$ gives a deformation retraction from $R_{v}(\ohur)(S)$ to the subspace where the Moore path is trivial, and this subspace is isomorphic to $L_{v}(\bM)(S)$
    %$\colim_{(T,\tau)\in (-\sqcup [p]\downarrow S)}\bM(T)$ 
    as an $\fb$-space and as a right $\ohur$-module.

    Using the same formula as the one describing the right $\ohur$-module structure on $R_{v}(\ohur)$, we have maps
    \begin{align*}
        \psi_{v}' \colon \text{Ind}^{\Sigma_{S\sqcup W}}_{\Sigma_{S}\times \Sigma_{W}} R_{v}(\ohur)(S)\times \bM(W) &\to R_{v}(\bM)(S\sqcup W)\\
        \big(((h, b), \gamma), a\big)
        %(([((T, \tau), b)], \gamma), a)
        &
        \mapsto ((\tilde{h}, b\cdot a), \gamma\cdot a),
        %(([((T, \tau), b)], \gamma), a)&\mapsto ([((W\sqcup T, \text{id}_{W}\sqcup\tau), b\cdot a)], \gamma\cdot a)
    \end{align*}
    For any $d\in \ohur$, we have that 
    $$\psi_{v}'((h, b), \gamma)\cdot d, a)=\psi_{v}'((h, b), \gamma), d\cdot  a).$$
    %$$\psi_{v}'(([((T, \tau), b)], \gamma)\cdot d, a)=\psi_{v}'(([((T, \tau), b)], \gamma), d\cdot  a).$$
    Therefore, $\psi_{v}'$ factors through a map 
    $$\psi_{v} \colon R_{v}(\ohur)\otimes_{\ohur} \bM \to R_{v}(\bM)$$
where  $R_{v}(\ohur)\otimes_{\ohur} \bM$ is the coequalizer of the maps described above
$$ R_{v}(\ohur)\otimes_{\fb} {\ohur} \otimes_{\fb} \bM \rightrightarrows R_{v}(\ohur)\otimes_{\fb} \bM.  $$

    %The map $\psi_{v}$ is a weak homotopy equivalence, by using that 
    %$R_v(\ohur)(T)\simeq \bigoplus_{g\colon [p]\hookrightarrow T}\ohur(T\setminus\im(g))\otimes i(v, g)$ and
    We will show that the composition
    $$\Vert B_{\bullet}(R_{v}(\ohur), \ohur, \bM)\Vert\to R_{v}(\ohur)\otimes_{\ohur} \bM\xrightarrow{\psi_{v}} R_{v}(\bM)$$
    of the augmentation map and the map $\psi_{v}$ is a weak homotopy equivalence.
    First, observe that 
    $$R_{v}(\ohur)(S)\simeq L_{v}(\ohur)(S)=\bigsqcup_{h\colon [p]\hookrightarrow S}\ohur(S\setminus\im(h)).$$
    Similarly,  $R_{v}(\bM)(S)\simeq \bigsqcup\limits_{h\colon [p]\hookrightarrow S}\bM(S\setminus\im(h))$ for all $S$. An extra degeneracy argument shows that the augmentation map
    $\Vert B_{\bullet}(\ohur, \ohur, \bM)\Vert\to \bM$ is a homotopy equivalence. Combining these observations, 
   \begin{align*}
   & B_{q}(R_{v}(\ohur), \ohur, \bM)(S) \\
   & =\bigsqcup_{S=A_0 \sqcup A_1 \sqcup \dots \sqcup A_{q+1}}   R_{v}(\ohur)(A_0) \times \ohur(A_1) \times \dots \times  \ohur(A_q) \times \bM(A_{q+1}) 
    \\ & \simeq \bigsqcup_{S= A_0  \sqcup \dots \sqcup A_{q+1}}  \left( \bigsqcup_{h \colon [p] \hookrightarrow A_0} \ohur(A_0\setminus \mathrm{im}(h)) \right) \times \ohur(A_1) \times \dots \times  \ohur(A_q) \times \bM(A_{q+1}) 
       \\
   & \cong \bigsqcup_{h \colon [p] \hookrightarrow S} \left( \bigsqcup_{S\setminus \mathrm{im}(h) = W
   _0  \sqcup \dots \sqcup W_{q+1}}   \ohur(W_0)  \times \ohur(W_1) \times \dots \times  \ohur(W_q) \times \bM(W_{q+1}) \right)\\
      & = \bigsqcup_{h \colon [p] \hookrightarrow S} B_q(   \ohur, \ohur,  \bM)(S \setminus \mathrm{im}(h)).
    \end{align*}

Since these equivalences are compatible semi-simplicial structure, we infer
       \begin{align*}
      &   \Vert B_{\bullet}(R_{v}(\ohur), \ohur, \bM)(S)\Vert  \\
& \simeq  \bigsqcup_{h \colon [p] \hookrightarrow S} \left\Vert B_\bullet\left(   \ohur, \ohur,  \bM)(S \setminus \mathrm{im}(h)\right) \right\Vert\\
&  \simeq  \bigsqcup_{h \colon [p] \hookrightarrow S} \bM(S \setminus \mathrm{im}(h)) \\
&  \simeq  R_v(\bM)(S).
    \end{align*}
    
    \begin{comment}

    \begin{align*}
        \Vert B_{\bullet}(\ohur[,+], \ohur, \bM)\Vert& \xleftarrow{\sim} \Vert B_{\bullet}(\hocolim\limits_{\blacktriangle\in \delc} R_{\blacktriangle}(\ohur), \ohur, \bM)\Vert\\
        &\xrightarrow{\sim} \hocolim\limits_{\blacktriangle\in \delc} R_{\blacktriangle}(\bM)
    \end{align*}
    %$$\Vert B_{\bullet}(\ohur[,+], \ohur, \bM)\Vert\xleftarrow{\sim} \Vert B_{\bullet}(\hocolim\limits_{v\in \delc} R_{v} R_{v}(\ohur), \ohur, \bM)\Vert\xrightarrow{\sim} \hocolim\limits_{v\in \delc} R_{v}(\bM)$$
    over $\bM$. 
    \end{comment}

    Using \cref{lem:res of OHur} and the fact that we can commute the bar construction with the homotopy colimit defining $\hocolim\limits_{\blacktriangle\in \delc}R_{\blacktriangle}(\bM)$, we have a weak homotopy equivalence
$$\Vert B_{\bullet}(\ohur[,+], \ohur, \bM)\Vert \xleftarrow{\sim} \Vert B_{\bullet}(\hocolim\limits_{\blacktriangle\in \delc} R_{\blacktriangle}(\ohur), \ohur, \bM)\Vert.$$
Using the above paragraph, we obtain another weak homotopy equivalence $$\Vert B_{\bullet}(\hocolim\limits_{\blacktriangle\in \delc} R_{\blacktriangle}(\ohur), \ohur, \bM)\Vert \xrightarrow{\sim} \hocolim\limits_{\blacktriangle\in \delc} R_{\blacktriangle}(\bM).$$ These homotopy equivalences are compatible with the maps to $\bM$.
This identifies the homotopy cofiber of $\hocolim\limits_{\blacktriangle\in \delc} \epsilon_{\blacktriangle}\colon \hocolim\limits_{\blacktriangle\in \delc} R_{\blacktriangle}(\bM)\to \bM$ with the homotopy cofiber of the composition
    $$\Vert B_{\bullet}(\ohur[,+], \ohur, \bM)\Vert\xrightarrow{i} \Vert B_{\bullet}(\ohur, \ohur, \bM)\Vert\xrightarrow{\simeq}\bM.$$ 
    The homotopy cofiber of $i$ is weakly homotopy equivalent to the homotopy cofiber of the composition of  $i$ with the augmentation map $\hocolim\limits_{\blacktriangle\in \delc} \epsilon_{\blacktriangle}\colon \hocolim\limits_{\blacktriangle\in \delc} R_{\blacktriangle}(\bM)\to \bM$ because the map $\Vert B_{\bullet}(\ohur, \ohur, \bM)\Vert\to\bM$ is a weak homotopy equivalence. 

    Observe that $\ohur/\ohur[,+]$ is a point in degree 0. In other words, $$\ohur/\ohur[,+]\simeq \bunit.$$ 
    Hence, the homotopy cofiber of $\hocolim\limits_{\blacktriangle\in \delc} \epsilon_{v}\colon \hocolim\limits_{\blacktriangle\in \delc} R_{\blacktriangle}(\bM)\to \bM$ is weakly homotopy equivalent to $\Vert B_{\bullet}(\bunit, \ohur, \bM)\Vert$. 
    
    To conclude the proof, we show that $Q^{\ohur}_{\mathbb{L}}(\bM)$ is weakly homotopy equivalent to $\Vert B_{\bullet}(\bunit, \ohur, \bM)\Vert$ by checking that the hypotheses of \cref{lem: model for der-ind} hold. 
    %By assumption $\bM$ is cofibrant, so we only need to check that $\ohur$ is cofibrant in $\topo^{\fb}$. 
    The space $\ohur[,\bn]$ is $\Sigma_n$-homotopy equivalent to a $\Sigma_n$-CW complex since $\ohur[,\bn]$ is an $\Sigma_n$-principal bundle over $\hur[,n]$, and $\hur[,n]$ is homotopy equivalent to a CW complex by \cite[Proposition 2.5]{MR3488737}. Thus $\ohur$ is cofibrant in $\topo^{\fb}$. 
    Since $\ohur$ is cofibrant and,  by assumption, $\bM$ is cofibrant, by \cref{lem: model for der-ind}, $Q^{\ohur}_{\mathbb{L}}(\bM)$ is weakly homotopy equivalent to $\Vert B_{\bullet}(\bunit, \ohur, \bM)\Vert$. \end{proof}
    %\footnote{add a paragraph explaining why this is Q and a few lemmas in section 2.}    
    %By LEMMA 2\footnote{add lemma}, since $\bM$ and  this agrees with with $Q^{\ohur}_{\mathbb{L}}(\bM)$ . By LEMMA\footnote{add lemma}, this agrees with     
    %is weakly homotopy equivalent to the homotopy cofiber of the map $i$. The homotopy cofiber of $i$ is weakly homotopy equivalent to $$\Vert B_{\bullet}(\ohur/\ohur[,+], \ohur, \bM)\Vert,$$ which agrees with $Q^{\ohur}_{\mathbb{L}}(\bM)$.
    %provided that $\bM$ and $\ohur$ are cofibrant objects in $\topo^{\fb}$. By the same argument as in the last paragraph of \cite[Theorem 3.1]{MR4767884}, we may ignore the cofibrancy requirement for $\bM$.\footnote{I haven't checked that this is a true statement, but I assume that it is true. Oscar is working in a slightly different setting than here (he is working with $\N$-graded spaces instead of $\fb$-spaces), but I don't think this distinction should really matter.}
    %\footnote{I need to check if this sentence is actually true because now we're working in the functor category $\fb\to \topo$. Presumably it is still true, but this still requires actually checking stuff.}
    %Each $\ohur(S)$ is cofibrant because it is a covering space of $\hur(|S|)$ and $\hur(|S|)$ is homotopy equivalent to a CW complex.\footnote{Alternatively, $\ohur(S)$ is cofibrant because it is a covering of $\conf[|S|]$, and $\conf[|S|]$ has the structure of a smooth manifold with corners.}
    %, since $\hur(S)$ is a covering of $\conf[|S|]$.} 
    %Therefore, $\ohur$ is a cofibrant object in $\topo^{\fb}$.
\section{A small model for the \texorpdfstring{$\ohur$}{OHur}-module cells of \texorpdfstring{$\pi_0\left(\ohur\right)$}{the connected components of ordered Hurwitz space}}\label{DiscreteModelsOHur}
%\section{A small model for the $\ohur$-module cells of \texorpdfstring{$\pi_0\left(\ohur\right)$}{A small model for the OHur-module cells of its connected components}} \label{DiscreteModelsOHur}
Throughout this section, let $\bM\in \mathrm{Set}^{\fb}$ be a left $\pi_0(\ohur)$-module. 
The goal of this section is to compare the resolution $R_{\blacktriangle}(\bM)$ with the much smaller resolution $L_{\blacktriangle}(\bM)$.

\begin{lemma}\label{lem:RL}
     Let $\bM\in \mathrm{Set}^{\fb}$ be a left $\pi_0(\ohur)$-module. There is a natural transformation of $\fb$-spaces  
      $$ \psi_{v} \colon R_{v}(\bM)  \longrightarrow L_{v}(\bM) \times [0, \infty).$$
      that induces, for each finite set $S$, a homeomorphism 
     $$  R_{v}(\bM)_S  \xrightarrow{\cong} L_{v}(\bM)_S \times [0, \infty).$$ 
\end{lemma}

\begin{proof}
    This is a consequence of the observation that a Moore path in a discrete set is necessarily a constant path. Thus the data of a Moore path in $L_{v}(\bM)$ is a point in $L_{v}(\bM)$ and a Moore parameter $T_0 \in [0,\infty)$. 
\end{proof}
Recall the map $ \Upsilon\colon \widetilde{\Delta}_{\mathrm{inj},c}  \to {\Delta}_{\mathrm{inj},c}$ of \cref{lem: delc is disc}. Recall moreover from \cref{LDelStructure} the action of ${\Delta}_{\mathrm{inj},c}$ on $L_{\blacktriangle}(\bM)$. 
For $v\in c^{p+1}$ and $w \in c^{q+1}$, consider a morphism $(a, \mu) \in \delc(v,w)$. Let $T_0 \in [0,\infty)$ be the Moore parameter associated to the Moore path $\mu$ in $\hur[, p]$. We obtain a map 
\begin{align*} (a, \mu)_* \colon L_v(\bM)_S \times [0, \infty) \longrightarrow L_w(\bM)_S \times [0, \infty)
\end{align*}
where we act by $\Upsilon(a, \mu)$ on $L_v(\bM)_S$ and by addition by $T_0$ on the factor $[0,\infty)$. 
The maps  $ \psi_{v}$ assemble to give an isomorphism of $\delc$-$\fb$-spaces  
$$ \psi_{\blacktriangle} \colon R_{\blacktriangle}(\bM)  \to L_{\blacktriangle}(\bM) \times [0, \infty).$$
We obtain the following corollary to \cref{lem:RL}. 

\begin{corollary} \label{RvsL}
     Let $\bM\in \mathrm{Set}^{\fb}$ be a left $\pi_0(\ohur)$-module. There is 
     a weak homotopy equivalence
    $$\psi \colon \hocolim\limits_{\blacktriangle\in \delc^{\mathrm{op}}} R_{\blacktriangle}(\bM)\xrightarrow{\simeq} \hocolim\limits_{\blacktriangle\in \Delta^{\rm{op}}_{\rm{inj,c}}} L_{\blacktriangle}(\bM)$$ that is compatible with the augmentation maps for $R_{\blacktriangle}(\bM)$ and $L_{\blacktriangle}(\bM)$. 
\end{corollary}

\begin{proof}
The isomorphism of of $\delc$-$\fb$-spaces  
$$ \psi_{\blacktriangle} \colon R_{\blacktriangle}(\bM)  \to L_{\blacktriangle}(\bM) \times [0, \infty)$$
implies an isomorphism of $\fb$-spaces
$$\hocolim_{\blacktriangle \in \delc^{\mathrm{op}}} R_{\blacktriangle}(\bM)\xrightarrow{\cong} \hocolim_{\blacktriangle \in \delc^{\mathrm{op}}} (L_{\blacktriangle}(\bM) \times [0,\infty)).$$
The homotopy equivalences $L_v(\bM)_S \times [0,\infty) \to L_v(\bM)_S$ given by forgetting the Moore parameters are compatible with the $\delc$-structure. Hence they induce a weak equivalence of $\fb$-spaces
$$\hocolim_{\blacktriangle \in \delc^{\mathrm{op}}} (L_{\blacktriangle}(\bM) \times [0,\infty)) \xrightarrow{\simeq} \hocolim_{\blacktriangle \in \delc^{\mathrm{op}}} (L_{\blacktriangle}(\bM) ) .$$

\cref{lem: delc is disc} shows that the functor $\delc^{\mathrm{op}} \to \Delta^{\mathrm{op}}_{\mathrm{inj}, c}$ is the identity on objects and a homotopy equivalence on each morphism space. Thus there is a weak homotopy equivalence 
    $$ \hocolim_{\blacktriangle\in \delc^{\mathrm{op}}} (L_{\blacktriangle}(\bM) ) \xrightarrow{\simeq} \hocolim_{\blacktriangle \in \Delta^{\mathrm{op}}_{\mathrm{inj}, c}} (L_{\blacktriangle}(\bM) ).    $$
    These maps are compatible with the augmentations.
\end{proof}
We now introduce a construction that converts a $\delc$-space to a semi-simplicial space.
\begin{definition}
        Let $X_{\blacktriangle}$ be a $\Delta^{\mathrm{op}}_{\mathrm{inj}, c}$-space. Let $ \mathfrak{S}_{\bullet}(X_{\blacktriangle}) $ be the following semi-simplicial space. The $p$-simplices are the space 
        $$ \mathfrak{S}_{p}(X_{\blacktriangle}) = \bigsqcup_{v \in c^{p+1}} X_v.$$
        Given an order-preserving injection $f\colon [q] \to [p]$ and $v \in c^{p+1}$ we can view $f$ as a morphism in $\Delta^{\mathrm{op}}_{\mathrm{inj}, c}(v,w)$ for $w = (v_{f(0)}, v_{f(1)}, \dots, v_{f(q)}) \in c^q$. 
        Hence $f$ induces a map  $\mathfrak{S}_{p}(X_{\blacktriangle}) \to \mathfrak{S}_{q}(X_{\blacktriangle})$. This defines the semi-simplicial space structure on $ \mathfrak{S}_{\bullet}(X_{\blacktriangle}) $.
\end{definition}

\begin{proposition} \label{simplicialization}
    Let $X_{\blacktriangle}$ be a $\Delta^{\mathrm{op}}_{\mathrm{inj}, c}$-space. There is a weak homotopy equivalence 
    $$\hocolim_{\blacktriangle\in \delc^{\mathrm{op}}} X_{\blacktriangle} \xrightarrow{\simeq} \left\Vert \mathfrak{S}_{\bullet}(X_{\blacktriangle})\right\Vert.$$
\end{proposition}

\begin{proof} We defined the homotopy colimit $\hocolim\limits_{\blacktriangle\in \Delta^{\mathrm{op}}_{\mathrm{inj},c}} X_{\blacktriangle}$  to be the geometric realization of the semi-simplicial space $ B_{\bullet}\left(*, \Delta^{\mathrm{op}}_{\mathrm{inj}, c} ,X_{\blacktriangle}\right)$. This semi-simplicial space is naturally isomorphic to $B_{\bullet}\left(*, \Delta^{\mathrm{op}}_{\mathrm{inj}} , \mathfrak{S}_{\bullet}(X_{\blacktriangle})\right)$. The geometric realization $\left\Vert B_{\bullet}\left(*, \Delta^{\mathrm{op}}_{\mathrm{inj}} , \mathfrak{S}_{\bullet}(X_{\blacktriangle})\right) \right\Vert$ is the homotopy colimit 
$\hocolim\limits_{\bullet\in \Delta^{\mathrm{op}}_{\mathrm{inj}}} \mathfrak{S}_{\bullet}(X_{\blacktriangle})$, which is weakly homotopy equivalent to $\Vert \mathfrak{S}_{\bullet}(X_{\blacktriangle}) \Vert$ (see Dugger \cite[Example 9.15]{dug}). 
\end{proof}

Combining \cref{prop: relating canonical res and derived indecomposables}, \cref{RvsL}, and \cref{simplicialization}, we obtain the following corollary. Note the cofibrancy condition of \cref{prop: relating canonical res and derived indecomposables} is automatic for set-valued $\fb$-spaces. 

\begin{corollary} \label{CorIndecomposablesSL} Let $\bM\in \mathrm{Set}^{\fb}$ be a left $\pi_0(\ohur)$-module and let $S$ be a finite set. There is a natural isomorphism of $\fb$-modules $$\widetilde{H}_{i}\Big(\mathrm{Cone} \Big( \Vert \mathfrak{S}_\bullet(L_\blacktriangle(\bM)_S) \Vert  \to {\bN}_S \Big)\Big) \cong H_i\left( Q^{\ohur}_{\mathbb{L}}(\bM)\right)_S.$$
\end{corollary}

Observe that the augmented semi-simplicial set $\mathfrak{S}_\bullet(L_\blacktriangle(\bN)_S)$ admits the following description. The set of $p$-simplices for $p \geq -1$ are given by tuples $(h,v,\eta)$ with $h \colon [p] \hookrightarrow S$, $v \in c^{p+1}$, and $\eta \in \bM_{S\setminus\im(h) }$. In particular, for $S=\bn$, the augmented semi-simplicial set has the form:
$$ \bM_{\bn} \leftarrow \Ind_{\Sigma_n-1}^{\Sigma_n} c \times \bM_{\bf n-1} \leftleftarrows \Ind_{\Sigma_n-2}^{\Sigma_n}  c^2 \times  \bM_{\bf n-2} \ldots$$
The face map $d_i \colon [p-1] \to [p]$ acts via the formula
$$(h,(v_0,\ldots,v_p), \eta)\mapsto (h \circ d_i,(v_0,\ldots,\hat v_i,\ldots,v_p), (v_{i})^{v_{i+1}\cdots v_{p}}\cdot  \eta).$$ 
For $p=0$ the augmentation map is given by the same formula.

We now construct a chain complex $\widetilde{\mathcal{K}}_{*}(S)$, constructed to be the augmented cellular chain complex of $\mathfrak{S}_\bullet(L_\blacktriangle(\pi_0(\ohur))_S)$. 

\begin{notation}
Now consider the case $\bM =\pi_0(\ohur)$. Let $S$ be a finite set and $p \geq -1$.
Let $$\widetilde{\mathcal{K}}_{p}(S) = \bK\left[ \bigsqcup_{\substack{f\colon [p] \hookrightarrow S}}  c^{p+1} \times \pi_0\left(\ohur[, {S \setminus f([p]) }] \right)\right]. $$ 
For $p \geq 1$ the differential is induced by the alternating sum of the face maps, 
\begin{align*}
    \partial = \sum_{i=0}^p (-1)^i(d_i)_* \colon \widetilde{\mathcal{K}}_{p}(S) \longrightarrow \widetilde{\mathcal{K}}_{p-1}(S) 
\end{align*}
and for $p=0$ it is induced by the augmentation. 
\end{notation} 

With this notation, \cref{CorIndecomposablesSL} states the following. 
\begin{corollary} \label{CorIndecomposablesK} There is a natural isomorphism of $\fb$-modules 
\begin{align*}  
  {H}_{i-1}\Big( \widetilde{\mathcal{K}}_{*}(S) \Big)    
\cong  H_i\left( Q^{\ohur}_{\mathbb{L}}\left(\pi_0(\ohur)\right)\right)_S.
\end{align*}
\end{corollary}

By the corollary, to show vanishing of these derived indecomposables, it suffices to show the homology groups  $  {H}_{i}\left( \widetilde{\mathcal{K}}_{*}(S) \right) $ vanish in a range. The complex $\widetilde{\mathcal{K}}_{*}(S)$ is an ordered analogue of the ``Koszul-like complex'' $\mathcal{K}_{*}(H_0(\hur))$ of Ellenberg--Venkatesh--Westerland \cite[Section 4.1]{MR3488737}, and in the following section we adapt their proof of vanishing \cite[Lemma 4.11]{MR3488737}. 

\begin{remark}
In \cite[Theorem 6.2]{MR4224644}, Randal-Williams used ideas related to ``scanning'' to compute the associative Koszul dual of $\hur$ (see also Landesman--Levy \cite[Appendix A]{landesman2025cohenlenstramomentsfunctionfields} as well as Landesman--Levy \cite[Proof of Lemma 3.2.8]{landesman2025homologicalstabilityhurwitzspaces}). From this calculation, Randal-Williams deduced that Ellenberg--Venkatesh--Westerland's Koszul-like complex computes derived $C_*(\hur)$-module indecoposables. It seems plausible that Randal-Williams' strategy can be adapted to give an alternative proof of the main results of this section. 

 Using generalizations of Fox--Neuwirth--Fuks cells, Ellenberg--Tran--Westerland proved a result which is Koszul dual to Randal-Williams' calculation of the Koszul dual of $\hur$  
 \cite[Theorem 1.3]{ellenberg2023foxneuwirthfukscellsquantumshuffle}. Likely one could adapt these techniques to give a third proof of the results of this section.

%We chose to adapt the arguments of \cite{MR4767884} due to our love of ad hoc constructions, especially those involving Moore parameters. 

\end{remark}

\section{Vanishing for the \texorpdfstring{$\ohur$-module}{OHur-module} cells of \texorpdfstring{$\pi_0(\ohur)$}{the connected components of ordered Hurwitz space}}

The goal of this section is to show $H_i\left( Q^{\ohur}_{\mathbb{L}}(\pi_0(\ohur))\right)_{\bn} \cong 0$ for $n$ sufficiently large relative to $i$,  provided $c$ satsifies the following non-splitting property.

\begin{definition} \label{defNonSplit}
    Let $G$ be a finite group and $c$ a conjugacy class of $G$.  We say the pair $(G, c)$ has the \emph{non-splitting property} if $c$ generates $G$ and for every subgroup $H$ of $G$, the intersection $c\cap H$ is either empty or a conjugacy class of $H$.
\end{definition}

We now recall the definition of a quandle.

\begin{definition}
    A \emph{quandle} is a set $X$ with a binary operation $x^{y}$ for $x, y\in X$ such that for every $y\in X$, the function $x\mapsto x^{y}$ is a bijection from $X$ to $X$, $x^{x}=x$ for all $x\in X$, and for every $x, y, z\in X$, we have $(z^{x})^{y}= (z^{y})^{x^{y}}$. 
\end{definition}

An example of a quandle is a conjugacy invariant subset $c$ of $G$ with the binary operation given by conjugation. 

\begin{definition}
    Let $X$ be a finite quandle. The (directed unlabeled) \emph{Schreier graph} of $X$ is the graph whose set of vertices is $X$ and whose set of edges is $\{(x, x^{y})\colon (x, y)\in X\times X\}$. We say that $X$ is \emph{connected} if its Schreier graph is connected.
\end{definition}

The following (well-known) lemma provides examples of connected quandles. 

\begin{lemma} \label{LemNonSplittingConnectedQuandle}
    Let $G$ be a group and $c \in G$ a  conjugation-invariant subset. Suppose that $c$ is a single conjugacy class of $G$ and $c$ generates $G$. Then $c$ is connected when viewed as a quandle with conjugation. In particular this holds whenever $c$ has the nonsplitting property. 
\end{lemma}
\begin{proof}
Let $c_0, c_1 \in c$. Since $c$ is a conjugacy class, there exists some $g \in G$ such that $c_0 = c_1^g$. Since $c$ generates $G$, we can express $g$ as a product $g=a_1a_2 \dots a_{\ell}$ for $a_i \in c$. Conjugation by each of the elements $a_i$ define a sequence of edges in the Shreier graph $(c_0, c_0^{a_1})$, $(c_0^{a_1}, c_0^{a_1 a_2})$, \dots, $(c_0^{a_1 a_2 \dots a_{\ell-1}}, c_1)$ from $c_0$ to $c_1$. 
\end{proof}

%\begin{definition}
  %  Let $X$ be a finite quandle. We say that a  subset $X_{0}$ of $X$ is a \emph{subquandle} if for all $x, y\in X_{0}$, we have $x^{y}\in X_{0}$.  We say that elements $x_{1},\ldots, x_{n}\in X$ \emph{generate} $X$ if there is no proper subquande of $X$ containing $x_{1},\ldots, x_{n}$.
%\end{definition}
Let $X$ be a quandle. The braid group $\Br_n$ acts on $X^n$ by using the same formula from \cref{eq: brd grp action}. The following theorem is a reformulation of a special case of Shusterman \cite[Theorem 2.4]{MR4666043}.
\begin{theorem}[{\cite[Theorem 2.4]{MR4666043}}]\label{MarkTheorem}
    Let $X$ be a connected finite quandle. 
    Then there is a natural number $N_X$, depending on $X$, such that for all $n\geq N_X$ and all $(x_{1},\ldots, x_{n})\in X^{n}$, with $x_{1},\ldots, x_{n}$ generating $X$,
    %Let $n$ be a  sufficiently large positive integer. any number larger than a certain quantity, and let $(x_{1},\ldots, x_{n})\in X^{n}$ be an $n$-tuple of elements in $X$ such that $x_{1},\ldots, x_{n}$ generate $X$. Then 
    $$\PBr_{n}\cdot (x_{1},\ldots, x_{n})= \Br_{n}\cdot (x_{1},\ldots, x_{n}).$$
\end{theorem}

\begin{corollary} \label{MarksConstantNc} Assume $(G,c)$ satisfies the non-splitting property. There exists $N_c$ such that for all $n \geq N_c$, the quotient map $\ohur(\bn) \to \hur(n)$ induces a bijection $\pi_0(\ohur(\bn)) \to \pi_0(\hur(n))$. 
\end{corollary}
\begin{proof}  
Consider $v=(v_{1},\ldots, v_{n})\in c^{n}$. Let $V=\{v_1, \dots, v_n\}$ be the underlying subset of $c$. Let $H_V\subseteq G$ be the subgroup generated (as a subgroup) by $V$. It is not difficult to verify that the non-splitting assumption on $c$ implies that $V$ generates (as a quandle) the quandle $c\cap H_V$.

The non-splitting property implies that $c\cap H_V$ is a single conjugacy class of $H_V$, hence $c\cap H_V$ is a connected quandle by \cref{LemNonSplittingConnectedQuandle}. For $n \geq N(H_V)$,  \cref{MarkTheorem} implies that the $\PBr_{n}$-orbit of $(v_{1},\ldots, v_{n})$ is the same as its $\Br_{n}$-orbit. Let $$N_c=\max_{V \subseteq c} (H_V).$$ For $n \geq N_c$,  the $\PBr_n$ and $\Br_n$ orbits of $c^n$ agree: $$c^{n}\slash \PBr_{n}=c^{n}\slash \Br_{n}.\mbox{\qedhere}$$
\end{proof}

Fix $n \in \N$ and $p \geq -1$.  Recall the complex $\widetilde{\mathcal{K}}_{*}(\bn)$ has $p$-chains
\begin{align*}
\widetilde{\mathcal{K}}_{p}(\bn) & = \bK\left[ \bigsqcup_{\substack{f\colon [p] \hookrightarrow \bn}}  c^{p+1} \times \pi_0\left(\ohur[, {\bn \setminus f([p]) }] \right)\right] \\ 
& \cong \Ind_{1\times \Sigma_{n-p-1}}^{\Sigma_n}  \bK\left[ c^{p+1}\right] \otimes_{\bK} \bK\left[c^{n-p-1}/\PBr_{n-p-1} \right] \\ 
& \cong \Ind_{\Sigma_{n-p-1}}^{\Sigma_n} \bK\left[ c^{n}/\PBr_{n-p-1} \right].
\end{align*}

\cref{MarksConstantNc} therefore implies the following identification. 

\begin{corollary} Assume $(G,c)$ satisfies the non-splitting property. Let  $N_c$ be as in \cref{MarksConstantNc}.     For $p \leq n-N_c-1$ there are isomorphisms
    $$\widetilde{\mathcal{K}}_{p}(\bn) 
 \cong \Ind_{\Sigma_{n-p-1}}^{\Sigma_n} \bK \left[ c^{n}/\Br_{n-p-1} \right]$$
and for $p=n-N_c$ there is a surjection 
    $$\widetilde{\mathcal{K}}_{p}(\bn)  \twoheadrightarrow \Ind_{\Sigma_{n-p-1}}^{\Sigma_n} \bK \left[ c^{n}/\Br_{n-p-1} \right].$$
\end{corollary}
The following lemma is a reformulation of \cite[Lemma 3.5]{MR3488737}.
\begin{lemma}[{\cite[Lemma 3.5]{MR3488737}}]\label{lem: stability for pi0 Hur}
    %Let $\bK$ be a field of characteristic prime to $|G|$.
    Suppose that $(G, c)$ has the non-splitting property and $|G|$ is a unit in the coefficient field $\bK$. Then there are natural numbers $N, N_{0}$, with $N_{0}\geq N\geq 1$, and an element $U \in H_0(\hur[,N])$ in the center of the ring $\bigoplus_{n\in \N}H_0(\hur[,n])$ such that the multiplication map
    %$$U\cdot -\colon H_0(\hur_{,n})\to H_0(\hur_{,n+N})$$
    \begin{align*}
        -\cdot U\colon H_0(\hur[,n])&\to H_0(\hur[,n+N])
        %-\cdot U \colon R_{n}&\to R_{n+N}
    \end{align*}
    is an isomorphism for all $n\geq N_{0}$.
\end{lemma}

Note that the natural map  $ \ohur[,\bn]  \longrightarrow \hur[,n] $ induces a surjection
\begin{align*}
   H_0 \left( \ohur[,\bn] \right) & \longrightarrow  H_0 \left( \hur[,n]  \right)
\end{align*}
since $\hur[,n] \cong \ohur[,\bn] /\Sigma_n$.

\begin{notation} \label{NotationUTwiddle}
Let $\widetilde{U} \in H_0(\ohur[,\bN])$ be a lift of $U \in H_0(\ohur[,N])$.
\end{notation}

The following corollary follows from \cref{MarksConstantNc} and \cref{lem: stability for pi0 Hur}. 

\begin{corollary} \label{CorUIsoOnOHur} Suppose that $(G,c)$ has the non-splitting property and $|G|$ is a unit in the coefficient field $\bK$.     Let $N_c, N_0$, and $N$ be as in \cref{MarksConstantNc} and \cref{lem: stability for pi0 Hur}, and $\widetilde{U}$ be as in Notation \ref{NotationUTwiddle}. Then
    \begin{align*}
       - \cdot \widetilde{U} \colon H_0(\ohur[,S])&\to H_0(\ohur[,S \sqcup \bN])
        %-\cdot U \colon R_{n}&\to R_{n+N}
    \end{align*}
    is an isomorphism for all finite sets $S$ with $|S|\geq \max(N_{0}, N_c)$.
\end{corollary} 

An element $\alpha \in \ohur[,T]$, defines a map $$- \cdot \alpha\colon \widetilde{\mathcal{K}}_p(S) \to \widetilde{\mathcal{K}}_p(S \sqcup T)$$ induced  by the right multiplication by $\alpha$ on the groups
$$- \cdot \alpha\colon H_0\left(\ohur[,S\setminus \im(f)]\right) \to H_0\left(\ohur[,(S\setminus \im(f)) \sqcup T]\right)$$ for each injective map $f\colon [p] \to S$. Since the differentials of the complex $\widetilde{\mathcal{K}}_*(S)$ are defined via left multiplication on the factors $H_0\left(\ohur[,S\setminus \im(f)]\right)$, the differentials commute with the action of $\alpha$, and the map induced by $\alpha$ on $\widetilde{\mathcal{K}}_*(S)$ is a chain map. 

 We will now prove that the induced map $\alpha$ on  $\widetilde{\mathcal{K}}_*(S)$ is the zero map provided $|T|>0$. The analogous proof of  Ellenberg--Venkatesh--Westerland \cite[Lemma 4.11]{MR3488737} readily adapts to this setting.

\begin{lemma} \label{ChainHomotopy-RightMultVanishes} Let $g \in c$. We conflate $g \in c$ with the corresponding element $g  \in H_0(\ohur[, \bf 1])$.  Then right multiplication by $g$ 
\begin{align*}
    H_0(\ohur[, \bN]) & \longrightarrow H_0(\ohur[, {\bf  n+1}]) \\ 
    s & \longmapsto s \cdot g
\end{align*}
induces the zero map on the homology groups $$ H_*(\widetilde{\mathcal{K}}_*(\bn)) \to H_*(\widetilde{\mathcal{K}}_*({\bf n+1 })).$$  
\end{lemma}
\begin{proof}
  We will use the following notation.  We represent a coset in $\Sigma_n / \Sigma_{n-p-1}$ via a word $(w_0, \dots, w_p)$ of $(p+1)$ distinct letters in $\bn$.  For $s \in \pi_0(\ohur[,\bf n-p-1])$ corresponding to the orbit $$\PBr_{n-p-1}\cdot (c_{p+1}, \dots, c_n) \in c^{n-p-1}/\PBr_{n-p-1},$$ let $\partial s= c_{p+1} \cdots c_n$ denote the (well-defined) product in $G$. 

Then, as in \cite[Proof of Lemma 4.11]{MR3488737}, we define a chain homotopy $H_g: \widetilde{\mathcal{K}}_{p}(\bn) \to \widetilde{\mathcal{K}}_{p}(\bf{n+1})  $ via the map 
\begin{align*}
   H_g \colon  \bigoplus_{\Sigma_n/ \Sigma_{n-p-1}}  \bK\left[ c^{p+1}\right] \otimes_{\bK} \bK\left[ c^{n-p-1}/\PBr_{n-p-1} \right] \longrightarrow&  \bigoplus_{\Sigma_{n+1}/ \Sigma_{n-p}}  \bK\left[ c^{p+1}\right] \otimes_{\bK} \bK\left[ c^{n-p}/\PBr_{n-p} \right]\\
    \big( (w_0, \dots, w_p); (g_{0},\dots, g_p); s\big) \longmapsto& \left( (n+1, w_0, \dots, w_p); \left(g^{(g_0 \cdots g_q  \partial s)^{-1}}, g_{0},\dots, g_p\right); s \right)
\end{align*}
Just as in \cite[Proof of Lemma 4.11]{MR3488737},
     a routine calculation shows that $H_g$ defines a chain homotopy between the zero map and the map
     $$ \big( (w_0, \dots, w_p); (g_{0},\dots, g_p); s\big) \longmapsto \big((w_0, \dots w_p); (g_0, \dots,  g_p); g^{(g_0 \dots g_p \partial s)^{-1}} \cdot s\big) .$$ 
     We claim that (up to post-composition with a chain isomorphism) this agrees with the map 
      \begin{equation} \label{EqnMultByg}
 \big( (w_0, \dots, w_p); (g_{0},\dots, g_p); s\big) \longmapsto \big((w_0, \dots w_p); (g_0, \dots,  g_p);  s \cdot g\big) .
      \end{equation}
     Again as in \cite[Proof of Lemma 4.11]{MR3488737},   there exists a braid  $b \in \Br_{n-p}$ that maps   $s \cdot g$ to $g^{(g_0 \dots g_p \partial s)^{-1}} \cdot s$ in $c^{n-p} / \PBr_{n-p}$. Note  
     that $b$ does not depend on $s,g_0,\ldots,g_p$ or $g$ and only depends on $p$. Let $\sigma$ be the image of $b$ under the map $\Br_{n-p} / \PBr_{n-p} \cong \Sigma_{n-p} \hookrightarrow \Sigma_{n+1}$, so  $s \cdot g$ and $g^{(g_0 \dots g_p \partial s)^{-1}}  \cdot s$ differ by the action of the permutation $\sigma$. Thus $ (\sigma)_* \circ H_g$ gives the desired chain homotopy between the zero map and the map of Formula  (\ref{EqnMultByg}). 
\end{proof}

We will prove our vanishing result by studying the spectral sequences associated to the following double complex. 

\begin{definition} Suppose that $(G,c)$ has the non-splitting property and $|G|$ is a unit in the coefficient field $\bK$. Let $N$ be as in \cref{lem: stability for pi0 Hur}, and $\widetilde{U}$ be as in Notation \ref{NotationUTwiddle}.  Define, for each finite set $S$, a double complex for $p, q \geq -1$,  
$$\mathcal{C}_{p,q}(S) = \bigoplus_{ \left(g \sqcup f\right)\colon \big( (\bN \times [p]) \sqcup [q] \big) \hookrightarrow S} \bK[ c^{q+1}] \otimes_{\bK} H_0(\ohur(S \setminus \im(g \sqcup f))). $$ 
The vertical differential $\partial_V$ and horizontal differential $\partial_H$ are defined as follows. 
The differential $\partial_V \colon  \mathcal{C}_{p,q}(S) \to  \mathcal{C}_{p,q-1}(S)$ is defined so that the natural isomorphism of $\fb$-$\bK$-modules
$$\mathcal{C}_{p,q}(S) \cong \bigoplus_{ g\colon (\bN \times [p]) \hookrightarrow S} \widetilde{\mathcal{K}}_q(S \setminus \im(g))$$
is an isomorphism of chain complexes, 
\begin{equation} \label{Ccols}
\mathcal{C}_{p,*}(S) \cong \bigoplus_{ g\colon (\bN \times [p]) \hookrightarrow S} \widetilde{\mathcal{K}}_*(S \setminus \im(g)).
\end{equation}
The horizontal differential $\partial_H \colon \mathcal{C}_{p,q}(S) \to \mathcal{C}_{p-1,q}(S)$ is defined by the alternating sum of the maps induced by right multiplication by $\widetilde{U}$. 
\end{definition}
The two differentials $\partial_V$ and $\partial_H$ commute because $\partial_V$ is defined by left multiplication and $\partial_H$ is defined by right multiplication on $H_0(\ohur)$. 

We begin by analyzing the rows of the double complex. Observe that
\begin{equation} \label{Crows} \mathcal{C}_{*,q}(S) =
\bigoplus_{f \colon [q] \to S } \bK[c^{q+1}]\otimes_{\bK} \mathcal{C}_{*,-1}( S \setminus \im(f)) . 
\end{equation}
Thus to prove that the homology of the rows vanishes in a range, it suffices to prove that $\mathcal{C}_{*,{-1}}(S)$ is highly acyclic. We do so by relating it to the following simplicial complex. 

\begin{definition}
Let $X_{\bullet}(S, T)$ be the semi-simplicial set whose  $p$-simplices are the set of injective maps $T \times [p] \to S$, and whose face maps are given by precomposition with $d_i \colon [p-1] \to [p]$. 
\end{definition}

The following result is surely well-known, but for completeness we include an argument. 

\begin{lemma} \label{XBulletConnectivity}
    $X_{\bullet}(S, T)$  is $\left( \frac{|S|}{2|T|} - \frac32\right)$-connected. 
\end{lemma}

We will use the following version of a simplicial ``badness argument''; this formulation appears in Galatius--Randal-Williams \cite[Proposition 2.5.]{GRW}. 

\begin{proposition}[{\cite[Proposition 2.5.]{GRW}}]\label{GRW-Badness} Let $X$ be a simplicial complex, and $Y \subset X$ be a full subcomplex.
Let $d$ be an integer with the property that for each $p$-simplex $\sigma \subset X$ having no
vertex in $Y$, the complex $Y \cap \mathrm{Lk}_X (\sigma)$ is $(d - p - 1)$-connected. Then the inclusion
$\vert X\vert\to \vert Y\vert$ is $d$-connected. 
%$\vert Y\vert ֒\to \vert X\vert$ is $d$-connected. 
\end{proposition} 

\begin{proof}[Proof of Lemma \ref{XBulletConnectivity}]

Define  $Z(S, T)$ to be a simplicial complex where the vertices are injective maps $f_0 \colon T \hookrightarrow S$, and a collection of vertices $f_0, \dots, f_p$ span a $p$-simplex precisely when their images are disjoint subsets of $S$. 

The realization of the semi-simplicial set   $X_{\bullet}(S, T)$ is the \emph{ordered complex} associated to $Z(S, T)$ (see for example Hatcher--Vogtmann \cite[Section 2.5]{HV-tethers}) so it suffices to show that $ $ is weakly Cohen-Macaulay at level $\left( \frac{|S|}{2|T|} - \frac12\right)$ (e.g. Hatcher--Vogtmann \cite[Proposition 2.10]{HV-tethers}).  In other words, our goal is to show that $Z(S, T)$ is $\left( \frac{|S|}{2|T|} - \frac32\right)$-connected and that the link of any $p$-simplex in $Z(S,T)$ is $\left( \frac{|S|}{2|T|} - p -  \frac52\right)$-connected. We will prove this by adapting the proof of Kupers--Lemann--Malkiewich--Miller--Sroka
\cite[Lemma 3.10 i]{kupers2024scissorsautomorphismgroupshomology}.

    Fix a finite set $S$. We proceed by induction on $|S|$.
    The statement holds vacuously when $|S|<|T|$, so fix a finite set $S$ and suppose that $|S| \geq |T|$. Assume by induction that the result holds for all sets of cardinality smaller than $S$.    Given a nonempty $p$-simplex $(f_0 , \dots, f_p) \in Z(S,T)$, its link is isomorphic to the complex $Z( (S \setminus \bigcup_i f_i(T)), T)$, which by inductive hypothesis has connectivity at least 
    $$\left(\frac{|S|-(p+1)|T|}{2|T|} - \frac32\right)  =  \left(\frac{|S|}{2|T|} -\frac{p}{2} - \frac52\right) \geq   \left(\frac{|S|}{2|T|} - {p} - \frac52\right).$$ It remains to check that $Z(S,T)$ is $\left( \frac{|S|}{2|T|} - \frac32\right)$-connected.  Because $|S| \geq |T|$, there exists at least one vertex $f_0 \colon T \hookrightarrow S$. The inclusion of its link $\mathrm{Lk}_{Z(S,T)}(f_0)$ in  $Z(S,T)$ is nullhomotopic---the vertex $f_0$ is a cone point---so to complete the argument we will use the criteria in Lemma \ref{GRW-Badness} to prove that the inclusion of the link  is $\left( \frac{|S|}{2|T|} - \frac32\right)$-connected. Let $\sigma=(g_0, \dots, g_p)$ be any $p$-simplex of $Z(S,T)$ with  no vertex in $\mathrm{Lk}_{Z(S,T)}(f_0)$. In other words,  $\sigma=(g_0, \dots, g_p)$ is a simplex such that each vertex $g_i$ intersects the image of $f_0$ in $S$ nontrivially.  Then there is an  isomorphism
    $$\mathrm{Lk}_{Z(S,T)}(f_{0})\cap \mathrm{Lk}_{Z(S,T)}(g_0, \dots, g_p) \quad \cong \quad Z\left(S\setminus \Big(\im(f_{0}) \cup \bigcup_{i=0}^{p}\im(g_{i})\Big), T\right).$$
    which by inductive hypothesis has connectivity at least 
    $$\left(\frac{|S|-|T|-(p+1)(|T|-1)}{2|T|} - \frac32\right) \geq \left(\frac{|S|}{2|T|} - p - \frac52\right).$$ The result follows by \cref{GRW-Badness}. \end{proof} 
    
    %When $|S|=|T|$, the complex $Z_{\bullet}(S, T)$ is $(-1)$-connected since $Z_{0}(S, T)$ is non-empty. Without loss of generality, assume that $|T|< |S|$. We will prove this lemma by induction. If $|S|=2$, then $|T|=1$ and $|X(S, T)|$ is path-connected. By induction, suppose that the lemma holds for all $S$ with $|S|< N$. We will show that the lemma holds for $|S|=N$. Fix a map $f_{0}\colon T\to S$. Then the link $\lk(f_{0})$ in $X(S, T)$ is equal to  $X(S, T\setminus\im(f_{0}))$. Since the inclusion $$i\colon \lk(f_{0})\to Z(S, T)$$ is null-homotopic, to show that $Z(S, T)$ has the desired connectivity, it suffices to show that $i$ is $\left( \frac{|S|}{2|T|} - \frac32\right)$-connected. 
   % For any $p$-simplex $\sigma=(g_{0},\ldots, g_{p})$ having no vertices in $\lk(f_{0})$, we have $$\lk(f_{0})\cap \lk(\sigma)= X\big(S\setminus (\im(f_{0})\cup_{i=0}^{p}\im(g_{i})), T\big).$$
   % Since $$|S\setminus (\im(f_{0})\cup_{i=0}^{p}\im(g_{i}))|= |S|-(p+1)|T|<|S|,$$ by induction $\lk(f_{0})\cap \lk(\sigma)$ is $\left( \frac{|S|-(p+1)|T|}{2|T|} - \frac32\right)$-connected. Therefore, $\lk(f_{0})\cap \lk(\sigma)$ is $\big(\big( \frac{|S|}{2|T|} - \frac32\big)-p-1\big)$-connected since
 %   $$\left( \frac{|S|-(p+1)|T|}{2|T|} - \frac32\right)=\left( \frac{|S|}{2|T|} - \frac32\right)- \frac{(p+1)}{2}> \left( \frac{|S|}{2|T|} - \frac32\right)- p-1.$$

Given a semi-simplicial set $X_{\bullet}$, let $\widetilde{C}_*(X_{\bullet})$ denote the cellular chain complex supported in degrees $* \geq -1$ that computes reduced homology of the geometric realization $\Vert X_{\bullet}\Vert$. Concretely, $\widetilde{C}_{-1}(X_{\bullet}) \cong \bK$ and $\widetilde{C}_p(X_{\bullet})$ is the free $\bK$-module on the set $X_p$ for $p \geq 0$. The differential $\partial : C_p(X_\bullet) \to C_{p-1}(X_\bullet)$ is given by the alternating sum of the maps induced by the facemaps for $p \geq  1$ and is induced by the augmentation for $p=0$.

\begin{lemma} \label{LemmaCvsX}  Suppose that $(G,c)$ has the non-splitting property and $|G|$ is a unit in the coefficient field $\bK$.     Let $N_c, N_0$, and $N$ be as in \cref{MarksConstantNc} and \cref{lem: stability for pi0 Hur}, and $\widetilde{U}$ be as in Notation \ref{NotationUTwiddle}. Fix a finite set $S$. 
     Let $M$ be a fixed number such that $M \geq \max(N_0, N_c)$ and is congruent to $|S|$ modulo $N$. For $* \leq \frac{|S|- \max(N_0,N_c)}{N} - 1$, there is an isomorphism of  chain complexes 
    $$\mathcal{C}_{*,-1}( S) \cong H_0(\hur[,M]) \otimes_{\bK} \widetilde{C}_*(X_{\bullet}(S,\bN)).$$ 
\end{lemma}
\begin{proof}
    The result holds because, in the range $p \leq \frac{|S|- \max(N_0,N_c)}{N} - 1$, the quotient map $\ohur \to \hur$ induces an isomorphism
$$ H_0(\ohur(S \setminus f([p] \times {\bN}))) \cong   H_0(\hur[, |S|-(p+1)N])$$ 
and right multiplication by ${U}$ induces an isomorphism
$$  H_0(\hur[, |S|-(p+1)N])  \cong H_0(\hur[, M]). $$  These isomorphisms are compatible with the differentials. 
\end{proof}

 Let $\mathcal{C}_*(S)$ be the total complex of this double complex. 

\begin{lemma} \label{TotalComplexConnectivity}  Suppose that $(G,c)$ has the non-splitting property and $|G|$ is a unit in the coefficient field $\bK$.     Let $N_c, N_0$, and $N$ be as in \cref{MarksConstantNc} and \cref{lem: stability for pi0 Hur}.  The homology of total complex $H_i(\mathcal{C}_*(S))$ vanishes for $$ i \leq \min\left( \left( \frac{|S|}{2N} - \frac32\right)  \;, \;  \left(\frac{|S|- \max(N_0,N_c)}{N} - 2 \right) \right)-1.$$ 
\end{lemma}
\begin{proof} 
By \cref{LemmaCvsX},
 $$\mathcal{C}_{*,-1}( S) \cong H_0(\hur)_{M} \otimes_{\bK} \widetilde{C}_*(X_{\bullet}(S,\bN))$$  whenever  $* \leq \frac{|S|- \max(N_0,N_c)}{N} - 1$. Thus whenever $i \leq \frac{|S|- \max(N_0,N_c)}{N} - 2$ the homology groups coincide,  
 $$H_i\left(\mathcal{C}_{*,-1}( S) \right) \cong H_0(\hur)_{M} \otimes_{\bK} \widetilde{H}_i(X_{\bullet}(S,\bN)).$$
By \cref{XBulletConnectivity}, the homology of
$\widetilde{H}_i(X_{\bullet}(S,\bN))$ vanishes whenever $ i \leq \left( \frac{|S|}{2N} - \frac32\right)$.

By Equation (\ref{Crows}),  $$H_i \left(\mathcal{C}_{*,q}(S) \right) =
\bigoplus_{f \colon [q] \hookrightarrow S } \bK[c^{q+1}]\otimes_{\bK} H_i\left( \mathcal{C}_{*,-1}( S \setminus \im(f)) \right)  $$
therefore vanishes whenever 
$$ i \leq \min\left( \left( \frac{|S|-q-1}{2N} - \frac32\right)  \;, \;  \left(\frac{(|S|-q-1)- \max(N_0,N_c)}{N} - 2 \right) \right).$$
Thus the homology of the total complex vanishes when 
$$ i \leq \min\left( \left( \frac{|S|}{2N} - \frac32\right)  \;, \;  \left(\frac{|S|- \max(N_0,N_c)}{N} - 2 \right) \right)-1. \mbox{\qedhere}$$ 
\end{proof}

\begin{theorem} \label{OhurVSpi0} Suppose that $(G,c)$ has the non-splitting property and $|G|$ is a unit in the coefficient field $\bK$.      
    There exist constants $\aleph$ and $\beth$ (depending on $c$ as a quandle) such that 
    $$H_i\left( Q^{\ohur}_{\mathbb{L}}(\pi_0(\ohur))\right)_{S} \cong 0 \qquad \text{for $|S| \geq \aleph i +  \beth$}.$$  
\end{theorem}

\begin{proof}

Recall by \cref{CorIndecomposablesK} there is an isomorphism of $\fb$-modules 
\begin{align*}  
  {H}_{i-1}\Big( \widetilde{\mathcal{K}}_{*}(S) \Big)    
\cong  H_i\left( Q^{\ohur}_{\mathbb{L}}\left(\pi_0(\ohur)\right)\right)_S.
\end{align*} so it suffices to prove that the homology groups $  {H}_{*}\Big( \widetilde{\mathcal{K}}_{*}(S) \Big) $  vanish in a linear range. Let $\aleph=2N$ and pick $\beth$ so that  $ {H}_{i-2}( \mathcal{C}_*(S)) = 0$ whenever $|S| \geq \aleph i +  \beth$; this is possible by \cref{TotalComplexConnectivity}.

We will proceed by induction on $i$. For $i=-1$, for any set $S$ the group ${H}_{i-1}\Big( \widetilde{\mathcal{K}}_{*}(S) \Big)={H}_{-2}\Big( \widetilde{\mathcal{K}}_{*}(S) \Big)$ vanishes vacuously. Fix $i$ and assume by induction that for all $j<i$ that $${H}_{j-1}\Big( \widetilde{\mathcal{K}}_{*}(S) \Big) \cong 0 \qquad \text{for all $|S| \geq \aleph j +  \beth$} .$$
Fix a finite set $S$ with $|S| \geq \aleph i +  \beth$. Let $(E^r_{p,q}(S), \partial^r)$ be the spectral sequence of the double complex $\mathcal{C}_{*,*}(S)$ such that  $\partial^0$ is the vertical differential $\partial_{V}$. 
We have chosen $\aleph$ and  $\beth$ so that $E^{\infty}_{p,q}(S) \cong 0$ for all $|S| \geq \aleph (p+q+2) +  \beth$. 

Observe that $$E^1_{-1, i-1}(S) \cong {H}_{i-1}\Big( \widetilde{\mathcal{K}}_{*}(S) \Big).$$ By construction the differential $\partial^1 \colon  E^1_{0, i-1}(S) \to E^1_{-1, i-1}(S)$ is induced by multiplication by $\widetilde{U}$ (Notation \ref{NotationUTwiddle}). This induced map is zero on homology by \cref{ChainHomotopy-RightMultVanishes}, hence  $$E^1_{-1, i-1}(S) \cong E^2_{-1, i-1}(S).$$ The induction hypothesis, Equation (\ref{Ccols}), and the fact that $\aleph = 2N$ imply that the the groups $E^1_{p, i-1-p}(S)$ vanish for $p \geq 1$. This observation rules out any nonzero higher differentials $\partial^r$ for $r \geq 2$ with codomain $E^r_{-1, i-1}(S)$. Thus 
$$ H_{i-1}( \widetilde{\mathcal{K}}_*(S)) \cong E^1_{-1, i-1}(S) \cong E^{\infty}_{-1, i-1}(S) \cong 0 $$ where the vanishing follows from our assumption that $|S| \geq \aleph i +  \beth$. 
\end{proof} 

We end with an explicit description of the derived indecomposables for $i=0$ and $i=1$. 

\begin{lemma}\label{lem:deg01indec}
    %Let $\bK$ be a field.
For all finite sets $S$, 
     $$H_{0}(Q^{\ohur}_{\mathbb{L}}\big(\pi_0(\ohur ))\big)_S \cong \begin{cases} \bK, & \text{if } |S|=0,\\
    0, & \text{if } |S|>0,
    \end{cases}$$
    %$H_{S, 0}(Q^{(\ohur)_{\bK}}_{\mathbb{L}}(H_{0}(\ohur ;\bK)))=0$ for $|S|>0$ 
  $$H_{1}\big(Q^{\ohur}_{\mathbb{L}}(\pi_0(\ohur ))\big)_S \cong 0. $$ 
    %\footnote{might not need this sentence}
    When $S=\varnothing$, 
    $$ H_{i}\left(Q^{\ohur}_{\mathbb{L}}(\pi_{0}(\ohur))\right)_{\varnothing} = \begin{cases} \bK, & \text{if } i=0,\\
    0, & \text{if } i>0.
    \end{cases} $$
\end{lemma}
\begin{proof}
    %The proof for the case $d=0$ is the same argument as in the end of the proof of \cite[Lemma 7.3]{MR4224644}.
    %\footnote{cite lemma 7.3 of Oscar's Hurwtiz paper. Also, might not need this} 
    %We have that $H_{*, 0}(Q^{(\ohur)_{\bK}}_{\mathbb{L}}(-))$ preserves surjections, there is a surjection $(\ohur)_{\bK}\to H_{0}(\ohur)$, and  $H_{*, 0}(Q^{(\ohur)_{\bK}}_{\mathbb{L}}((\ohur)_{\bK}))$ is one-dimensional and supported in grading $|S|=0$.\footnote{might want to make into an equation environment}
    The proof is the same argument as in the first paragraph of Miller--Patzt--Petersen--Randal-Williams \cite[Lemma 2.8]{miller2024uniformtwistedhomologicalstability}.
    Let $$C\pi\colonequals\mathrm{Cone}\Big(\ohur \xrightarrow{\pi} \pi_0\left(\ohur\right)\Big)\in \ohur\text{-mod}$$ denote the mapping cone, formed in the category of left $\ohur$-modules. The map
    $\pi_*$ induced by $\pi$ on homology is an isomorphism when $i=0$ and it is a surjection when $i=1$. Therefore, \begin{equation} \label{H0H1Vanishes} \widetilde{H}_{i}(C\pi)_S=0 \text{ for all finite sets $S$ and $i \leq 1$.} \end{equation} Note moreover that $\pi_S$ is a homotopy equivalence when $|S| \leq 1$, so $\widetilde{H}_i(C\pi)_S = 0$ for $|S| \leq 1$. 
    
By Galatius--Kupers--Randal-Williams \cite[Corollary 11.14(ii)]{e2cellsI}, 
\begin{equation} \label{OscarOperad} \bK\otimes_{H_0(\ohur)}\widetilde{H}_{i}(C\pi)_S\xrightarrow{\cong} H_{i}\left(Q^{\ohur}_{\mathbb{L}}(C\pi)\right)_S \quad \text{ for $i\leq 2$ and all $S$. }
\end{equation} 
To apply \cite[Corollary 11.14(ii)]{e2cellsI}, consider $\fb$-($\fb$-space) $\mathcal{O}$ with values 
$$ \mathcal{O}(S,T) = \left\{ \begin{array}{ll} 
\varnothing, & |S| \neq 1 \\ 
\ohur[, T], & |S|=1. \end{array} \right. 
$$ 
Recall that the category of $\fb$-spaces is a symmetric monoidal category with Day convolution. The category of $\fb$-($\fb$-spaces) is monoidal with  composition product. Since $\ohur[]$ is a monoid in $\fb$-spaces, $\mathcal{O}$ is a monoid in $\fb$-($\fb$-spaces), equivalently, it is an operad in $\fb$-spaces with Day convolution. This is the operad to which we apply \cite[Corollary 11.14(ii)]{e2cellsI}. 

We deduce from Equations (\ref{H0H1Vanishes}) and (\ref{OscarOperad}) that
$$H_{i}\left(Q^{\ohur}_{\mathbb{L}}(C\pi)\right)_S=0 \quad \text{ for all finite sets $S$ and $i\leq 1$.}$$ 
Observe that, when $S = \varnothing$,  $$H_i\left(Q^{\ohur}_{\mathbb{L}}(C\pi)\right)_{\varnothing} = 0 \quad \text{ for all $i$}$$ since $\widetilde{H}_{i}(C\pi)_{\varnothing} \cong 0$ for all $i$. Applying the long exact sequence on $H_{*}(Q^{\ohur}_{\mathbb{L}}(-))$ to $$\ohur\to \pi_0\left(\ohur\right)\to C\pi$$ yields
    \begin{equation}\label{eq: low deg calc1}
H_{0}\left(Q^{\ohur}_{\mathbb{L}}\left(\pi_0\left(\ohur\right)\right)\right)_S
\cong H_{0}\left(Q^{\ohur}_{\mathbb{L}}\left(\ohur\right)\right)_S
\cong \begin{cases} \bK, & \text{if } |S|=0,\\
    0, & \text{if } |S|>0.
    \end{cases}
    \end{equation}
    \begin{equation}\label{eq: low deg calc2}
        H_{1}\left(Q^{\ohur}_{\mathbb{L}}(\pi_0\left(\ohur\right)\right)_S
        \cong H_{1}\left(Q^{\ohur}_{\mathbb{L}}\left(\ohur\right)\right)_S
        \cong 0 \text{ for all $S$.} 
    \end{equation}
    \begin{equation}\label{eq: low deg calc3}
H_{i}\left(Q^{\ohur}_{\mathbb{L}}\left(\pi_{0}\left(\ohur\right)\right)\right)_{\varnothing}
\cong H_{i}\left(Q^{\ohur}_{\mathbb{L}}(\ohur)\right)_{\varnothing} 
\cong \begin{cases} \bK, & \text{if } i=0,\\
    0, & \text{if } i>0.
    \end{cases}
    \end{equation}
This completes the proof. 
\end{proof}

\section{Uniform twisted homological stability for \texorpdfstring{$H_0\left(\ohur\right)$}{zeroth homology groups of ordered Hurwitz space}}

In this section, we prove a form of twisted homological stability for $H_0\left(\hur\right)$. Most of the section concerns basic results from the representation theory of the symmetric groups $\Sigma_n$ and the theory of $\FI$-modules. 

Throughout this section, we assume that $\bK$ is a field of characteristic zero. All group representations are implicitly assumed to take coefficients in $\bK$.

\begin{definition} Let $\mathcal{G}$ be a groupoid. 
    Let $\underline{\bK}_{\mathcal{G}}\in \Mod_{\bK}^{\mathcal{G}}$ denote the constant functor with value $\bK$ and trivial group actions. When $\mathcal{G}$ is monoidal, the isomorphism 
    $ \bK \otimes_{\bK} \bK \to  \bK$ endows $\underline{\bK}_{\mathcal{G}}$ with the structure of a commutative monoid object in the category of $\mathcal{G}$-modules with Day convolution.  
    When $\mathcal{G}=\fb$, we sometimes write $\underline{\bK}$ for $\underline{\bK}_{\mathcal{G}}$. 
\end{definition}

When $\mathcal{G}$ is the groupoid $\fb$,  a module over  $\underline{\bK}$ is equivalent to a module over the category $\FI$ of finite sets and injective maps; see Church--Ellenberg--Farb \cite{CEF}.  This is the \emph{twisted commutative algebra} perspective on $\FI$-modules (see Sam--Snowden \cite[Section 10.2]{SSIntroTCA}).  

%The object $\underline{\bK}_{\mathcal{G}}$ has the structure of a commutative algebra object in $\sMod_{\bK}^{\mathcal{G}}$. The data of an $\fin$-module $V$ gives $V$ the structure of a left $\underline{\bK}_{\mathcal{\fb}}$-module.

We now specialize to the study of $\Sigma_n$-representations and $\FI$-modules.

\begin{notation}
    Let $\lambda = (\lambda_1, \ldots, \lambda_{\ell})$ be a partition. When convenient, we will view $\lambda$ as a Young diagram. Let $|\lambda| $ denote the sum $\lambda_1 + \cdots + \lambda_{\ell}$. Let $\lambda\langle n\rangle$ denote the \emph{padded partition} $ (n-|\lambda|, \lambda_1, \ldots \lambda_{\ell})$ for $n \geq |\lambda| + \lambda_1$. The notation $\lambda\langle n\rangle$ is undefined for  $n < |\lambda| + \lambda_1$.
\end{notation}
\begin{notation} Over a field $\bK$ of characteristic zero, the irreducible representations of $\Sigma_n$ are in canonical bijection with partitions of $n$. See (for example) Fulton--Harris \cite[Chapter 4]{FultonHarris}.  
    Given a partition $\mu$, let  $V_{\mu}$ denote the corresponding irreducible $\Sigma_{|\mu|}$-representation. 
\end{notation}

\begin{definition} \label{DefVLn}
    Let $V(\lambda)$ denote the FI-module constructed in Church--Ellenberg--Farb \cite[Proposition 3.4.1]{CEF}  
    %as a suitable subquotient of the representable FI-module $\bK[\Hom_{\fin}([|\lambda|-1], -)]$ 
    whose underlying $\fb$-module satisfies
$$ V(\lambda)_{\bn} \cong  \begin{cases} 0, & n < |\lambda| + \lambda_1 \\  V_{\lambda\langle n\rangle}, & n \geq |\lambda| + \lambda_1 .
\end{cases}  $$ 
We let $\varphi_n \colon V(\lambda)_{\bn} \to V(\lambda)_{\bf n+1}$ denote the map induced by the inclusion $\bn \hookrightarrow {\bf n+1}$. Denote the composition $\varphi_{n, n'} :=\varphi_{n'-1} \circ \dots \circ \varphi_n: V(\lambda)_{\bn} \to V(\lambda)_{\bf n'}$. 
\end{definition}
\begin{definition}
    Given $T\in \fin$, let
    %\footnote{probably don't need to be this general} 
    $\text{sh}^{T}\colon \fin\to\fin$ be the functor sending $S$ to $S\sqcup T$ and a map $\varphi\colon S\hookrightarrow S'$ to $\varphi\sqcup \text{id}_{T}\colon S\sqcup T\hookrightarrow S'\sqcup T$. Given an $\fin$-module $V$, its \emph{shift} $\Sigma V$ is the $\fin$-module $V\circ \text{sh}^{{\bf 1}}$. We write $\Sigma^{r} V$ for the $r$-fold shift of $V$.
    %$\fin$-module 
    %which an object $S$ to $V(S\sqcup [0])$ and a map $\varphi\colon S\hookrightarrow T$ to  
    %with $(\Sigma V)(S)\colonequals V(S\sqcup [0])$ and given a map $\varphi\colon S\hookrightarrow T$ in $\fin$, the structure map $(\Sigma V)(\varphi)\colonequals$
\end{definition}
The shift $\Sigma V$ satisfies
    $ (\Sigma V)_{\bn} \cong \Res^{\Sigma_{n+1}}_{\Sigma_n} V_{\bf n+1}.$
\begin{proposition}\label{prop: coinv stabilize}  Let $\bK$ be a field of characteristic zero.  Fix a partition $\lambda=(\lambda_1, \ldots, \lambda_{\ell})$ and $r \geq 0$. 
\begin{itemize}
\item If $r <  |\lambda|$ then the $\Sigma_{n}$-coinvariants $(\Sigma^r V(\lambda)_{\bn})_{\Sigma_{n}}$ of the $r$-fold shift of $V(\lambda)$ vanish identically.  
\item If $r \geq |\lambda|$ then the induced map 
$$ (\Sigma^r \varphi_n)_* \colon (\Sigma^r V(\lambda)_{\bn})_{\Sigma_{n}} \to  (\Sigma^r V(\lambda)_{\bf n+1})_{\Sigma_{\bf n+1}}$$
is an isomorphism in the (sharp) stable range $n \geq \lambda_1$. 
\end{itemize} 
In particular, the $\FI$-module structure maps $\varphi_n$ induce isomorphisms on the coinvariants $(\Sigma^r V(\lambda)_{\bn})_{\Sigma_{\bn}}$ for all $n \geq r$, a stable range that does not depend on $\lambda$.
\end{proposition} 

To prove \cref{prop: coinv stabilize}, we first show that the dimensions of the coinvariants stabilize in the stated range, and then verify that the induced maps are injective. We will use the following standard notation.

\begin{notation}
    Given a finite group $H$, let $\langle -, - \rangle_{H}$ denote the standard associated inner product on $\bK\{\text{conjugacy-class-invariant functions $H \to \bK$}\}$; see for example Fulton--Harris \cite[Section 2.2 and Formula (2.11)]{FultonHarris}. 
    For $H$--representations $V, W$, let $\langle V, W\rangle_{H}$ denote the inner product  evaluated on the characters of $V$ and $W$. 
\end{notation}

%See examples on the last page. 

\begin{lemma}\label{ClaimFrobenius} The coinvariants $(\Sigma^r V(\lambda)_{\bn})_{\Sigma_n}$ vanish unless 
$n+r \geq |\lambda| + \lambda_1$. For $n+r \geq |\lambda| + \lambda_1$ the dimension of $(\Sigma^r V(\lambda)_{\bn})_{\Sigma_n}$ is 
$$ \sum_{ \substack{\text{partitions $\mu$ of $r$ such that}\\ \text{ $\lambda\langle n+r\rangle$ can be built from $\mu$} \\ \text{ by placing one box in each of $n$ distinct columns}}} \dim V_{\mu}$$ 
\end{lemma} 

\begin{proof}  Recall that $(\Sigma^r V(\lambda)_{\bn})_{\Sigma_n} = \Res^{\Sigma_{n+r}}_{\Sigma_n} V(\lambda)_{\bf n+r}$, and by construction $V(\lambda)_{\bf n+r}$ vanishes unless $n+r \geq |\lambda| + \lambda_1$. 

Assume  $n+r \geq |\lambda| + \lambda_1$. Recall that $\underline{\bK}_{\bn}$ is the 1-dimensional trivial $\Sigma_n$-representation $V_{(n)}$. 
\begin{align*}
&\dim (\Sigma^r V(\lambda)_{\bn})_{\Sigma_n} \\
&= \langle  \underline{\bK}_{\bn}, \; (\Sigma^r V(\lambda))_{\bn} \rangle_{\Sigma_n} &\\ 
& = \langle  \underline{\bK}_{\bn}, \; \Res^{\Sigma_{n+r}}_{\Sigma_n} V(\lambda)_{\bf n+r} \rangle_{\Sigma_n} \\ 
& = \langle  \underline{\bK}_{\bn}, \; \Res^{\Sigma_n \times \Sigma_r}_{\Sigma_n} \Res^{\Sigma_{n+r}}_{\Sigma_n \times \Sigma_r} V(\lambda)_{\bf n+r} \rangle_{\Sigma_n} \\ 
& = \langle \Ind^{\Sigma_n \times \Sigma_r}_{\Sigma_n}   \underline{\bK}_{\bn},  \; \Res^{\Sigma_{n+r}}_{\Sigma_n \times \Sigma_r} V(\lambda)_{\bf n+r} \rangle_{\Sigma_n \times \Sigma_r} \\ & \qquad \qquad  \text{(by Frobenius reciprocity, e.g. Fulton--Harris \cite[Corollary 3.20]{FultonHarris})}  \\ 
& = \langle   \underline{\bK}_{\bn} \boxtimes \bK[\Sigma_r] ,  \; \Res^{\Sigma_{n+r}}_{\Sigma_n \times \Sigma_r} V(\lambda)_{\bf n+r} \rangle_{\Sigma_n \times \Sigma_r}  \\ 
& = \sum_{\text{$\mu$ a partition of $r$}} (\dim V_{\mu}) \; \langle  \underline{\bK}_{\bn} \boxtimes V_{\mu} , \;\Res^{\Sigma_{n+r}}_{\Sigma_n \times \Sigma_r} V(\lambda)_{\bf n+r} \rangle_{\Sigma_n \times \Sigma_r} 
 \\ & \qquad \qquad  \text{(e.g. Dummit--Foote \cite[Section 18.2 Example (3)]{DummitFoote})}  \\ 
& = \sum_{\text{$\mu$ a partition of $r$}} (\dim V_{\mu}) \; \langle   \underline{\bK}_{\bn} \boxtimes V_{\mu} , \;\Res^{\Sigma_{n+r}}_{\Sigma_n \times \Sigma_k} V_{\lambda\langle n+r\rangle} \rangle_{\Sigma_n \times \Sigma_r}  \\ 
& = \sum_{\text{$\mu$ a partition of $r$}} (\dim V_{\mu}) \; \langle  \Ind^{\Sigma_{n+r}}_{\Sigma_n \times \Sigma_r} \underline{\bK}_{\bn} \boxtimes V_{\mu} , \; V_{\lambda\langle n+r\rangle} \rangle_{\Sigma_{n+r}} \qquad 
\\
&\text{(again by Frobenius reciprocity).} 
\end{align*}
The quantity $\langle  \Ind^{\Sigma_{n+r}}_{\Sigma_n \times \Sigma_r} \underline{\bK}_{\bn} \boxtimes V_{\mu} , \; V_{\lambda\langle n+r\rangle} \rangle_{\Sigma_{n+r}} $ is described by Pieri's rule; see for example Stembridge \cite[Section 1.C]{StembridgeNotes}. 
It is equal to $1$ if the Young diagram $\lambda\langle n+r\rangle$ can be built from the Young diagram $\mu$ by adding one box in each of $n$ distinct columns (i.e. adding a \emph{horizontal $n$-strip}), and it is $0$ otherwise. The lemma follows. 
\end{proof} 

\begin{lemma} \label{Inj} Let $\bK$ be a field of characteristic zero.  Fix a partition $\lambda=(\lambda_1, \ldots, \lambda_{\ell})$ and $r \geq 0$. For all $r \geq 0$ and all $n\geq 0$, the induced maps on the coinvariants of the shifted modules
    $$ (\Sigma^r \varphi_n)_* \colon  (\Sigma^r V(\lambda)_{\bn})_{\Sigma_{n}} \to  (\Sigma^r V(\lambda)_{\bf n+1})_{\Sigma_{\bf n+1}}$$
    are injective. 
\end{lemma}

\begin{proof} By construction (Church--Ellenberg--Farb \cite[Proposition 3.4.1]{CEF}), the $\FI$-module $V(\lambda)$ is a sub-$\FI$-module of a certain $\FI$-module they denote $M(\lambda)$. 
Church--Ellenberg--Farb \cite[Proposition 3.1.7]{CEF} prove that the $\FI$-modules $M(\lambda)$ have \emph{injectivity degree} zero in the sense of  \cite[Definition 3.1.5]{CEF}. This means that the stabilization maps $(\Sigma^r M(\lambda)_\bN)_{\Sigma_n} \to (\Sigma^r M(\lambda)_{\bf n+1})_{\Sigma_{n+1}}$  are injective for all $n$ and $r$. Consider the following diagram.
% https://q.uiver.app/#q=WzAsNCxbMCwwLCIoXFxTaWdtYV5yIFYoXFxsYW1iZGEpX3tcXGJmIG59KV97XFxTaWdtYV97bn19Il0sWzEsMCwiKFxcU2lnbWFeciBWKFxcbGFtYmRhKV97XFxiZiBuKzF9KV97XFxTaWdtYV97bisxfX0iXSxbMSwxLCIoXFxTaWdtYV5yIE0oXFxsYW1iZGEpX3tcXGJmIG4rMX0pX3tcXFNpZ21hX3tuKzF9fSJdLFswLDEsIihcXFNpZ21hXnIgTShcXGxhbWJkYSlfe1xcYmYgbn0pX3tcXFNpZ21hX3tufX0iXSxbMCwzXSxbMSwyXSxbMywyLCIoKikiXSxbMCwxLCIoXFxTaWdtYV5yIFxcdmFycGhpX24pXyoiXV0=
\[\begin{tikzcd} 
	{(\Sigma^r V(\lambda)_\bN)_{\Sigma_{n}}} & {(\Sigma^r V(\lambda)_{\bf n+1})_{\Sigma_{n+1}}} \\
	{(\Sigma^r M(\lambda)_\bN)_{\Sigma_{n}}} & {(\Sigma^r M(\lambda)_{\bf n+1})_{\Sigma_{n+1}}}
	\arrow["{(\Sigma^r \varphi_n)_*}", from=1-1, to=1-2]
	\arrow[from=1-1, to=2-1]
	\arrow[from=1-2, to=2-2]
	\arrow["{ }", from=2-1, to=2-2]
\end{tikzcd}\]
Since taking shifts is exact and taking coinvariants is exact for finite groups over characteristic zero, the vertical maps in the diagram are injective. Thus the maps $ (\Sigma^r \varphi_n)_*$ are injective as claimed. 
\end{proof}

\begin{proof}[Proof of \cref{prop: coinv stabilize}]
By \cref{Inj}, the stabilization maps are injective so it suffices to show the dimensions of the coinvariants agree.

First suppose $\lambda = \varnothing$. Then $V(\lambda)$ is a copy of the 1-dimensional trivial representation in every degree, and $V(\lambda) = \Sigma^r V(\lambda)$ for all $r$. The coinvariants of $\Sigma^r V(\lambda)$ stabilize for all $n\geq 0$ for all $r$.  The claim holds for $\lambda = \varnothing$.

Now suppose $\lambda \neq \varnothing$. Suppose $r < |\lambda|$. The largest part of $\lambda\langle n+r\rangle $ is $n+r-|\lambda| < n$, and the Young diagram $\lambda\langle n+r\rangle$ does not contain the diagram $(n)$.  Thus by the branching rules---see for example Fulton--Haris \cite[Exercise 4.44]{FultonHarris} for the relevant special case---the restriction of $V(\lambda)_{\bf n+r} \cong V_{\lambda\langle n+r\rangle}$ from $\Sigma_{n+r}$ to $\Sigma_n$ does not contain any copies of the trivial $\Sigma_n$-representation. We conclude that $\Sigma_n$-coinvariants of $\Sigma^r V(\lambda)$ vanish for all $n$. 
%(Alternatively, consider the formula in Claim \ref{ClaimFrobenius}. Notice that the index set for the sum is always empty -- when $r < |\lambda|$ the Young diagram $\lambda\langle n+r\rangle$ can never be constructed from a Young diagram of size $r$ by placing boxes in distinct columns.)   
In particular, if $r < |\lambda|$, the coinvariants $(\Sigma^r V(\lambda)_{\bn})_{\Sigma_n}$ are stable (and zero) for all $n\geq 0$.   

Now fix  $\lambda \neq \varnothing$ and  fix $r \geq |\lambda|$.  Consider the formula in  \cref{ClaimFrobenius}. We will compute the stable range for the coinvariants $(\Sigma^r V(\lambda)_{\bn})_{\Sigma_n}$.  

For $\mu$ a partition of $r$, the value $\dim V_{\mu}$ does not depend on $n$. We must determine the degree $n$ for which the sum's index set stabilizes.  

Suppose $\mu$ is a partition of $r$ that appears in the index set for some level $n$, i.e., $\lambda\langle n+r\rangle$ is defined and we can build $\lambda\langle n+r\rangle$ from $\mu$ through the addition of $n$ boxes in distinct columns. In this case, we can also build $\lambda\langle n+r+1\rangle$ from $\mu$ by adding an additional box to the top row. Hence $\mu$ also appears in the index set at level $(n+1)$.

Conversely, if $n > \lambda_1$, then to build  $\lambda\langle n+r\rangle$ from a partition $\mu$ by $n$ placing boxes in distinct columns, we must add at least one box to the top row. Deleting that box gives a valid construction of $\lambda\langle n+r-1\rangle$ from $\mu$ through the addition of boxes in distinct columns. This implies that the index set stabilizes once $n \geq \lambda_1$. 

It remains to show that this bound is sharp. The largest part of $\lambda\langle n+r\rangle$ is, by definition, $n+r-|\lambda|$.   Let $n=\lambda_1$. The largest part of $\lambda\langle n+r\rangle$ in this case is $\lambda_1+r-|\lambda|$. Hence this is the smallest $n$ such that 
 that $\mu = (\lambda_1 + r-|\lambda|, \lambda_2, \ldots, \lambda_{\ell})$ appears as a valid index.  Thus the stable range of $n \geq \lambda_1$ is sharp.
 \end{proof}
 \begin{definition}

   Let $\mathcal{G}$ be $\N$ or $\fb$. Let $\mathcal{C}$ be one of the categories $\Mod_{\bK}, \sMod_{\bK}$ or $\ch_{\bK}$. For objects $X, Y \in \mathcal{C}^{\mathcal{G}}$,  
    the \emph{pointwise tensor product} $X\oast  Y\in \mathcal{C}^{\mathcal{G}}$ is defined for each finite set $T$ by the formula
    $$(X\oast Y)(T)\colonequals X(T)\otimes_{\bK} Y(T)$$ and the action of  $\mathrm{End}_{\mathcal{G}}(T)$ on $(X\oast Y)(T)$ is the diagonal action.
    %induced by its action on $X(T)$ and $Y(T)$. 
\end{definition}
The pointwise tensor product provides $\mathcal{C}^{\mathcal{G}}$ with an additional monoidal structure, with unit given by $\underline{\bK}_{\mathcal{G}}$.
%\footnote{Define this. It isn't the unit under $\otimes$; rather, it is the constant functor. Also, explain how a coefficient system is a module over $k_{\mathcal{G}}$.} 
Since $\bK$ is a field of characteristic zero, $-\oast -$ preserves weak homotopy equivalences in each variable. The pointwise tensor product and Day convolution have a natural distributivity law
\begin{equation}\label{eq: distr law}
    (W \oast X)\otimes_{\mathcal{G}} (Y\oast Z)\to (W\otimes_{\mathcal{G}} Y) \oast (X\otimes_{\mathcal{G}} Z).
\end{equation}
This makes $(\mathcal{C}^{\mathcal{G}}, \oast,\otimes_{\mathcal{G}})$ into a \emph{duoidal category} in the sense of Aguiar--Mahajan \cite[Section 6.1]{AA-Duoidal}; they use the term ``2-monoidal category'' to refer to a duoidal category. 

Let 
%$X\in \sMod^{\fb}_{\bK}$ be a monoid object, 
$\bM\in \sMod^{\fb}_{\bK}$ be a left $(\ohur)_{\bK}$-module, and $V$ an $\fin$-module. We view $V$ as a $\sMod^{\fb}_{\bK}$-module supported in homological degree $0$.    The object $\bM \oast V\in \sMod^{\fb}_{\bK}$ has the structure of a left $(\ohur)_{\bK}$-module by using the identification $(\ohur)_{\bK}\cong (\ohur)_{\bK}\oast \underline{\bK}_{\mathcal{\fb}}$ and applying the distributivity law of Formula (\ref{eq: distr law})
$$  ((\ohur)_{\bK}\oast \underline{\bK})\otimes_{\mathcal{\fb}} (\bM\oast V)\to ((\ohur)_{\bK}\otimes_{\mathcal{\fb}} \bM )\oast (\underline{\bK}_{\mathcal{\fb}}\otimes_{\mathcal{\fb}} V)$$
followed by taking the pointwise tensor product of the maps $(\ohur)_{\bK}\otimes_{\fb} \bM\to \bM$ and  $\underline{\bK}_{\mathcal{\fb}}\otimes_{\fb} V\to V$ that arise from the structures of $\bM$ and $V$ as modules over $(\ohur)_{\bK}$ and $\underline{\bK}_{\mathcal{\fb}}$, respectively.

Let $\widetilde{U} \in H_0\left(\ohur[,\bN]\right)$ be as in Notation \ref{NotationUTwiddle}. By abuse of notation, we also let $\widetilde{U}$ denote a map $$S^{\bN, 0}_{\bK} \to \left(\ohur\right)_{\bK} \oast \underline{\bK}_{\fb}$$ representing the element $$\widetilde{U} \in H_0(\ohur[,\bN]) \cong \pi_0\left( \left(\ohur\right)_{\bK} \oast  \underline{\bK}_{\fb}\right).$$ The map $\widetilde U$ and the product on $(\ohur)_{\bK}$  define a map
\begin{equation}\label{eq: twisted stab map}
      \widetilde{U}_* \colon S^{\bN, 0}_{\bK} \otimes_{\fb} (\bM\oast V ) \longrightarrow \bM\oast V.
    \end{equation}
\begin{notation}
    Given a left $(\ohur)_{\bK}$-module $\bM\in \sMod^{\fb}_{\bK}$ and $V$ an $\fin$-module, let $(\bM\oast V)/ \widetilde{U}_*$ denote the mapping cone of the map
    \eqref{eq: twisted stab map} 
    in the model category structure on $\sMod^{\fb}_{\bK}$ described in Section \ref{SubsectionCategorySetup}. 
\end{notation}

\begin{notation}
    Given a group $K$ and a simplicial $\bK$-module $Y$ with simplicial $K$-action, we use the shorthand $H_i(K; Y)$ to denote the hyperhomology of $K$ with coefficients in the chain complex obtained by applying the Dold--Kan construction to $Y$.
\end{notation}

\begin{notation}
Let $$(-)_{\fb}\colon \sMod^{\fb}_{\bK}\to \sMod^{\N}_{\bK}$$ denote the functor given by the map sending $X\in \sMod^{\fb}_{\bK}$ to the object in $\sMod^{\N}_{\bK}$ with $(X)_{\fb}(n)\colonequals (X(\bn))_{\Sigma_{\bn}}=X(\bn)/\Sigma_{n}$.
%i.e. in degree $n$, $(X)_{\Sigma}$ is the chain complex given by taking orbits of $X([n-1]$ under the action by $\Sigma_{[n-1]})$.
\end{notation}
%Let \footnote{move this.} is a functor $$(-)_{\Sigma}\colon \ch^{\fb}_{\bK}\to \ch^{\N}_{\bK}$$ given by the map sending $X\in \ch^{\fb}_{\bK}$ to the object in $\ch^{\N}_{\bK}$ with $(X)_{\Sigma}(n)\colonequals (X([n-1]))_{\Sigma_{[N-1]}}=X([n-1])/\Sigma_ {[N-1]}$, i.e. in grading $n$, $(X)_{\Sigma}$ is the chain complex given by taking orbits of $X([n-1]$ under the action by $\Sigma_{[n-1]} $. 

%Equivalently, the functor $(X)_{\Sigma}$ is the left Kan extension of $X$ along the functor $\fb\to \N$.
The functor $(-)_{\fb}$ is strong monoidal with respect to the Day convolution structures on $\sMod^{\fb}_{\bK}$ and $\sMod^{\N}_{\bK}$. It admits a left derived functor
%\footnote{Need to check that this is actually true} 
$(-)_{\hfb}$ given by taking homotopy orbits, i.e. $(X)_{\hfb}(n)$ is obtained by taking homotopy orbits $ (X(\bn))_{\text{h}\Sigma_{n}}$ of $X(\bn)$ under $\Sigma_{n}$.   
The derived functor
%\footnote{have comment about left Kan extensions}
$(X)_{\hfb}$ satisfies
%\footnote{maybe need to probably explain the middle term involving homotopy groups}
$$H_{i}((X)_{\hfb})_n = H_{i}(\Sigma_{n}; X(\bn)).$$

%There is a homotopy orbit spectral sequence
%\begin{equation}\label{eq: hom orb specseq}
%    E^{2}_{n, s, t}= H_{s}(\Sigma_{n}; H_{t}(X)_{\bn})\Longrightarrow H_{ s+t}((X)_{h\Sigma})_n.
%x\end{equation}
%\begin{remark}
%    Equivalently, t the functor $(-)_{\Sigma}$ is that is the left Kan extension along the functor $\fb\to \N$.
%\end{remark}
%By composing the inclusion $M(T)\otimes V(T)\hookrightarrow \Ind^{\Sigma_{[N-1]\sqcup T}}_{\Sigma_{T}}\bM(T)\otimes V(T)$ with  the map $\widetilde{U}\cdot -$, we obtain an $\bK[\Sigma_{T}]$-equivariant map
%$$(\widetilde{U}\cdot -)(T)\colon \bM(T)\otimes V(T)\to \bM(T\sqcup [N-1])\otimes V(T\sqcup [N-1]).$$
By applying $(-)_{\hfb}$ to the map  $\widetilde{U}_*$ of Formula \eqref{eq: twisted stab map}, we get a map
%\footnote{The object $(\bM\boxtimes V)_{h\Sigma}$ is in $\sMod_{\bK}^{\N}$, not in $\sMod_{\bK}^{\fb}$, so the $T$ in $H_{T, d}((\bM\boxtimes V)_{h\Sigma})$ needs to be replaced with $n$} 
%\begin{equation}\label{eq: twisted hoorb map}
%    (\widetilde{U}-\cdot)_{h\Sigma} \colon S^{N, 0}_{\bK}\otimes (\bM\boxtimes V)_{h\Sigma}\to (\bM\boxtimes V)_{h\Sigma},
%\end{equation}
$$\widetilde{U}_{\hfb} \colon S^{N, 0}_{\bK}\otimes_{\N} (\bM\oast V)_{\hfb}\to (\bM\oast V)_{\hfb}.$$ 
and we have that $((\bM\oast V)/ \widetilde{U})_{\hfb}$ is the mapping cone of $\widetilde{U}_{\hfb}$.
%Let $T$ denote $[n-1]$.
%By passing to  $H_{n, *}(-)$, 
The induced map on homology has the form
$$\widetilde{U}_{\hfb, *}(n)\colon H_{*}(\Sigma_{n}; \bM_{n}\otimes_{\bK} V_{n})\to H_{*}(\Sigma_{n+N}; \bM_{n+N}\otimes_{\bK} V_{n+N}).$$

\begin{lemma}\label{lem: vanishing for base case}
Suppose that $(G, c)$ has the non-splitting property and $\bK$ is a field of characteristic zero.  Let $N_c$ be as in \cref{MarkTheorem}. Let $N$ and $N_0$ be as in \cref{lem: stability for pi0 Hur}. Fix a partition $\lambda=(\lambda_1, \ldots, \lambda_{\ell})$ and $r \geq 0$. 
    %Let $V$ be an irreducible FI-module. 
    Then
    %\footnote{fix notation}
    $$H_{i}\bigg(\Big(\left(H_0(\ohur) \oast \Sigma^{r}V(\lambda)\right)/\widetilde{U}\Big)_{\hfb}\bigg)_n=0$$
    %$H_{i}( \Sigma_n ; ( H_0(\ohur))\oast \Sigma^{\bK}V(\lambda) /\widetilde{U}))_n=0$ 
    %is independent of $n$ in the range 
    whenever $n\geq \max(r+N, N_0+N,N_c+N)$ or $i>0$.  
    %$n\geq d+\max(B_{2}, N_{1})+1$.
\end{lemma}
\begin{proof}
Let $V=\Sigma^{r}V(\lambda)$.
Let us unpack the statement. The groups $$H_{i}(((H_0(\ohur) \oast V)/\widetilde{U})_{\fb})_n $$ are simply the relative homology groups $$H_i\Bigl(\big(\Sigma_n,\Sigma_{n-N}\big);\big(V_n \otimes_{\bK} H_0(\ohur[,n]), V_{n-N} \otimes_{\bK} H_0(\ohur[,n-N])\big)\Bigr)$$ associated to the map given by multiplication by $\widetilde{U}$. Since symmetric groups are finite and we are working in characteristic zero, these relative homology groups vanish for $i>0$. 

Let $n \geq \max(r+N, N_0+N,N_c+N)$. By \cref{CorUIsoOnOHur}, multiplication by $\widetilde{U}$ induces an isomorphism $$H_0(\ohur[,n-N]) \to H_0(\ohur[,n]).$$ Moreover, by \cref{MarksConstantNc}, the symmetric group actions on $\ohur$ are trivial in a stable range. Thus, we can rewrite the relative homology groups as follows: 
\begin{align*}
H_0\Big( (\Sigma_n,\Sigma_{n-N});\big(&V_n \otimes_{\bK} H_0(\ohur[,n]), V_{n-N} \otimes_{\bK} H_0(\ohur[,n-N])\big) \Big) \cong \\
  H_0\Big(\big(\Sigma_n,\Sigma_{n-N}\big);\big(&V_n , V_{n-N} \big )  \Big) \otimes_{\bK} H_0(\ohur[,n]).
\end{align*}
%$$H_0{\huge \big (}\big(\Sigma_n,\Sigma_{n-N}\big);\big(V_n \otimes H_0(\ohur[,n]), V_{n-N} \otimes H_0(\ohur[,n-N])\big){\huge \big )} \cong$$ $$ H_0{\huge \big (}\big(\Sigma_n,\Sigma_{n-N}\big);\big(V_n , V_{n-N} \big )  \huge \big )} \otimes H_0(\ohur[,n]). $$ 
The claim now follows from \cref{prop: coinv stabilize} which implies $$H_0(\Sigma_n,\Sigma_{n-N};V_n,V_{n-N}) \cong 0$$ for $n \geq r+N$.
\end{proof}

\section{Uniform multiplicity stability for \texorpdfstring{$H_i(\ohur$}{the higher homology groups of ordered Hurwitz space}}

 \cref{lem: vanishing for base case} implies $\pi_0(\ohur)$ exhibits a strong form of uniform multiplicity stability. \cref{OhurVSpi0} implies $\pi_0(\ohur)$ is constructed using $\ohur$-module cells of high slope.  The philosophy of uniform stability lets us import uniform multiplicity stability from $\pi_0(\ohur)$ to $\ohur$. We will implement this strategy in this section.

    \begin{proposition}
\label{lem: uniform rep stab for OHur}
    Suppose that $(G, c)$ has the non-splitting property and $\bK$ is a field of characteristic zero. There are constants $A$ and $B$ depending only on $(G,c)$,   
    such that, for all partitions $\lambda$ and $r \geq 0$, $$H_{i} {\bigg (}\Big({\Big (}(\ohur)_{\bK}\oast \Sigma^{r}V(\lambda){\Big )}/\widetilde{U}{\Big)}_{\hfb} {\bigg )}_n \cong 0$$ for $n \geq A i + r +B$.

   % for $d < \frac{1}{\max(B_{2}, N_{1})+2}(n-(\max(r, N_{1})-1))$.
    %$n\geq d + \max(r, N_{1})$.
    %$d< n- \max(r, N_{1})+1$.
    %\footnote{need to fix the range. It might be  $d< T - \max(r, N_{1})$.}
    %$d< \frac{1}{2+\max(B_{2}, N_{1})}n - \max(r, N_{1})+1$.
\end{proposition}
\begin{proof}
    The proof  is the same as that of Miller--Patzt--Petersen--Randal-Williams \cite[Theorem 2.2]{miller2024uniformtwistedhomologicalstability}. %\footnote{I need to make the various inequalities consistent}
    %\footnote{I need to figure out the ranges once and for all. Also, I need to make the various inequalities consistent} %\footnote{cite theorem 2.2 of mpprw}
    %.\footnote{I need to figure out the ranges once and for all. Also, I need to make the various inequalities consistent}
    %Let $V$ denote $V(\lambda)$. 
    %For simplicity, we assume that $T=[n-1]$ for some $n\in \N$. 
     Let $V$ denote $V(\lambda)$. Let $A=4 \aleph + 2\beth$ with $\aleph$ and $\beth$ as in \cref{OhurVSpi0}. Let $B=\max(N_0+N, N_c+N)$ with $N,N_0$ as in \cref{lem: stability for pi0 Hur} and $N_c$ as in \cref{MarksConstantNc}.
     
     We will prove by induction on $d$ that for all $n \geq A d + r +B$,
$$H_{d}\bigg(\Big(\Big ((\ohur)_{\bK}\oast \Sigma^{r}V\Big)/\widetilde{U} \Big)_{\hfb } \bigg)_n \cong 0.$$
     The case $d<0$ is vacuous.  Fix $d$ and  suppose by induction that the conclusion of the proposition holds for degrees strictly smaller than $d$.
       
    Let $\bR\in \sMod_{\bK}^{\fb}$ denote $(\ohur)_{\bK}$ and $\bQ \in \sMod_{\bK}^{\fb}$ denote $H_0(\ohur)$.
    %\footnote{Need to explain why this is an $\bR$-module at some point. Also, $(\pi_{0}(\ohur))_{\bK}$ can be replaced with $H_{0}(\ohur)$}
    %, let $\bM\in\bR$-mod denote $(\pi_{0}(\ohur))_{\bK}$
    %\footnote{change $\bM$ to $Q$}
    %\footnote{fix mod. Also might want to introduce this notation earlier in the paper}  
    We will work in the setting of Galatius--Kupers--Randal-Williams \cite{e2cellsI}.
    %\footnote{cite cellular}. 
 Since we have a map $\bR\to \bQ$, we may attach an $\bR$-cell to $\bQ$ at $S=\varnothing$ and $i=0$.
    %\footnote{degree might not be the right word here} 
    Denote this $(\varnothing,0)$-cell by $y\bR$. We can form the quotient $\bR$-module $\bQ/y\bR$. 

Since $\bR$ is a free $\bR$-module on one generator, $$H_{i}(Q_{\mathbb{L}}^{\bR}(\bR))_S \cong 0 $$ unless $|S|=0$ and $i=0$, in which case it is $\bK$. The map $$H_{i}(Q_{\mathbb{L}}^{\bR}(\bR))_S  \to H_{i}(Q_{\mathbb{L}}^{\bR}(\bQ))_S $$ is an isomorphism for $i=0$ and $|S|=0$. Using these observations, the long exact sequence in homology associated to the cofiber sequence $$ Q_{\mathbb{L}}^{\bR}(\bR) \to Q_{\mathbb{L}}^{\bR}(\bQ) \to Q_{\mathbb{L}}^{\bR}(\bQ /y\bR),$$ as well as  \cref{OhurVSpi0} and \cref{lem:deg01indec}, we conclude that, $$H_{i}(Q_{\mathbb{L}}^{\bR}(\bQ/y\bR))_S \cong 0$$ if $i=0$, if $i=1$, if $|S|=0$,  or if $ |S| \geq \aleph i +\beth $.

  Therefore, we may find a relative $\bR$-module cellular approximation $\mathbf{D}\xrightarrow{\sim} \bQ/y\bR$ such that $\mathbf{D}$ only has relative $({\bf n'} , d')$-cells with $d'\geq 2$,  and $1 \leq n' < \aleph d' +\beth$ by \cite[Theorem 11.21]{e2cellsI}.
    %\footnote{cite cellular, Theorem 11.21}. 
    The skeletal filtration of $\bD$
    has associated graded
    $$\text{gr}(\bD)\simeq \bigoplus_{\text{cells }\alpha}S^{{\bf n_{\alpha}}, d_{\alpha}}_{\bK}\otimes_{\fb} \bR$$ with $d_{\alpha} \geq 2$,  and $1 \leq n_{\alpha} < \aleph d_{\alpha} +\beth$.
    This filtration induces a filtration of the object $((\bD \oast \Sigma^{r}V)/\widetilde{U})_{\hfb}$.
    %\footnote{need to explain this sentence prior to the proof} 
    This filtration yields a spectral sequence for each $n$ and $r$:    
    \begin{align*}
        E^{1}_{p, q}(n) &= \bigoplus_{\alpha \text{ with } d_{\alpha}=p} H_{p+q}((((S^{{\bf n_{\alpha}}, d_{\alpha}}_{\bK}\otimes_{\fb} \bR)\oast \Sigma^{r}V)/\widetilde{U})_{\hfb})_{n}\\
        &\Longrightarrow H_{p+q}(((\bD\oast \Sigma^{r}V)/\widetilde{U})_{\hfb})_{n}.
    \end{align*} 
There are isomorphisms
\begin{align*}
& H_{i}((((S^{ {\bf n_{\alpha}}, d_{\alpha}}_{\bK}\otimes_{\fb} \bR)\oast \Sigma^{r}V)/\widetilde{U})_{\hfb})_{n} \\ 
    & \cong H_{i-d_{\alpha}}(\Sigma_{ n-n_\alpha}, \Sigma_{n-n_\alpha -N}; \bR_{\bf n-n_\alpha}\otimes_{\bK}\Sigma^{r}V_{\bn}, \bR_{\bf n-n_\alpha -N}\otimes_{\bK}\Sigma^{r}V_{\bf n-N})\\
 & \cong H_{i-d_{\alpha}}(\Sigma_{n-n_\alpha}, \Sigma_{n-n_\alpha -N}; \bR_{\bf n-n_\alpha}\otimes_{\bK}\Sigma^{r+n_\alpha}V_{\bf n-n_\alpha}, \bR_{\bf n-n_\alpha -N}\otimes_{\bK}\Sigma^{r+n_\alpha}V_{\bf n-n_\alpha -N})\\
    & \cong H_{ i-d_{\alpha}}(((\bR\oast \Sigma^{r+{n_{\alpha}}}V)/\widetilde{U})_{\hfb})_{n-n_{\alpha}}.
\end{align*}
The inductive hypothesis implies
$$H_{i}((((S^{ {\bf n_{\alpha}}, d_{\alpha}}_{\bK}\otimes_{\fb} \bR)\oast \Sigma^{r}V)/\widetilde{U})_{\hfb})_{n} \cong 0$$
whenever $$ n - n_\alpha \geq A (i-d_\alpha) + (r + n_\alpha) +B$$
and $i -d_{\alpha} \leq d-1$. Since $d_{\alpha}\geq 2$, it suffices to assume $i \leq d+1$. 
Again, since $d_{\alpha}\geq 2$, $n_\alpha < \aleph d_\alpha +\beth$, and $A=4\aleph + 2\beth$, we conclude that 
$$H_{i}((((S^{ {\bf n_{\alpha}}, d_{\alpha}}_{\bK}\otimes_{\fb} \bR)\oast \Sigma^{r}V)/\widetilde{U})_{\hfb})_{n} \cong 0$$ whenever 
$$n \geq A(i-1)+r+B$$ and $i \leq d+1$. By considering the spectral sequence, it follows that 
$$H_{i}(((\bD\oast \Sigma^{r}V)/\widetilde{U})_{\hfb})_n\cong 0$$ whenever $n \geq A(i-1)+r+B$ and $i\leq d+1$.

We have a homotopy cofiber exact sequence of $\bR$-modules: $$\bR \to \bQ  \to \bD.$$ Consider the following portion of the associated long exact sequence:
$$H_{d+1}(((\bD\oast \Sigma^{r}V)/\widetilde{U})_{\hfb})_{n} \to H_{d}(((\bR\oast \Sigma^{r}V)/\widetilde{U})_{\hfb})_{n} \to H_{d}(((\bQ\oast \Sigma^{r}V)/\widetilde{U})_{\hfb})_{n}. $$ For $n \geq Ad+r+B$, $$H_{d+1}(((\bD\oast \Sigma^{r}V)/\widetilde{U})_{\hfb})_{n} \cong 0 $$ by the above paragraph. The term $$H_{d}(((\bQ\oast \Sigma^{r}V)/\widetilde{U})_{\hfb})_{n} \cong 0$$ whenever $d>0$ or $n \geq r+B$ by \cref{lem: vanishing for base case}. Hence $$H_{d}(((\bR\oast \Sigma^{r}V)/\widetilde{U})_{\hfb})_{n} \cong 0$$ whenever $n \geq A d+r +B$. The claim follows by induction. 
\end{proof}

We will use \cref{lem: uniform rep stab for OHur} to deduce
\cref{thm: mult stability for ohur}, that $H_i(\ohur ;\bK )$ has uniform multiplicity stability. Before we can do this, we will need to recall work of Davis--Schlank \cite{davis2023hilbertpolynomialquandlescolorings}.

%Throughout this section, we use the same assumptions stated in the beginning of \cref{subsec: FI-mods}.

\begin{comment} 
\begin{definition} Let $R$ be a graded $\bK$-algebra. 
    A left $R$-module $\bM\in\Mod^{\N}_{\bK}$ is \emph{finitely generated} if
    %if $\bM$, viewed as a graded module, is finitely generated over $R$, viewed as a graded ring.
    the graded module $\oplus_{n\in \N}\bM_{n}$ is a finitely generated module of the graded ring $\oplus_{n\in \N}R_{n}$.
    %if there is some $K\in \N$ such that, for each $i\leq K$, there is a finitely dimensional subspace $\bN_{i}\subseteq \bM_{i}$ and a surjection
    %$$R\otimes_{\bK}(\bigoplus_{i\leq K} \bN_{i}\to \bigoplus_{n\in \N} \bM_{n}.$$  
\end{definition}
\end{comment} 

%We need the following result of Davis--Schlank in the proof  \cref{thm: uniform mult stab}.
%\footnote{cite}. 
%It is a special case of \cite[Proposition 3.36]{davis2023hilbertpolynomialquandlescolorings} (see also \cite[Remark 3.38]{davis2023hilbertpolynomialquandlescolorings}).%\footnote{cite Proposition 3.36 of Davis--Schlank and Remark 3.38}.
%\cite{davis2023hilbertpolynomialquandlescolorings}
\begin{theorem}[{Davis--Schlank \cite[Proposition 3.36]{davis2023hilbertpolynomialquandlescolorings}}]\label{prop: Hilbert for R-mods}
Suppose that $(G, c)$ has the non-splitting property and $\bK$ is a field of characteristic zero. 
    Let $\bM\in\Mod^{\N}_{\bK}$ be a left $H_{0}(\hur;\bK)$-module. Assume that $\bM$ is finitely generated in the sense that $\bigoplus_n \bM_n$ is finitely generated over the ring $\bigoplus_n H_{0}(\hur[,n];\bK)$. 
    Then $\dim_{\bK}\bM_{n}$ is constant for $n\gg 0$.
\end{theorem}

\begin{definition} \label{DefnUniformMultStability}
Let $X$ be an $\fb$-module over a field of characteristic zero such that $X_n$ is finite dimensional for all $n$. Let
$$X_n \cong \bigoplus_{\lambda} V_{\lambda \langle n\rangle}^{\oplus c_{\lambda, n}}$$ 
be the decomposition of $X_n$ into irreducible $\Sigma_n$-representations, where the multiplicity $c_{\lambda, n}$ is defined to be zero whenever $n$ is too small for $V_{\lambda \langle n\rangle}$ to be well-defined, that is, $c_{\lambda, n}=0$ for $n<|\lambda|+\lambda_1$. Then we say that $X$ has \emph{uniform multiplicity stability} if there exists some $k$ such that, for all partitions $\lambda$, the multiplicity $c_{\lambda, n}$ is independent of $n$ for all $n \geq k$. 
\end{definition}

Note in particular that, if $X$ has multiplicity stability with stable range $k$, then every irreducible representation $V_{\lambda\langle n \rangle}$ occuring with nonzero multiplicity in $X_n$ for any $n$ must satisfy $|\lambda| + \lambda_1 \leq k$. 

The following is a mild generalization of \cref{thm: mult stability for ohur}, uniform multiplicity stability for $H_i(\ohur)$.

%\cref{thm: mult stability for ohur}.
\begin{theorem}\label{thm: uniform mult stab}
Suppose that $(G, c)$ has the non-splitting property and $\bK$ is a field of characteristic zero.  There exists constants $\alpha$ and $\beta$ such that, for all $i$, as a sequence of $\Sigma_n$-representations, $H_{i}(\ohur; \bK)_\bN$ is uniformly multiplicity stable for all $n \geq \alpha i+\beta $.

%for $d< \frac{1}{\max(B_{2}, N_{1})+2}(n-\max(r, N_{1}))$.
\end{theorem}
\begin{proof}
    %We have that the multiplicity of $V(\lambda)$ in $H_{d}(\ohur[,\text{[n-1]}])$ is equal to $\langle H_{d}(\ohur[,\text{[n-1]}]; \bK), V(\lambda)\rangle_{\Sigma_{n}}$.
 Fix a partition $\lambda$ and fix $i$.  Let $\bM\in\Mod^{\N}_{\bK}$ be the graded $\bK$-vector space with $$\bM_n \cong H_i(\hur[,n]; V_{n}(\lambda)).$$ Using that duoidal relationship between pointwise tensor product and Day convolution, we obtain a $H_0(\hur)$-module structure on $\bM$. 
    Note that $\dim \bM_n$ is equal to the multiplicity of $V(\lambda)$ in $H_i(\ohur)_\bn$. Thus, our goal is to show $\bM_n$ stabilizes. By \cref{prop: Hilbert for R-mods}, it suffices to show $\bM$ is finitely generated as a $H_0(\hur)$-module. 

Let $A$, $B$ be as in \cref{lem: uniform rep stab for OHur}. The proposition implies that for all $n \geq A i +B$, 
$$H_{i}\bigg(\Big(\Big ((\ohur)_{\bK}\oast V(\lambda)\Big)/\widetilde{U} \Big)_{\hfb } \bigg)_{n} \cong 0.$$ As in the proof of \cref{lem: vanishing for base case}, we have that $$H_{i}\bigg(\Big(\Big ((\ohur)_{\bK}\oast V(\lambda)\Big)/\widetilde{U} \Big)_{\hfb } \bigg)_n \cong H_i(\hur[,n],\hur[,n-N] ; V(\lambda)_n ,V(\lambda)_{n-N}  ).$$ The long exact sequence of a pair now implies that multiplication by $U \in H_0(\hur)_N$ induces a surjection  
$$\bM_{n-N}  \to \bM_n$$ for $n \geq Ai+B$ and an isomorphism for $n \geq A(i+1)+B$. This implies $\bM$ is generated by $\bigoplus_{n \leq Ai+B} \bM_n$ as a $H_0(\hur)$-module. 

By Ellenberg--Venkatesh--Westerland \cite[Proposition 2.5]{MR3488737}, $\hur[,n]$ is homotopy equivalent to a CW complex with finitely many cells. Since  $\ohur[,n]$ is a finite sheeted cover of $\hur[,n]$, it is also homotopy equivalent to a CW complex with finitely many cells. Since $V(\lambda)$ is finite dimensional, we conclude each $\bM_n$ is finite dimensional. Thus, $\bigoplus_{n \leq Ai+B} \bM_n$ is finite dimensional and so $\bM$ is a finitely generated $H_0(\hur)$-module. Thus, there exists $j$ such that $\dim \bM_n = \dim \bM_j$ for all $n \geq j$.

We will verify that these isomorphisms in fact hold for all $n$ in the range where we have established periodicity of $\bM_n$.  Fix $n \geq A(i+1) +B -N$. We have proved that under this assumption,  $\bM_n \cong \bM_{n+kN}$ for all $k \geq 0$. By picking $k$ sufficiently large, we may assume $n+kN \geq j$. Thus, $\bM_n \cong \bM_j$. Hence the multiplicities stabilize once $$n \geq Ai+A+B-N.$$ Take $\alpha=A$ and $\beta = A+B-N$.
\end{proof}

\cref{thm: mult stability for ohur} and the following proposition now imply \cref{OhurFIcIntro}, which states that the dimension of the homology groups $H_i\left(\ohur[, {\bn}]\right)$ grow polynomially in $n$ for each fixed degree $i$.

\begin{proposition}
    \label{PropPoly} Let $\bK$ be a field of characteristic zero.
Let $X$ be an $\fb$-module over $\bK$ that is uniformly multiplicity stable in the sense of Definition \ref{DefnUniformMultStability}, with stable range $k$. Then there is an integer-valued polynomial $p_X$ of degree $\leq k$ such that $$\dim_{\bK} X_n =p_X(n)  \text{ for all }n\geq k.$$
\end{proposition}

\begin{proof} 
By our definition of uniform multiplicity stability, any partition $\lambda=(\lambda_1, \dots, \lambda_{\ell})$ such that $V_{\lambda\langle n\rangle}$ occurs in $X_n$ for any $n$ must satisfy $\lambda + |\lambda_1| \leq k$. 
Thus it suffices to show that, for any partition $\lambda$, the dimension of $V_{\lambda\langle n\rangle}$  agrees, for all $n$ at least $|\lambda|+\lambda_1$, with a polynomial in $n$ of degree at most $|\lambda|$. But these dimensions are given by the following explicit formula, which we obtain by evaluating MacDonald's formula  \cite[I.7.14 Formula (4)]{MacdonaldSymmetric} for the character of $V_{\lambda\langle n\rangle}$  at the identity element of $\Sigma_n$. 

Fix a partition $\lambda$. Then

\begin{equation} \label{DimensionIrrep} \dim_{\bK}\left(V_{\lambda\langle n\rangle}\right)  = \sum_{\substack{ \text{Partitions } \rho, \sigma \\ |\rho| + |\sigma| = |\lambda| }} \frac{(-1)^{\ell(\sigma)} \; \chi^{\lambda}_{(\rho \cup \sigma)} }{z_{\sigma}} \;   \;  \binom{n}{m_1(\rho)} \qquad \text{for $n \geq |\lambda|+\lambda_1$}, 
\end{equation} 
where
\begin{itemize}
    \item $\ell(\sigma)$ is the number of parts of the partition $\sigma$,
    \item $\rho \cup \sigma$ is the partition of the integer $|\rho|+|\sigma|$ whose parts are the union of the parts of $\rho$ and the parts of $\sigma$, 
    \item $\chi^{\lambda}_{(\rho \cup \sigma)}$ is the character of the $\Sigma_{|\lambda|}$-representation $V_{\lambda}$ evaluated on elements of cycle type $\rho \cup \sigma$, 
    \item $z_{\sigma}$ is the integer defined so that $\frac{|\sigma|!}{z_{\sigma}}$ is the number of permutations in $\Sigma_{|\sigma|}$ with cycle type $\sigma$, 
    \item $m_1(\rho)$ is the number of parts of the partition $\rho$ of size 1, 
    \item  $\binom{n}{m_1(\rho)}$ is the binomial coefficient. 
\end{itemize}
The binomial coefficient $\binom{n}{m_1(\rho)}$ is polynomial in $n$ of degree $m_1(\rho)$ and its coefficient in Formula (\ref{DimensionIrrep}) is independent of $n$.  Thus for $n \geq |\lambda|+\lambda_1$, the dimensions $\dim_{\bK}V(\lambda)_n$ agree with a polynomial in $n$ of degree at most $|\lambda|$. (In fact, the polynomial has degree exactly $|\lambda|$, since the term $\dim_{\bK}(V_{\lambda}) \binom{n}{|\lambda|}$   corresponding to the indices $\sigma=\varnothing$ and $\rho=(1, 1, \dots, 1)$ is nonvanishing and is the only term of degree $|\lambda|$.)   The result follows. 
\end{proof}

Observe that \cref{thm: uniform mult stab} and \cref{PropPoly} imply 
\cref{HurPolyGrowth}, polynomial growth for the Betti numbers of ordered Hurwitz space.

\section{A Hurwitz analogue of \texorpdfstring{$\FI$}{FI}} \label{sec: uni rep stab}

We now define a category $\FIc$ which is a generalization of the category $\FI$ to the context of Hurwitz spaces.

\begin{definition}
    The objects of $\FIc$ are finite sets. Let $\Hom_{\FIc}(S , T)$ be the set of pairs $(f,x)$ with $f$ an injection $f \colon S \to T$ and $x \in \pi_0\left(\ohur[, T \setminus \im(f)]\right)$. The composition $\Hom_{\FIc}(S , T)\times \Hom_{\FIc}(T , V)\to \Hom_{\FIc}(S , V)$ is given by the map
    \begin{align*}
        \Hom_{\FIc}(S , T)\times \Hom_{\FIc}(T , V)&\to \Hom_{\FIc}(S , V)\\
        \Big( (f,x), (g,y) \Big)&\mapsto \Big(g \circ f, \left(\left(g|_{T\setminus \im(f)}\right)_*(x) \right) \cdot y    \Big).
    \end{align*}
    %Given $(f,x) \in \Hom_{\FIc}(S , T)$ and $(g,y) \in \Hom_{\FIc}(T , V)$, let the composite be $$(g,y) \circ (f,x) \colonequals (g \circ f, \left(\left(g|_{T\setminus \im(f)}\right)_*(x) \right) \cdot y    ).$$
\end{definition}
Recall that an $\FIc$-module over a ring $\bK$ is a covariant functor from $\FIc$ to the category of left $\bK$-modules. Note that left $H_0(\ohur)$-modules are the same thing as $\FIc$-modules. 
\begin{remark} Observe that given $(f,x) \in \Hom_{\FIc}(\bm, \bn)$, we may view $x$ as the orbit of a tuple in $c^{n-m}$ modulo the action of the $\PBr_{n-m}$.  
Whenever these pure braid group actions are trivial (e.g., when $G$ is abelian), this category $\FIc$ coincides with the category $\FI_c$, where the morphisms $\Hom_{\FI_c}(\bm, \bn)$ are injective maps $f \colon \bm \to \bn$ plus the data of a choice of coloring of the complement of the image of $f$ by elements of $c$ (see Ramos \cite{RamosFId}). 
%More generally, the category $\FIc$ is a quotient of $\FI_c$ by the pure braid group actions.  %%%%% This comment is not correct -- we would need to take a 'noncommutative' version of FI_c. The FIc needs something like an ordering on the complements to make the braid group actions well-defined under composition. 
    
\end{remark}

\begin{definition}
   Let $\bM$ be an $\FIc$-module. We say that $\bM$ is \emph{generated in degrees $\leq d$} if for all finite sets $S$ with $|S| > d$ the map 
 \begin{equation} \label{FinGenFI(c)}
\bigoplus_{\substack{ \text{inclusions of proper} \\ \text{subsets } \iota \colon A \hookrightarrow S \text{ and } \\ (\iota, x)\in \Hom_{\FIc}(A,S)}} \bM_A \xrightarrow{(\iota,x)_*} \bM_S 
 \end{equation} 
   is surjective. 
%%%%
   %We say that $\bM$ is \emph{presented in in degrees $\leq d$} if it is generated in degrees $\leq d$ and, in addition, for $|S|>d$, the kernel of the map of $\FIc$-modules  \begin{align*}  \bigoplus_{\text{proper subsets $A \subseteq S$}} \bK[\Hom_{\FIc}(A,S)] \otimes_{\Hom_{\FIc}(A,A)} \bM_A  & \longrightarrow \bM_S  \\ (f,x) \otimes m & \longmapsto (f,x)_*(m) \end{align*}  is generated in degrees $\leq d$. 
%%
 %  For a finite set $S$, let $$\mathcal{S} = \bigcup_{A \subsetneq S} \pi_0(\ohur)_{S \setminus A}.$$ An element $(A,x) \in \mathcal S$ can naturally be viewed as an element in $\Hom^{\FIc}(A,S)$.  We say $M$ is generated in degrees $\leq k$ if for all sets $S$ with $|S| \geq k$, $$ \bigoplus_{\mathcal{S}} M_A \to M_\bn $$ is surjective. 
 %%
\end{definition}

We remark that an $\FIc$-module $\bM$ is generated in degrees $\leq d$ if equivalently $$H_0\left(Q_{\mathbb{L}}^{H_0\left(\ohur\right)}(\bM) \right)_S \cong 0  \qquad \text{ for $|S| >d$}.$$

Our goal is to prove the following Theorem which is a mild generalization of  \cref{OhurFIcIntro}.

\begin{theorem} \label{OhurFIc}  Suppose that $(G, c)$ has the non-splitting property and $\bK$ is a field of characteristic zero. There exists constants $a,b$ such that $H_i(\ohur)$ is generated as an $\FIc$-module degrees $ \leq ai+b$.
  \end{theorem}

\cref{OhurFIc} follows from \cref{lem: uniform rep stab for OHur} and the following more general result.

\begin{theorem}   Suppose that $(G, c)$ has the non-splitting property and $\bK$ is a field of characteristic zero. Let $\bM$ be an $\FIc$-module over $\bK$ and assume that there exists some $k \geq 0$ such that multiplication by $\widetilde U $ induces a surjection $$H_0(\Sigma_n ; \bM_\bn \otimes_{\bK}  V(\lambda)_n ) \to H_0(\Sigma_{n+N} ; \bM_{\bf n+N} \otimes_{\bK}  V(\lambda)_{n+N} ) $$ for all partitions $\lambda$ provided $n \geq k$. Then $\bM$ is generated as an $\FIc$-module in degrees $\leq k+N$. 
  \end{theorem}

  \begin{proof}

After reindexing, our assumption is that, for the $\FIc$-module $\bM$ and integer $k$, the induced map $$H_0(\Sigma_{n-N}; \bM_{\bf n-N} \otimes_{\bK}  V(\lambda)_{n-N} ) \longrightarrow H_0(\Sigma_{n} ; \bM_\bN \otimes_{\bK}  V(\lambda)_{n} ) $$ is an isomorphism for all $n \geq k+N$. The goal is to show that, under this assumption, the map
 \begin{equation} \label{GoalStableSurjection}
\bigoplus_{\substack{ \text{inclusions of proper} \\ \text{subsets } \iota \colon A \hookrightarrow \bn \text{ and } \\ (\iota, x)\in \Hom_{\FIc}(A,\bn)}} \bM_A \xrightarrow{(\iota,x)_*} \bM_\bn
 \end{equation} 
 is surjective for all $n \geq k+N$. We view the map (\ref{GoalStableSurjection}) as an equivariant map of $\Sigma_n$-representations, where the action of $\Sigma_n$ on the left-hand side arises from its action on the indexing set. Since $\bK$ is a field of characteristic zero, there is an internal direct sum decomposition of $\bK[\Sigma_n]$-modules $\bM_\bn={\bf C}_\bn \oplus {\bf V}_\bn$, where ${\bf V}_\bn$ is the image of (\ref{GoalStableSurjection}) and ${\bf C}_\bn$ is isomorphic to the cokernel of (\ref{GoalStableSurjection}). Significantly, by construction of $\widetilde{U}$, the $\bK[\Sigma_n]$-span of the image of the map $\bM_{\bf n-N} \to \bM_\bn$ induced by multiplication by $\widetilde{U}$ must be contained in ${\bf V}_\bn$.  
 
 Fix $n \geq k+N$. Suppose for the sake of contradiction that ${\bf C}_\bn \not\cong 0$. Then ${\bf C}_\bn$ must contain the irreducible $\bK[\Sigma_n]$-representation $V_{\mu\langle n \rangle}$ for some partition $\mu$. But then the induced map
\begin{align*}
\bM_{\bf n-N} \otimes_{\bK[\Sigma_{\bf n-N}]} V(\mu)_{n-N} & \xrightarrow{\widetilde{U}_* \otimes \varphi_{n-N, n}}  \bM_\bn \otimes_{\bK[\Sigma_{\bf n-N}]} V(\mu)_n \\
& \twoheadrightarrow \bM_\bn \otimes_{\bK[\Sigma_{\bn}]} V(\mu)_n \cong \big({\bf C}_{\bn}  \otimes_{\bK[\Sigma_{\bn}]} V(\mu)_n\big) \oplus \big({\bf V}_{\bn}  \otimes_{\bK[\Sigma_{\bn}]} V(\mu)_n\big)
\end{align*} is not surjective:  its image is contained in the summand $({\bf V}_{\bn}  \otimes_{\bK[\Sigma_{\bn}]} V(\mu)_n)$, whereas the complement $({\bf C}_{\bn}  \otimes_{\bK[\Sigma_{\bn}]} V(\mu)_n)$ is nonzero by choice of $\mu$. This contradicts our hypothesis and completes the proof. 
 \end{proof}

\bibliographystyle{halpha}
\bibliography{bibliography}
\end{document}